\documentclass[11pt]{article}
\usepackage{fullpage}
\usepackage[hyphens]{url}
\usepackage{hyperref}
\hypersetup{
breaklinks=true,
    colorlinks=true, 
    linkcolor=black, 
    citecolor=black, 
    filecolor=black,
    urlcolor=black 
}
\usepackage[utf8]{inputenc} % allow utf-8 input
\usepackage[T1]{fontenc}    % use 8-bit T1 fonts
\usepackage{booktabs}       % professional-quality tables
\usepackage{amsfonts}       % blackboard math symbols
\usepackage{nicefrac}       % compact symbols for 1/2, etc.
\usepackage{microtype}      % microtypography
\usepackage{tcolorbox}
\usepackage{diagbox}
\usepackage{enumerate}
\usepackage[shortlabels]{enumitem}
\usepackage{algorithmic}
\usepackage{algorithm}
\usepackage{subfigure}
\usepackage{tabularx}
\usepackage{verbatim}

\usepackage{graphicx} % more modern
\usepackage{caption}
\usepackage{amsmath}
\usepackage{amsthm}
\usepackage{amssymb}
\usepackage{tablefootnote}
\usepackage{multirow}
\usepackage{enumerate}
\usepackage{color}
\usepackage{xcolor}
\newcommand{\citep}[1]{\cite{#1}}
\newcommand{\citet}[1]{\cite{#1}}

\allowdisplaybreaks[4]

\usepackage{mathrsfs}

\usepackage{bm,todonotes}

\allowdisplaybreaks

\usepackage{amsthm}

\newtheorem{thm}{Theorem}[section]
\newtheorem{lem}{Lemma}[section]
\newtheorem{cor}{Corollary}[section]
\newtheorem{prop}{Proposition}[section]

\newtheorem{hypo}{Working Hypothesis}[section]

\newcounter{subassumption}[asu]

\makeatletter
\renewcommand{\p@subassumption}{\theasu}% Counter prefix.
\makeatother

\newtheoremstyle{remarkstyle}
  {}                    % Space above
  {}                    % Space below
  {\normalfont}         % Body font
  {}                    % Indent amount
  {\itshape}            % Theorem head font
  {.}                   % Punctuation after theorem head
  { }                   % Space after theorem head
  {}                    % Theorem head spec (can be left empty, meaning ‘normal’)
\theoremstyle{remarkstyle}
\newtheorem{rem}{Remark}[section]

\definecolor{wjs}{RGB}{200,0,50}
\newcommand{\wjs}[1]{\textcolor{wjs}{[Weijie: #1]}}
\newcommand{\lx}[1]{\textcolor{magenta}{[Xiang: #1]}}
\definecolor{hyw}{RGB}{153,000,000}

\long\def\new#1{\bgroup\color{black}#1\egroup}

\hypersetup{
colorlinks=true,
filecolor=black,
citecolor=blue,
urlcolor=black,
}

\usepackage{amsmath,amsfonts,bm}

\def\floor#1{\lfloor #1 \rfloor}
\def\1{\bm{1}}

\def\eps{{\varepsilon}}

\def\rd{{\textnormal{d}}}
\def\re{{\textnormal{e}}}

\def\vx{{\bm{x}}}
\def\vy{{\bm{y}}}

\DeclareMathAlphabet{\mathsfit}{\encodingdefault}{\sfdefault}{m}{sl}
\SetMathAlphabet{\mathsfit}{bold}{\encodingdefault}{\sfdefault}{bx}{n}

\def\R{{\RB}}

\renewcommand{\xi}{\zeta}

\def\0{{\bf 0}}
\def\1{{\bf 1}}

\def\MM{{\mathcal M}}
\def\AM{{\mathcal A}}

\def\IM{{\mathcal I}}

\def\PM{{\mathcal P}}

\def\SM{{\mathcal S}}
\def\TM{{\mathcal T}}

\def\UM{{U}}

\def\WM{{\mathcal W}}

\def\RB{{\mathbb R}}

\DeclareMathOperator{\EB}{\mathbb{E}}

\def\PB{{\mathbb P}}

\def\FPM{\PM_{\Delta}}

\def\SPM{\overline{\PM}_{\Delta}}
\def\TPM{\PM_{\Delta}'}
\def\TPMo{\PM_{\Delta_0}'}
\def\Psecond{P_{(2)}}
\def\Ptop{P_{(1)}}

\def\Rlimit{\bR_{\Delta}}

\newcommand{\KL}{D_{\mathrm{KL}}}
\newcommand{\Var}{\mathrm{Var}}

\def\rd{{\mathrm{d}}}
\def\re{{\mathrm{e}}}

\def\Ent{{\mathrm{Ent}}}

\def\Key{{\mathtt{Key}}}

\def\gum{{\mathrm{gumb}}}
\def\Simplex{{\mathrm{Simp}}}

\def\token{{w}}
\def\Voca{{\WM}}
\newcommand\bP{\bm{P}}

\def\SMmax{\SM^{\mathrm{gum}}}
\def\SMinv{\SM^{\mathrm{inv}}}

\def\dif{{\mathrm{dif}}}

\newcommand{\argmax}{{\rm argmax}}
\newcommand{\gumbel}{{\rm gum}}

\def\Ydif{Y^{\mathrm{dif}}}
\def\Yars{Y^{\mathrm{gum}}}

\def\bR{\overline{R}}

\newcommand{\Ext}{{\rm Ext}}

\def\Sars{T^{\mathrm{ars}}}

\def\hid{h_{\mathrm{id}}}
\def\hllr{h_{\mathrm{llr}}}
\def\hars{h_{\mathrm{ars}}}
\def\hind{h_{\mathrm{ind},\delta}}
\def\hindo{h_{\mathrm{ind},\mathrm{e}^{-1}}}

\def\hneg{h_{\mathrm{neg}}}
\def\hlog{h_{\mathrm{log}}}

\def\hopt{{h_{\Delta}^{\star}}}

\def\hoptars{{h_{\mathrm{gum}, \Delta}^{\star}}}
\def\hoptdif{{h_{\mathrm{dif}, \Delta}^{\star}}}

\title{A Statistical Framework of Watermarks for Large Language Models: Pivot, Detection Efficiency and Optimal Rules}

\author{
{Xiang Li\thanks{University of Pennsylvania; Email: \texttt{lx10077@upenn.edu}. } } 
\and
{Feng Ruan\thanks{Northwestern University; Email: \texttt{fengruan@northwestern.edu}. }} 
\and
{Huiyuan Wang\thanks{University of Pennsylvania; Email: \texttt{huiyuanw@upenn.edu}. }  }  
\and
{Qi Long\thanks{University of Pennsylvania; Email: \texttt{qlong@upenn.edu}. }}
\and
{Weijie J.~Su\thanks{University of Pennsylvania; Email: \texttt{suw@wharton.upenn.edu}. }}
}

\date{March 28, 2024}

\begin{document}

\maketitle

\begin{abstract}
Since ChatGPT was introduced in November 2022, embedding (nearly) unnoticeable statistical signals into text generated by large language models (LLMs), also known as watermarking, has been used as a principled approach to provable detection of LLM-generated text from its human-written counterpart. In this paper, we introduce a general and flexible framework for reasoning about the statistical efficiency of watermarks and designing powerful detection rules. Inspired by the hypothesis testing formulation of watermark detection, our framework starts by selecting a pivotal statistic of the text and a secret key---provided by the LLM to the verifier---to control the false positive rate (the error of mistakenly detecting human-written text as LLM-generated). Next, this framework allows one to evaluate the power of watermark detection rules by obtaining a closed-form expression of the asymptotic false negative rate (the error of incorrectly classifying LLM-generated text as human-written). Our framework further reduces the problem of determining the optimal detection rule to solving a minimax optimization program. We apply this framework to two representative watermarks---one of which has been internally implemented at OpenAI---and obtain several findings that can be instrumental in guiding the practice of implementing watermarks. In particular, we derive optimal detection rules for these watermarks under our framework. These theoretically derived detection rules are demonstrated to be competitive and sometimes enjoy a higher power than existing detection approaches through numerical experiments.

\end{abstract}

\newpage

%\section{Preliminaries on (Unbiased) Watermarks}\label{sec:framework}

%Large Language Models (LLMs) are advanced artificial intelligence systems designed to understand, generate, and interact with human language. 

%\wjs{maybe to distinguish $\xi_{\token}$ and $\bm\xi_{t}$?}

% \wjs{
% \begin{itemize}
% \item remove any unused labels
% \item check language of the footnotes
% \item be consistent, type I error/Type I error/Type 1 error/type I error \lx{I unify them to Type I error}
% \item CDF is redefined on page 50 \lx{Solved}
% \end{itemize}
% }

\section{Introduction}\label{sec:intro}

%\wjs{call $\bP_t$ the next-token prediction (NTP) distribution.}
%\wjs{what is $\xi$ called? a better name?} \lx{I used to say it pseudorandom variable}

% \wjs{add the citations to the three bullets on page 1: Large language models (LLMs) have gained significant popularity recently, thanks to their ability to produce text resembling human writing and to address various tasks via a user-friendly chatbot interface \citep{touvron2023llama,openai2023,achiam2023gpt}.
% One critical challenge is to identify the abuse of LLM-generated texts and prevent LLMs from malicious purposes such as ethical and social risks \citep{weidinger2021ethical}, fake news \citep{zellers2019defending}, and academic fraud \citep{stokel2022ai}, with a survey found in \citep{yang2023survey,ghosal2023towards,wu2023survey}.}

%\wjs{refer to as Working Hypothesis instead of Hypothesis, to differentiate from H0 and H1.}\\

%\wjs{do to, use $\FPM$-efficiency instead of $\FPM$-dependent efficiency \lx{Done}. need to declare this earlier.}\\ 
%\wjs{use $\re$ instead of $e$ \lx{I think it's already done.}}\\

Large language models (LLMs) have emerged in recent years as a disruptive technology to generate human-like text and other media \citep{touvron2023llama,openai2023,achiam2023gpt}. While this reality enhances productivity in many sectors, the mismatch between ownership and generation of content could lead to several unwanted outcomes:
\begin{enumerate}

\item[] \textit{Exacerbating misinformation}. The ability of LLMs to generate a large amount of text in parallel can be easily leveraged to exacerbate the spread of misinformation \citep{zellers2019defending,weidinger2021ethical}, which can be enabled by deploying automated bots on social media platforms \citep{starbird2019disinformation}. It may also facilitate fraud and deception by pretending to be humans interacting with their relatives and acquaintances.

\item[] \textit{Facade of AI-assisted education}. LLMs impose challenges to education because students may use powerful LLMs to write essays for themselves \citep{stokel2022ai,milano2023large}. This deprives students of opportunities to practice their own skills and creates inequalities among students depending on the capabilities of the LLMs they access.

\item[] \textit{Data pollution}. The internet will soon consist of more LLM-generated text than human-written text. If LLM-generated text is indiscriminately mixed with human-written text for training, it becomes difficult to create high-quality data for developing next-generation LLMs \cite{radford2023robust,shumailov2023curse,das2024under}.

\end{enumerate}

An initial effort to mitigate these problems has been to leverage specific patterns of LLM-generated text to distinguish it from human-generated text \citep{gptzero2023,zerogpt2023,mitchell2023detectgpt}. However, this approach has become increasingly ineffective as models such as {ChatGPT-4}, {Claude 3}, {Gemini 1.5 Pro}, and many others have reached a level that makes it significantly difficult, if not impossible, to distinguish their generated text from human-written text \cite{weber2023testing}. 

A more viable approach is to watermark text by embedding a signal into the LLM-generated text in a manner that allows the watermark to be \textit{provably} detected \cite{christ2023undetectable,scott2023watermarking}. The property of provable detection is crucial because it allows the verifier to identify LLM-generated text for malicious purposes without relying on assumptions about the text, which may not always hold. Moreover, a reasonable watermarking scheme should be approximately unbiased, meaning it does not significantly distort the meaning or style of the original text.

%%%%%%%%%%%%

The necessity of watermarking LLM-generated text was highlighted in the Biden administration's October 2023 executive order, which incorporated proposals for watermarking LLM-generated text and other AI-generated content \cite{biden2023executive}. As part of the executive order, the U.S. Department of Commerce will help develop standards to watermark LLM-generated content. Accordingly, OpenAI, Google, Meta, and other tech giants have pledged to develop watermarking systems \cite{reuters2023aiwatermark}.

This reality has made it imperative for researchers to develop watermarking methods for LLM-generated text. Within a year since 2023, numerous watermarks have been proposed \cite{kirchenbauer2023reliability,fernandez2023three,kuditipudi2023robust,hu2023unbiased,wu2023dipmark,zhao2024permute,zhao2024provable,liu2024adaptive,giboulot2024watermax,xie2024debiasing}. A common feature of these methods is leveraging the probabilistic nature of LLMs. In essence, these methods incorporate pseudorandomness into the text-generation process of LLMs, and the coupling between the LLM-generated text and the pseudorandomness serves as a signal that can be used for detecting the watermark. The constructed signal becomes pronounced for detection only when the pseudorandom numbers are known, making it practically difficult for one to remove the watermark without access to the pseudorandom numbers \cite{christ2023undetectable}.

To understand how watermarks work for LLMs in more detail, we must first introduce the concept of ``tokenization'' in LLM text processing. Tokenization involves breaking down the text into smaller units called tokens, which are informally known as sub-words. These tokens can be words, parts of words, or even punctuation marks. For example, the sentence ``Hello, world!'' when tokenized, might be split into four tokens: [``Hello'', ``,'', `` world'', ``!'']\footnote{Refer to the website \url{https://platform.openai.com/tokenizer} for user-customized examples of tokenization.}. An LLM generates each token sequentially by sampling from a probability distribution conditioned on prior tokens, among other things. Typically, the size of the token vocabulary is of the order of $10^4$ and varies with language models. Letting $\Voca$ denote the vocabulary of all tokens, for example, $|\Voca|=50,272$ for the OPT-1.3B model \citep{zhang2022opt}, $|\Voca|=32,000$ for the LLaMA series models \citep{touvron2023llama}, and $|\Voca|=50,257$ for GPT-2 and GPT-3.5 series models \citep{radford2019language,brown2020language}.

%%%%%%%%%%%%%%%%%

After generating text in the form of a token sequence, $\token_1 \token_2 \cdots \token_{t-1}$, the (unwatermarked) LLM generates the next token $\token_t$ according to a multinomial distribution $\bP_t := (P_{t,w})_{w \in \Voca}$ on the vocabulary $\Voca$, satisfying $\sum_{w \in \Voca} P_{t,w} = 1
$.\footnote{As a convention, throughout this paper three subscripts appear: $\text{letter}_{t}, \text{letter}_{\token}$, and $\text{letter}_{t, \token}$, which correspond to step $t$, token $\token$, and token $\token$ at step $t$, respectively.}
% \[
% \sum_{w \in \Voca} P_{t,w} = 1.
% \]
We call $\bP_t$ the \textit{next-token prediction} (NTP) distribution at step $t$, and it depends on all prior generated tokens, the user-supplied prompt, as well as system prompts that are hidden from users \citep{vaswani2017attention,radford2019language,brown2020language}.

In contrast, a watermarked LLM generates the next token that is jointly determined by a pseudorandom variable and the NTP distribution. Let $\xi_t$ denote the pseudorandom variable at step $t$, which is available only to the verifier. Formally, the watermarked LLM samples a token according to the rule \citep{hu2023unbiased,wu2023dipmark,li2024robust}
\[
\token_t := \SM(\bP_t, \xi_t),
\]
where $\SM$ is a (deterministic) decoder function. To achieve approximate unbiasedness, we require that the probability distribution of $\token_t \equiv \SM(\bP_t, \xi_t)$ over the randomness\footnote{Strictly speaking, $\xi_t$ has no randomness. However, modern cryptographic theories ensure that $\xi_t$ behaves very much like a random variable. See more details in Section \ref{sec:overview}.} embodied in $\xi_t$ is close to $\bP_t$, conditional on all previous tokens, $\token_1\cdots\token_{t-1}$ \cite{hu2023unbiased,kuditipudi2023robust}. In this regard, an unbiased watermark roughly corresponds to a sampling method from multinomial distributions.

With the generated text now indistinguishable from that of the unwatermarked LLM, it seems at first glance hopeless to obtain provable detectability for watermarked text \cite{sadasivan2023can,zhang2023watermarks}. Interestingly, detection can be made possible by carefully designing the decoder $\SM$ to impose a coupling relationship between the token and pseudorandom variable, even without knowing the NTP predictions \citep{christ2023undetectable}.\footnote{We cannot use the probabilities $\bP_t$'s as the verifier in general does not have access to the LLM that generates the \new{tokens}. Moreover, the prompt for generating the text is unavailable to the verifier, hence the verifier cannot obtain $\bP_t$ even having access to the LLM.} In the following, we elaborate on this point by considering perhaps the simplest example of a watermark that achieves both unbiasedness and provable detectability.

\paragraph*{A baby watermark} Envision an LLM that involves only two tokens, 0 and 1---that is, $\Voca = \{0, 1\}$. Let $\bP_t = (P_{t,0}, P_{t,1})$ denote the NTP distribution at step $t$, and let $\xi_t$ be i.i.d.~copies of the standard uniform random variable $U(0, 1)$. Set the decoder as follows:
\begin{equation}\label{eq:baby_wmk}
\SM(\bP, \xi) = 
\begin{cases}
0 & \text{ if } \xi \le P_0 \\
1 & \text{ otherwise}.
\end{cases}
\end{equation}
This watermark is unbiased. Intuitively, if $\xi_t$ is large, then $w_t$ is more likely to be 1 instead of 0, and vice versa. This intuition suggests using the following statistic for detecting the watermark:
\begin{equation}\label{eq:simple-cov}
\sum_{t=1}^n (2w_t-1)(2\xi_t-1),
\end{equation}
which measures the correlation between the tokens and pseudorandom variables. When the watermark is present, this statistic tends to be larger in distribution than when the watermark is absent. One can conclude that a watermark is detected if this statistic is above a certain threshold. Despite being intuitive, however, this statistic is ad-hoc, and one cannot rule out the possibility of a better detection rule.

%\lx{Above, you introduce “baby watermarks.” Here, you move on to practical watermarking schemes. Should you make some connection between the two? For example, “Based on the baby watermark framework, a number of watermark schemes have been developed for use in practice.” Also, at this point in the paper, the reader doesn’t know why you need to better understand practical watermark schemes. Is it to determine if there are possibly better watermarking schemes than the ones currently in use? Is it to show the need for a new scheme (one that you will provide in the paper)?}

To better understand practical watermarking schemes, we turn to real-world LLMs where the size of the token vocabulary is very large. As a recap, at the core of an unbiased watermark is a sampling method for multinomial distributions. Perhaps the two most common sampling methods are the Gumbel-max trick \citep{maddison2014sampling,jang2016categorical} and the inverse transform \citep{devroye2006nonuniform}. Interestingly, these two sampling methods correspond precisely to two important and representative watermarks \citep{kuditipudi2023robust,scott2023watermarking}.

\paragraph*{Gumbel-max watermark \citep{scott2023watermarking}}
Let $\xi = (U_\token)_{\token \in \Voca}$ consist of $|\Voca|$ i.i.d.~copies of $U(0, 1)$. A version of the Gumbel-max trick \citep{gumbel1948statistical} states that $\arg\max\limits_{\token \in \Voca} \frac{\log U_w}{P_w}$
% \[
% \arg\max_{\token \in \Voca} \frac{\log U_w}{P_w}
% \]
follows the NTP distribution $\bP \equiv (P_{\token})_{\token \in \Voca}$.\footnote{\new{It is noteworthy that the Gumbel-max trick has a broad range of applications, including partition function estimation \citep{hazan2012partition}, ranking under the random utility model \citep{soufiani2012random}, and computational sampling in machine learning \citep{papandreou2011perturb,maddison2014sampling,jang2016categorical}.}} 
Recognizing this fact, Scott Aaronson proposed the following decoder \citep{scott2023watermarking}:
\begin{equation}\label{eq:it}
\SMmax(\xi, \bP) := \arg\max_{\token \in \Voca} \frac{\log U_w}{P_w},
\end{equation}
which is, by definition, unbiased \citep{fernandez2023three,piet2023mark,zhao2024permute}. This watermark has recently been implemented internally at OpenAI \citep{scott2023watermarking}. Aaronson suggested declaring the presence of the watermark if the following statistic is above a certain threshold: $\Sars_n = -\sum_{t=1}^n \log(1- U_{t, \token_t})$. The intuition is that, when the watermark is employed, \eqref{eq:it} implies that a token $w_t$ is more likely to be selected when its associated pseudorandom number $\xi_{t, \token_t}$ is large. In contrast, when the text is written by a human, $\xi_{t, \token_t}$ would not be larger than other entries in the distribution at step $t$. Despite being intuitive, it is worth mentioning that there are countless detection rules capable of capturing this distribution shift. In particular, it is not clear whether this detection rule is optimal in any sense.

\paragraph*{Inverse Transform Watermark \citep{kuditipudi2023robust}} It is well-known that any univariate distribution can be sampled by applying the inverse cumulative distribution function (CDF) to $U(0,1)$. Given an NTP distribution $\bP$ and $\pi$ that maps all tokens in $\Voca$ to a permutation of $1, 2, \ldots, |\Voca|$, consider the multinomial distribution with probability mass $P_{\pi^{-1}(i)}$ at $i$ for $1 \le i \le |\Voca|$. The CDF of this distribution takes the form $F(x; \pi) = \sum_{\token' \in \Voca} P_{\token'} \cdot {\mathbf{1}}_{\{\pi(\token')\le x\}}.$
% \[
% F(x; \pi) = \sum_{\token' \in \Voca} P_{\token'} \cdot {\mathbf{1}}_{\{\pi(\token')\le x\}}.
% \]
Taking as input $U \sim U(0,1)$, the generalized inverse of this CDF is defined as $F^{-1}(U; \pi) = \min\{ i:\sum_{\token' \in \Voca} P_{\token'} \cdot
{\mathbf{1}}_{\{\pi({\token'})\le i\}}\ge U\},$
% \[
% F^{-1}(U; \pi) = \min\bigg\{ i:\sum_{\token' \in \Voca} P_{\token'} \cdot
% {\mathbf{1}}_{\{\pi({\token'})\le i\}}\ge U\bigg\},
% \]
which by construction follows the multinomial distribution $\bP$ after applying the permutation $\pi$. Making use of this fact, \citep{kuditipudi2023robust} proposed the inverse transform watermark with the following decoder:
\begin{equation*}
\SMinv(\bP, \xi) :=\pi^{-1}(F^{-1}(U; \pi)),
\end{equation*}
where the pseudorandom variable $\xi := (\pi, U)$ and the permutation $\pi$ is uniformly at random. By definition, this watermark is unbiased. To detect the watermark, \citep{kuditipudi2023robust} proposed detection rules that, roughly speaking, examine the absolute difference between $U_t$ and the rank of the generated token, $\pi_{t}(\token_t)$, and sum the differences across the token sequence (see details in Section~\ref{sec:inverse}). When the text is watermarked, the difference turns out to be small because the involved two random variables are highly dependent. Similar to the Gumbel-max watermark, the detection of this watermark could potentially benefit from a more principled derivation to enhance its power.

%\lx{ISo, is the goal of this paper to test the detection rules proposed by Kuditipudi, Thickstun, Hashimoto, and Liang ?}

\subsection{This Paper}
From a statistical viewpoint, a watermark scheme's performance at the detection phase is evaluated by two competing criteria. The first is how likely it would mistakenly detect human-written text as LLM-generated text, and the second is the probability that it would mistakenly classify LLM-generated text as its human-written counterpart. This perspective relates watermarks to hypothesis testing, which formally calls the two aforementioned errors Type I error and Type II error, respectively. In this regard, the effectiveness of the Gumbel-max watermark and inverse transform watermark, even including the baby watermark, is not clear yet, though \citep{fernandez2023three,piet2023mark} compared their efficiency through empirical experiments. For example, even though they come with detection rules that seem intuitive, it is not clear if they are statistically optimal in the sense of having the optimal trade-off between Type I and Type II errors. If they are not optimal, it is of interest to find a better detection rule.

In general, we wish to have a general and flexible framework that can guide the development of watermarks through optimizing its detection phase and assessing watermarks in a principled manner. Specifically, the challenge lies in how we can provably control the Type I error for any given watermark, considering that the NTP distributions, which are not accessible to the verifier, vary from token to token. Once Type I error control is achieved, the next question is how to evaluate the Type II error, preferably through a closed-form expression, which also hinges on the unknown NTP distributions. Having known both the Type I and Type II errors, the final step involves comparing watermark detection rules to ultimately identify the optimal detection rule based on our knowledge of the LLM.

In this paper, we address these questions and challenges in a unified way by making the following contributions.

\begin{enumerate}
\item[]

\textbf{A Statistical framework for watermarks.} A major contribution of this paper is a general statistical framework for developing statistically sound watermark detection rules through precisely evaluating the Type I and Type II errors from the hypothesis testing viewpoint. Under this framework, Type I error is controlled by leveraging a pivotal statistic that is distributionally independent of the NTP distributions under the null. This framework is accompanied by a technique for evaluating the asymptotic Type II error using large deviation theory, and moreover, relies on the notion of class-dependent efficiency to tackle the challenge of unknown and varying NTP distributions. Finally, this framework formulates the problem of finding the most powerful detection rule as a minimax optimization program.

This framework is formally developed in Section~\ref{sec:overview}.

\item[]

\textbf{Application to the Gumbel-max watermark.} We apply our framework to Aaronson's Gumbel-max watermark in Section~\ref{sec:gumbel}. Our main finding is that Aaronson's detection rule is suboptimal in the sense that its class-dependent efficiency is relatively low. Moreover, by maximizing the class-dependent efficiency, we obtain the optimal detection rule, which admits a simple analytical expression. This optimal detection rule is shown to outperform existing rules in numerical experiments. Underlying these results is a technique that can reduce the optimality problem to a convex geometry problem for the Gumbel-max watermark, which is a contribution of independent interest to future research on watermarks.

\item[]

\textbf{Application to the inverse transform watermark.} Next, we apply our framework to the inverse transform watermark in Section~\ref{sec:inverse}. Our main finding is twofold. First, we overcome a significant challenge in applying our framework to analyze this watermark by deriving an asymptotic distribution when the text is watermarked. Second, we obtain the optimal detection rule for the inverse transform watermark in a closed-form expression by maximizing its class-dependent efficiency. Our numerical experiments corroborate its efficiency.

\end{enumerate}

\subsection{Related Work}

The most closely related work to ours is the Gumbel-max watermark \cite{scott2023watermarking} and the inverse transform watermark \cite{kuditipudi2023robust}. In \cite{zhao2024permute,fernandez2023three}, the authors introduced unbiased watermarks that are robust to probability perturbations or multi-bit processing. Other unbiased watermarks in this fast-growing line of research include \citep{christ2023undetectable,wu2023dipmark,hu2023unbiased}. In addition, a popular example is the red-green list watermark, which splits the vocabulary into red-green lists based on hash values of previous n-grams and slightly increases the probability of green tokens embedding the watermark \citep{kirchenbauer2023watermark,kirchenbauer2023reliability,zhao2024provable,liu2024adaptive,cai2024towards}. In the detection phase, a high frequency of green tokens suggests the text is LLM-generated. This type of watermark is biased since the NTP distributions have been altered, thereby leading to performance degradation of the LLM.

In contrast, there is much less work on the theoretical front. \new{
A notable exception is the work of \cite{huang2023towards}, which approached watermark detection from the perspective of composite independence testing. Their framework requires model providers to offer random rejection regions. In contrast, our approach employs a pivotal statistic to detect distributional shifts between human-written and LLM-generated text, allowing verifiers to choose any rejection region. While both approaches aim to minimize a form of worst-case Type II error, the optimal detection rule proposed by \citep{huang2023towards} is hindered by a non-vanishing Type II error due to the consideration of many worst-case scenarios, and it is computationally inefficient because of the exponentially large rejection regions required.\footnote{See Theorem 3.10 of \citep{huang2023towards} and the subsequent discussion.} Conversely, our optimal rule generally achieves a vanishing Type II error by addressing fewer worst-case scenarios, benefiting from the well-regularized null behavior of pivotal statistics, and is computationally efficient due to its sum-based structure.}

Research on watermarking text has been conducted well before the advent of modern language models. This body of research focuses on watermarking text by modifying it to introduce specific patterns that are unlikely to be noticeable to readers. This includes synonym substitution \citep{topkara2006hiding}, syntactic restructuring \citep{atallah2001natural}, and linguistic steganography \citep{cox2007digital}. A common weakness of these approaches lies in their biasedness and vulnerability to attacks that aim to remove watermarks \citep{cayre2005watermarking,zhou2009security}.

In a different direction, many methods have been proposed to detect text generated by LLMs---often not watermarked---from human-written counterparts. A common feature of these methods is to examine the complete context, linguistic patterns, and other potentially revealing markers in the given text to assess whether it is likely LLM-generated. The simplest method is to build a classifier using synthetic and human text data, which is adopted by some commercial detection platforms \citep{gptzero2023,zerogpt2023}. Another category is training-free and leverages the inherent stylistic differences between human and machine writing without specific training data, using techniques such as log probability curvature \citep{mitchell2023detectgpt,bao2023fast}, divergent $n$-gram analysis \citep{yang2023dna}, and intrinsic dimension estimation \citep{tulchinskii2024intrinsic}. However, \citep{weber2023testing} find that most post-hoc detection methods are neither accurate nor reliable and suffer from a significant bias towards classifying the output as human-written rather than detecting LLM-generated text. Furthermore, these methods have proven fragile to adversarial attacks and biased against non-native English writers \citep{krishna2024paraphrasing,sadasivan2023can,liang2023gpt}.

%In a different direction, many methods have been proposed to detect text generated by LLMs---often not watermarked---from human-written counterparts. A common feature of these methods is to examine the complete context, linguistic patterns, and other potentially revealing markers in the given text to assess whether it is likely LLM-generated. The simplest method is to build a classifier using synthetic and human text data, which is adopted by some commercial detection platforms \citep{gptzero2023,zerogpt2023}. Another category is training-free and leverages the inherent stylistic differences between human and machine writing without specific training data, using techniques such as log probability curvature \citep{mitchell2023detectgpt,bao2023fast}, divergent $n$-gram analysis \citep{yang2023dna}, and intrinsic dimension estimation \citep{tulchinskii2024intrinsic}. However, \citep{weber2023testing} find that most post-hoc detection methods are neither accurate nor reliable and suffer from a significant bias towards classifying the output as human-written rather than detecting LLM-generated text. Furthermore, they have proven fragile to adversarial attacks and biased against non-native English writers~\citep{krishna2024paraphrasing,sadasivan2023can,liang2023gpt}.

\section{A Statistical Framework for Watermark Detection}\label{sec:overview}

In this section, we formally introduce our framework that enables statistical analysis of watermarks. Our focus is on the development of effective techniques for assessing the statistical efficiency of the watermarks accompanying this framework. Henceforth, in this paper, we write $\token_{1:n} := \token_1 \cdots \token_n$ and $\xi_{1:n} := \xi_1 \cdots \xi_n$ for the text and associated pseudorandom variables, respectively.

\subsection{Working Hypotheses}
\label{sec:working}

The problem of determining whether a watermark is present or not in the text can be formulated as a hypothesis testing problem~\citep{huang2023towards,chakraborty2023possibilities,kirchenbauer2023watermark,wu2023dipmark,giboulot2024watermax}: 
\begin{equation}\label{eq:original_test}
H_0: \token_{1:n}~\text{is written by a human}~~~~~~~H_1: \token_{1:n}~\text{is written by a watermarked LLM}.
\end{equation}

In addition to the text, this hypothesis testing problem also uses the pseudorandom variables as data. A unified way to represent $\xi_t$ based on existing constructions \citep{wu2023dipmark,kirchenbauer2023watermark,kirchenbauer2023reliability,piet2023mark,fernandez2023three} is to take the form 
\begin{equation}\label{eq:am_gen}
\xi_t = \AM(\token_{1:(t-1)}, \Key),
\end{equation}
where $\Key$ denotes a secret key that will be passed to the verifier. The (deterministic) hash function $\AM$ maps any token sequence and a key to a pseudorandom number that will be used to generate the next token. More precisely, the watermarked LLM generates the next token according to the rule
\begin{equation}\label{eq:s_deco}
\token_t := \SM(\bP_t, \xi_t),
\end{equation}
for some (deterministic) decoder $\SM$. Recall that $\bP$ denotes the probability distribution of the next token generated by the (unwatermarked) LLM.

With these notations in place, a watermark can be formally presented by the tuple $(\AM, \SM, \Key)$. To lay a solid footing of watermarks on statistical hypothesis testing, we need two working hypotheses.

\begin{hypo}[Soundness of pseudorandomness]\label{hypo:cryptography}
In the watermarked LLM, the pseudorandom variables $\xi_{1:n}$ constructed above are i.i.d.~copies of a random variable. Furthermore, $\xi_t$ is (statistically) independent of $\token_{1:(t-1)}$.
\end{hypo}

Working Hypothesis \ref{hypo:cryptography} is grounded purely in cryptographic considerations. In cryptography, there are well-established approaches to efficiently constructing and computing the pseudorandom number as a function of text and the secret key~\citep{barak2021book,paar2009understanding,stinson2005cryptography,schneier1996applied,katz2011introduction}. The pseudorandom number is very sensitive to the key, making it computationally \textit{indistinguishable} from its truly random counterpart without knowledge of the key.\footnote{\new{The phrase ``computationally indistinguishable'' means that no polynomial-time algorithm can distinguish the pseudorandom number from its truly random counterpart without knowledge of the key.}} Specifically, although $\xi_t$ is completely determined by the prior text $\token_{1:(t-1)}$ (and the secret key), a run of the hash function $\AM$ could effectively introduce fresh randomness, thereby making $\xi_t$ statistically independent of $\token_{1:(t-1)}$. Note that, during detection, the communication cost is merely to pass the secret key to the verifier as the prior tokens are publicly available. 

Hereafter, we regard $\xi$ as a random variable, which allows us to formally define the \textit{unbiasedness} of a watermark.
We say a watermark is unbiased if, for any multinomial distribution $\bP$, $\SM(\bP, \xi)$ follows $\bP$. That is, for any NTP distribution $\bP$ and token $\token \in \Voca$,
\[
\PB\left(\SM(\bP, \xi)= \token\right) = P_{w},
\]
where the expectation is taken over the randomness embodied in $\xi$. Together with the joint independence of $\xi_t$'s across the sequence of tokens, unbiasedness holds at every step conditional on prior tokens.

%%%%%%%%%%%%%%%%%%%%%%%%%%%%%%%

Our next working hypothesis concerns the joint distribution of $\token_{1:n}$ and $\xi_{1:n}$ when the text is written by a human.

\begin{hypo}[Intrinsic nature of human randomness]\label{hypo:human}
Let $\token_{1:n}$ be a sequence of tokens generated by a human who has no knowledge of the secret key. Then, the human-generated token $\token_t$ and $\xi_t$ are (statistically) independent conditional on $(\token_{1:(t-1)},\xi_{1:(t-1)})$, for all $1 \le t \le n$.
\end{hypo}

\begin{rem}
To clear up any confusion, we remark that the verifier uses the human-generated tokens $\token_1, \ldots, \token_n$ to compute the random variables $\xi_{1:n}$ in \eqref{eq:am_gen}. In particular, Working Hypothesis \ref{hypo:cryptography} remains valid for human-generated text, due to the construction of the hash function $\AM$. 
\end{rem}

In particular, this working hypothesis shows that, for human-written text, the token is not generated according to \eqref{eq:s_deco}. The rationale of this working hypothesis is that how a human writes text is intrinsically random and cannot be captured by pseudorandomness. Therefore, the human-generated token $\token_t$ has nothing to do with $\xi_t$. 
Moreover, one can also argue for this working hypothesis by recognizing that it is practically impossible for a human to generate text such that $\xi_t$ and $\token_t$ are dependent because the secret key is not available. This is even the case for a different LLM without having the secret key because it is computationally infeasible to replicate the pseudorandomness without the key.

The two working hypotheses have been adopted by the literature, albeit not as explicit as our treatment. For example, \citep{wu2023dipmark,zhao2024permute} assumed conditions on the pseudorandom hash function so that the first working hypothesis is valid.\footnote{In particular, they assumed that $\AM(\token_{1:n}, \Key)$ is i.i.d.~for any $\token_{1:n}$.} Other works directly impose distributional assumptions on some summary statistics to effectively satisfy the two working hypotheses~\citep{kirchenbauer2023watermark,kirchenbauer2023reliability,zhao2024provable,fernandez2023three}. As a departure from these works, \citep{kuditipudi2023robust} considered a hash function of the form $\xi_t = \AM(t, \Key)$. Consequently, the independence between the pseudorandom number and prior tokens follows by construction. 
However, it is worthwhile mentioning that this form of the hash function hampers computational efficiency in detection due to the necessity of iterating through the entire text sequence multiple times.

%%%%%%%%%%%%%%%%%%%%%%%

\subsection{Pivotal Statistics}

Understanding the hypothesis testing problem \eqref{eq:original_test} requires analyzing the differences between the joint distribution of $(\token_{1:n}, \xi_{1:n})$ under Working Hypotheses \ref{hypo:cryptography} and \ref{hypo:human}. This can be seen by first decomposing the joint probability density of $(\token_{1:n}, \xi_{1:n})$
into a product of conditional probabilities:\footnote{As an abuse of notation, $\PB(\cdot)$ is considered the probability density function when applied to a continuous variable, and probability mass function when applied to a discrete variable.}
\begin{equation}
\label{eqn:expansion-of-joint-distribution-into-conditionals}
\PB(\token_{1:n}, \xi_{1:n}) = \prod_{t=1}^n \PB(\token_t, \xi_t \mid \token_{1:(t-1)}, \xi_{1:(t-1)}).
\end{equation}
\begin{itemize}
\item Under $H_0$, Working Hypotheses \ref{hypo:cryptography} and \ref{hypo:human} show that $\token_t$ and $\xi_t$ are independent given $(\token_{1:(t-1)}, \xi_{1:(t-1)})$. Hence,\footnote{Here, we conceptualize a human as an LLM, using a multinomial distribution to model the selection of the next token based on prior tokens.}  
\[
\PB(\token_t, \xi_t \mid \token_{1:(t-1)}, \xi_{1:(t-1)}) = \PB(\token_t\mid \token_{1:(t-1)}, \xi_{1:(t-1)}) \PB(\xi_t \mid \token_{1:(t-1)}, \xi_{1:(t-1)}) = P_{t, \token_t} \cdot \PB_{\xi}(\xi_t).
\]
For example, the display above is equal to $P_{t, \token_t}$ for the baby watermark defined in \eqref{eq:baby_wmk}.

\item Under $H_1$, $\token_t$ and $\xi_t$ are dependent given $(\token_{1:(t-1)}, \xi_{1:(t-1)})$. By Working Hypothesis \ref{hypo:cryptography} and the decoding construction \eqref{eq:s_deco}, we have 
\begin{equation*}
\PB(\token_t, \xi_t \mid \token_{1:(t-1)}, \xi_{1:(t-1)}) = 
\begin{cases}
\PB_{\xi}(\xi_t) & \text{ if }  \SM(\bP_t, \xi_t) = \token_t\\
0 & \text{ if }  \SM(\bP_t, \xi_t) \ne \token_t
\end{cases}
\end{equation*}
because the decoder $\SM$ is deterministic. For example, the baby watermark satisfies $\PB(\token_t, \xi_t \mid \token_{1:(t-1)}, \xi_{1:(t-1)}) = 1$ if $\xi_t > P_{t,0}, \token_t = 1$, or $\xi_t \le P_{t,0}, \token_t = 0$. It is 0 otherwise.
\end{itemize}

By the Neyman--Pearson lemma, the most powerful test is based on the likelihood ratio:
\begin{equation}\label{eq:npratio}
\frac{\PB_{H_0}(\token_{1:n}, \xi_{1:n})}{\PB_{H_1}(\token_{1:n}, \xi_{1:n})} = \frac{\prod_{t=1}^n P_{t, \token_t} \cdot \PB_{\xi}(\xi_t)}{\prod_{t=1}^n \mathbf{1}_{\SM(\bP_t, \xi_t) = \token_t} \cdot \PB_{\xi}(\xi_t)} = 
\begin{cases}
P_{1, \token_1} \cdots P_{n, \token_n} & \text{ if } \SM(\bP_t, \xi_t) = \token_t \text{ for all } t\\
\infty & \text{ otherwise}.
\end{cases}
\end{equation}
Although the likelihood ratio appears simple, taking only two values, we encounter a significant challenge in distinguishing between the two cases mentioned above. This challenge arises because the NTP distribution $\bP_t$ is unknown and, worse, can vary with $t$. This nuisance parameter remains unknown even if the verifier has complete access to the LLM, as the prompt used for generating the text is usually not available to the verifier.

%\footnote{If the text were generated by the LLM, under $H_1$, it must be the first case in \eqref{eq:npratio}. However, it is unknown whether it is generated by a human or an LLM in practice.}

To address this challenge, we seek a pivotal statistic $Y_t = Y(\token_t, \xi_t)$ such that its distribution is the same for any NTP distribution $\bP_t$ under the null. Such a pivot allows us to construct test statistics for watermark detection with known distributions and consequently obtain detection rules with provable Type I error control, though at the price of information loss compared to using the full data, $\token_t$ and $\xi_t$. Formally, the original testing problem \eqref{eq:original_test} is reduced to the following: 
\begin{equation}\label{eq:H-surrogate}
H_0: Y_t \sim \mu_0~~\text{i.i.d.~for } 1 \le t \le n \qquad H_1: Y_t \sim \mu_{1,\bP_t}  \text{ for } 1 \le t \le n,
\end{equation}
where $\mu_0$ denotes the (known) distribution of $Y_t$ when the text is human-written, and $\mu_{1,\bP_t}$ denotes the (unknown) distribution of $Y_t$ conditional on $(\token_{1:(t-1)}, \xi_{1:(t-1)})$ when the text is generated by the watermarked LLM. As is clear, $\mu_{1,\bP_t}$ is determined by the NTP distribution $\bP_t$.

As a caveat, the choice of $Y(\token_t, \xi_t) \equiv \xi_t$ satisfies pivotality but is useless since the alternative distribution is the same as the null distribution. A useful choice of $Y$, while being pivotal, should allow the selected token to ``pull'' the alternative distribution toward the same direction for any NTP distribution $\bP_t$.
For example, $Y(\token_t, \xi_t) = (2\token_t-1)(2\xi_t - 1)$ is a good choice for the baby watermark since the dependence between $\token_t$ and $\xi_t$ (see \eqref{eq:baby_wmk} in Section~\ref{sec:intro}) tends to make $Y(\token_t, \xi_t)$ larger. In general, it is a case-by-case approach to find a reasonable pivot for a given watermark (see Sections~\ref{sec:gumbel} and \ref{sec:inverse}).

To test \eqref{eq:H-surrogate}, it is natural to use the sum of $Y_t$'s across the token sequence as a test statistic. To enhance the flexibility, we consider a score function $h$ that applies to the pivot $Y$. This leads to the following rejection rule for the hypothesis testing problem:
\begin{equation}\label{eq:Th}
T_h(Y_{1:n}) := 
\left\{ \begin{array}{ll}
1 &~\text{if } \sum_{t=1}^n   h(Y_t) \ge \gamma_{n, \alpha }\\
0 &~\text{if } \sum_{t=1}^n   h(Y_t) < \gamma_{n, \alpha }.\\
\end{array} \right. 
\end{equation}
That is, we reject that the text is written by a human if $T_h$ is above $\gamma_{n, \alpha }$. The threshold $\gamma_{n, \alpha }$ is chosen to ensure significance level at $\alpha$: $\PB_{H_0}(T_h(Y_{1:n}) = 1) = \alpha$. \new{For certain score functions discussed in this work, the distribution of the sum of $Y_t$'s under the null hypothesis admits a closed-form expression, allowing for analytical evaluation of $\gamma_{n, \alpha}$. } In general, as the null distribution of $h \circ Y$ is known, an estimator of $\gamma_{n,\alpha}$ when the text is sufficiently long is
\begin{equation*}
\hat{\gamma}_{n,\alpha} = n \cdot \EB_0 h (Y) +  \Phi^{-1}(1-\alpha) \cdot \sqrt{n \cdot \Var_0(h(Y))},
\end{equation*}
where $\EB_0$ and $\Var_0$ indicate that $\mu_0$ is used to take the expectation and variance. As can be seen, the underlying NTP distributions $\bP_t$'s are not involved in this detection rule.

\subsection{Class-Dependent Efficiency}
\label{sec:efficiency-inhomo}

Once the Type I error is controlled, we seek to evaluate the Type II error and use it as a measure to ascertain which choice of score function is more desired than others. If the NTP distributions $\bP_t$ were known and remained the same with respect to varying $t$, the optimal score function would simply be given by the log-likelihood ratio, according to the Neyman--Pearson lemma. However, this is not the case. We manage this challenge by assuming that the NTP distributions belong to a distribution class, denoted as $\PM$. The flexibility of using a distribution class lies in that one can choose a small class when much is known about the distributional properties of the LLM, and choose a large class if little is known.

Given a distribution class $\PM$, we can evaluate the Type II error of the test statistic $T_h(Y_{1:n})$ over the least-favorable NTP distributions in $\PM$. This gives rise to a notion of \textit{class-dependent efficiency}. For any score function $h$ and NTP distribution $\bP$, we define the following moment-generating function (MGF):
\begin{equation}\label{eq:class-dependent-MGF}
\phi_{\bP,h}(\theta) := \EB_{1, \bP} \mathrm{e}^{-\theta h(Y)},
\end{equation}
where $\EB_{1, \bP}$ indicates that the expectation is taken over the randomness embodied in $Y \sim \mu_{1, \bP}$ in \eqref{eq:H-surrogate}. Assuming that the NTP distributions are all in $\PM$, the following result delineates the Type II error in the large sample limit. We defer its proof to Appendix \ref{proof:main}, which relies on techniques in large deviation theory \citep{dembo2009large,van2000asymptotic}.

\begin{thm}\label{thm:inhomo-type-II}
Assume $\bP_t \in \PM$ for all $t$. For any $h$ satisfying $\EB_0 |h| < \infty$, the Type II error of the detection rule $T_h$ defined in \eqref{eq:Th} obeys
\begin{equation}\label{eq:inequality_Eff}
\limsup_{n \to \infty} \PB_{H_1}(T_h(Y_{1:n}) = 0)^{1/n} \le  \mathrm{e}^{-R_{\PM}(h)},
\end{equation}
where $R_{\PM}(h)$ is given by 
\begin{equation}\label{eq:efficiency-exponent}
R_{\PM}(h) = -\inf_{\theta\ge 0} \sup_{\bP \in \PM}
\left(\theta  \EB_0 h(Y) +  \log  \phi_{\bP, h}(\theta) \right) = -\inf_{\theta\ge 0} \left(\theta  \EB_0 h(Y) + \sup_{\bP \in \PM} \log  \phi_{\bP, h}(\theta) \right).
\end{equation}
\end{thm}

\begin{rem}
$R_{\PM}(h)$ is the $\PM$-dependent efficiency rate of $h$, or simply the $\PM$-efficiency rate. 
\end{rem}

\begin{rem}\label{rem:tight}
Under a certain regularity condition, \eqref{eq:inequality_Eff} is tight in the sense that there exists $\bP^\star$ in $\PM$ such that, if $\bP_t = \bP^\star$ for all $t$, then $\PB_{H_1}(T_h(Y_{1:n}) = 0) \ge  \mathrm{e}^{-(R_{\PM}(h) + \varepsilon)n}$
% \[
% \PB_{H_1}(T_h(Y_{1:n}) = 0) \ge  \mathrm{e}^{-(R_{\PM}(h) + \varepsilon)n}
% \]
for any positive $\varepsilon$ and sufficiently large $n$. The regularity condition is satisfied by both the Gumbel-max and inverse transform watermarks and is detailed in the remark following the proof of Theorem~\ref{thm:inhomo-type-II} in Appendix \ref{proof:main}.

\end{rem}
% This condition is satisfied by both the Gumbel-max and inverse transform watermark (see Lemma \ref{lem:worst-MGF} and \ref{lem:monotone-bar-L} for the details), implying our rates are tight in the least-favorable case.

% Moreover, if there exists some $\bP$ in the closure of $\PM$ which maximizes $\phi_{\bP, h}(\theta)$ simultaneously for all $\theta \ge 0$,
% then \eqref{eq:inequality_Eff} is tight in the sense that there exists $\bP^\star$ in $\PM$ such that, if $\bP_t = \bP^\star$ for all $t$, then
% \[
% \PB_{H_1}(T_h(Y_{1:n}) = 0) \ge  \mathrm{e}^{-(R_{\PM}(h) + \varepsilon)n}
% \]
% for any positive $\varepsilon$ and sufficiently large $n$.
% This condition is satisfied by both the Gumbel-max and inverse transform watermark (see Lemma \ref{lem:worst-MGF} and \ref{lem:monotone-bar-L} for the details), implying our rates are tight in the least-favorable case.

Proposition \ref{thm:inhomo-type-II} shows the Type II error decays exponentially as long as the exponential component $R_{\PM}(h)$ is positive. The larger the value of $R_{\PM}(h)$ is, the more efficient the detection rule $T_h$ is. Put differently, this class-dependent measure of efficiency reduces comparing watermark detection rules to the rate of class-dependent efficiency. 
% If there exists an NTP distribution that maximizes all MGFs $\phi_{\bP, h}(\theta)$

\new{
This efficiency rate, obtained from large deviation theory, does not depend on the specific value of $\alpha$. The same is true for Bahadur efficiency \citep{bahadur1960asymptotic}, which quantifies the rate at which the significance level ($p$-value) of a test statistic approaches zero with increasing sample size.} It emphasizes the control of Type I errors for simple hypothesis testing with i.i.d. data. In contrast, our metric $R_{\PM}(h)$ centers on quantifying the least decline of Type II errors to zero when the underlying NTP distribution $\bP$ is picked up from the distribution class $\PM$.
In short, it focuses on composite hypothesis testing where the observed data is not identically distributed.

\paragraph*{Optimality via minimax optimization}
More importantly, the notion of class-dependent efficiency serves as a concrete approach to identifying the optimal score function that achieves the largest possible value of $\PM$-efficiency rate. Following from \eqref{eq:efficiency-exponent}, formally, this amounts to solving the following optimization problem:
\[
\begin{aligned}
\sup_{h} R_{\PM}(h)
%& = -\inf_h\inf_{\theta\ge 0} \sup_{\bP \in \PM}
%\left(\theta  \EB_0 h(Y) +  \log  \phi_{\bP, h}(\theta) \right)\\
&= -\inf_{h, \theta\ge 0} \sup_{\bP \in \PM}
\left(\EB_0 \theta h(Y) +  \log \EB_{1, \bP} \mathrm{e}^{-\theta h(Y)} \right).
\end{aligned}
\]
By viewing $\theta h$ as a new score function, finding the optimal $h$ is reduced to solving the following minimax optimization program:\footnote{To highly the minimax nature of this formulation, we use $\max$ and $\min$ in place of $\sup$ and $\inf$, respectively.} 
\begin{equation}\label{eq:minmax}
\min_{h}\max_{\bP \in \PM} L(h, \bP)
~~\text{where}~~L(h, \bP) :=
\EB_0 h(Y) + \log \EB_{1, \bP}\mathrm{e}^{-h(Y)}.
\end{equation}
The function $L(h, \bP)$ is convex in the score function $h$ for any fixed $\bP$, but is generally not concave in $\bP$ when $h$ is fixed. Therefore, this minimax optimization problem is generally not convex-concave, making it challenging to solve \eqref{eq:minmax} numerically~\citep{lin2020near,rahimian2022frameworks}.

Interestingly, we will show in Sections~\ref{sec:gumbel} and \ref{sec:inverse} that, for the Gumbel-max and inverse transform watermarks, the minimax optimization problem possesses certain structural properties that enable us to \emph{analytically} identify global solutions to \eqref{eq:minmax}.

\paragraph*{Choice of distribution classes}
The choice of the distribution class $\PM$ is crucial since it follows from \eqref{eq:efficiency-exponent} that $R_{\PM_1}(h) \ge R_{\PM_2}(h)$ if $\PM_1 \subset \PM_2$. While it might be plausible to assume a fraction of NTP distributions are in a ``nice'' class, some might be outside. The following result extends Proposition \ref{thm:inhomo-type-II-general} to this practical scenario.

\begin{prop}\label{thm:inhomo-type-II-general}
Assume that at least $\gamma$-fraction of $\bP_1, \ldots, \bP_n$ is in $\PM_1$ with the rest being in $\PM_2$,\footnote{We do not need to know which are in the $\gamma$ fraction.} where $0 < \gamma < 1$ and $\PM_1 \subset \PM_2$. Then, for any $h$, we have
\[
\limsup\limits_{n \to \infty}\PB_{H_1}(T_h(Y_{1:n}) = 0)^{1/n} \le  \mathrm{e}^{-\{\gamma \cdot R_{\PM_1}(h) + (1-\gamma) \cdot R_{\PM_2}(h)\}}.
\]
\end{prop}

When there is no prior at all about the LLMs, the associated class is the $(|\Voca|-1)$-dimensional simplex, which contains all possible NTP distributions and we denote by $\Simplex(\Voca)$. This distribution class includes \textit{singular} distributions---a distribution $\bP$ such that $P_{\token} = 1$ for some token $\token$---for which the pivotal statistic has the same distribution under the null and alternative in \eqref{eq:H-surrogate}. Formally, this gives $R_{\Simplex(\Voca)}(h)=0$, that is, the class-dependent efficiency of $\Simplex(\Voca)$ is zero. Taking $\PM_1 = \PM \subset \PM_2 = \Simplex(\Voca)$, Proposition~\ref{thm:inhomo-type-II-general} shows $\lim\sup\limits_{n \to \infty}\PB_{H_1}(T_h(Y_{1:n}) = 0)^{1/n} \le  \mathrm{e}^{- \gamma  R_{\PM}(h)}.$
% \[
% \lim\sup\limits_{n \to \infty}\PB_{H_1}(T_h(Y_{1:n}) = 0)^{1/n} \le  \mathrm{e}^{- \gamma  R_{\PM}(h)}.
% \]
This result implies that when only a fraction of NTP distributions is known to some extent, the Type II error is still governed by $R_{\PM}(h)$ up to a multiplicative constant $\gamma$. The comparison between choices of the score functions $h$ can still be reduced to comparing the value of $R_{\PM}(h)$.

\begin{figure}[t!]
\centering
\includegraphics[width=0.6\textwidth]{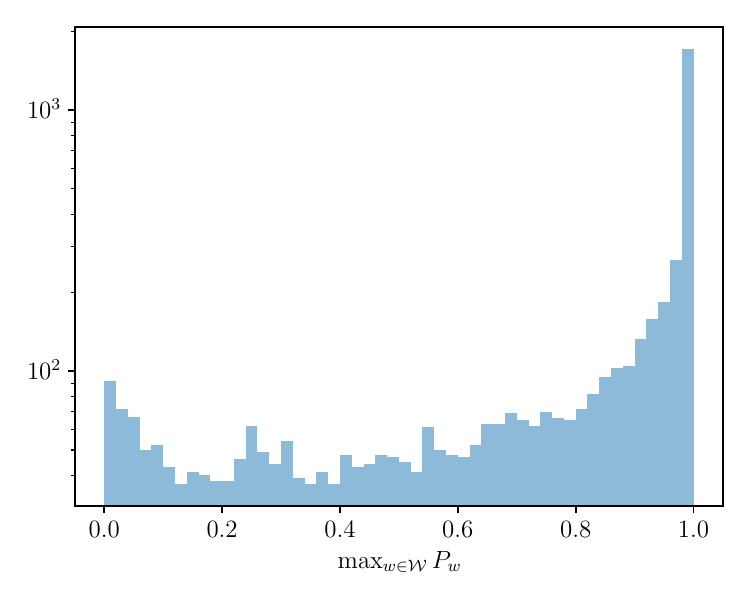}
\caption{
Empirical frequency of $\max_{\token \in \Voca} P_{t,\token}$ using outputs from {ChatGPT-3.5-turbo}. 
There are a total of 4,997 top-one probabilities recorded, 81.36\% of which are less than 0.999 while 99.79\% are larger than 0.001.  
See Appendix \ref{detail:empirical-delta} for the experimental setup.
}
\label{fig:empirical-delta}
\end{figure}

In addition, the discussion above reveals that any detection rule would become powerless if nothing is known about the NTP distributions. This justifies the necessity of imposing structural assumptions on the distribution class. For $0 < \Delta \le 1 - \frac1{|\Voca|}$, we call 
\begin{equation*}
% \label{eq:regular-PM}
\FPM := \Bigl\{\bP: \max_{\token \in \Voca} P_{\token} \le 1- \Delta \Bigr\}
\end{equation*}
the $\Delta$-regular distribution class. Accordingly, an NTP distribution $\bP$ is called $\Delta$-regular if $\bP\in \FPM$. This excludes singular distributions. Interestingly, $\Delta$-regularity is closely related to the Shannon entropy. For any $\Delta$-regular NTP distribution $\bP$, its Shannon entropy satisfies $\Ent(\bP) = \sum P_{\token} \log\frac1{P_\token} \ge \sum P_{\token} (1- P_\token) \ge \sum P_{\token} \cdot \Delta = \Delta$,
% \begin{equation*}
% % \label{eq:entropy}
% \Ent(\bP) = \sum P_{\token} \log\frac1{P_\token} \ge \sum P_{\token} (1- P_\token) \ge \sum P_{\token} \cdot \Delta = \Delta,
% \end{equation*}
where the first inequality follows because $\log \frac{1}{1-x} \ge x$ for $x < 1$. This shows that $\Delta$-regularity imposes a lower bound on how much information the NTP distribution offers. 
In practice, most NTP distributions are $\Delta$-regular for a proper value of $\Delta$.
See Figure \ref{fig:empirical-delta} for an empirical investigation where 88.65\% of the NTP distributions are 0.0001-regular, 81.36\% are 0.001-regular, and 70.28\% are 0.01-regular.

This distribution class offers a way to measure the performance of a detection rule using the class-dependent efficiency in \eqref{eq:efficiency-exponent}. For example, the $\FPM$-efficiency rate of the baby watermark using the statistic \eqref{eq:simple-cov} is
\[
-\inf_{\theta\ge0}\log\biggl[\frac{1}{\theta}\biggl\{\frac{\re^{\theta(1-2\Delta)}+\re^{-\theta(1-2\Delta)}}{2}-\re^{-\theta}\biggr\}\biggr].
\]
A proof of this fact is deferred to Appendix \ref{proof:baby-watermark}.

In practice, $\Delta$-regularity is closely related to the temperature parameter in LLMs \citep{ackley1985learning}, which is used to divide the raw output of a model's last layer by a temperature parameter before applying the softmax function. A high temperature leads to more uniform probabilities (encouraging exploration), while a low temperature makes the distribution sharper, emphasizing the most likely outcomes (encouraging exploitation).

\section{Application to the Gumbel-max Watermark}
\label{sec:gumbel}

In this section, we apply the framework to the Gumbel-max watermark \cite{scott2023watermarking}. Recall that the Gumbel-max decoder can also be written as %\wjs{the superscript of V should be gum.}
\begin{equation}
\label{eq:it-new}
\token_t = \SMmax(\bP_t, \xi_t) := \arg\max_{\token \in \Voca}  V^{\gumbel}(P_{t,\token}, U_{t, \token})~~\text{where}~~V^{\gumbel}(p, u) := \frac{\log u}{p}.
\end{equation}
Above, $\{\xi_t\}_{t=1}^n = \{ (U_{t, \token})_{\token \in \Voca} \}_{t=1}^n$ represents $n \times |\Voca|$ i.i.d.~replicates of the standard uniform random variable $U(0,1)$. As seen from \eqref{eq:it-new}, the Gumbel-max trick ensures that this decoder is unbiased for sampling from the NTP distribution $\bP_t$. 
In implementation, $\xi_t$ can be computed using a pseudorandom hash function $\AM$ that depends only on, for example, the last five tokens $\token_{t-5}, \ldots, \token_{t-1}$ in addition to the secrete key~\cite{scott2023watermarking}.

%%$\equiv \AM(\token_{1:(t-1)}, \Key),$ can be determined by, for example, the last five tokens using a pseudorandom hash function. 
%\paragraph*{Uniqueness of the decoder} 

\paragraph*{Uniqueness of the Gumbel-max decoder}
As a deviation from the main focus, we are tempted to ask whether other forms of Gumbel-max-style decoders remain unbiased. Consider an unbiased decoder that takes the form
\[
\SM^{V}(\bP, \xi) := \arg\max_{\token \in \Voca}  V(P_{\token}, U_{\token})
\]
for some function $V$. When the vocabulary size $|\Voca| = 2$, writing $\bP = (P_0, P_1)$, unbiasedness requires\footnote{In \eqref{eq:2-token}, it evaluates to $P_0$ since $P_0 + P_1 = 1$. But we regard \eqref{eq:2-token} as an identity for any $P_0$ and $P_1$ such that $P_0 + P_1 \le 1$.}
\begin{equation}\label{eq:2-token}
\PB(\SM^{V}(\bP, \xi) = 0) = \PB(V(P_0, U_0) \ge V(P_1, U_1)) = \frac{P_0}{P_0 + P_1}.
\end{equation}

When there are more than two tokens in the vocabulary, a natural requirement is that, by ``gluing'' two tokens into a new token, the decoder remains unbiased for the new vocabulary that has been reduced by one in size. This new token is sampled whenever either of the two original tokens is sampled, which occurs with probability $P_0+P_1$. This would be true if we have
\begin{equation}\label{eq:3-token}
\max\{V(P_0, U_0), V(P_1, U_1)\} \stackrel{d}{=} V(P_0+P_1, U),
%\PB(S^{V}(\bP, \xi) = 0) = \PB(V(P_0, U_0) \ge V(P_1, U_1)) = \frac{P_0}{P_0 + P_1}.
\end{equation}
where $\stackrel{d}{=}$ denotes equality in distribution and $U \sim U(0,1)$.

Both \eqref{eq:2-token} and \eqref{eq:3-token} are satisfied by Aaronson's choice $V^{\gumbel}$ for any $P_0$ and $P_1$ such that $P_0 + P_1 \le 1$. Interestingly, our following theorem shows that this choice is essentially unique. Its proof is deferred to Appendix \ref{proof:unique}.

\begin{thm}\label{prop:unique}
Suppose $V$ satisfies both \eqref{eq:2-token} and \eqref{eq:3-token}. $\SM^{V}$ is unbiased if and only if 
\[
V(p, u) = g\left(\frac{\log u}{p}\right) \equiv g\left(V^{\gumbel}(p, u)\right)
\]
for a strictly increasing function $g$.
 \end{thm}

\subsection{Main Results}

\paragraph*{Optimal score function}
Our framework starts by identifying a pivot for the hypothesis testing problem \eqref{eq:original_test}. The random number $U_{t, \token_t}$, which corresponds to the selected token $\token_t$ at step $t$, is a good candidate. This is because $U_{t, \token_t}$ follows the standard uniform distribution under $H_0$ while being stochastically larger under $H_1$, as ensured by Working Hypotheses \ref{hypo:cryptography} and \ref{hypo:human}. Indeed, $U_{t, \token_t}$ is used in \citep{scott2023watermarking} for its detection rule.

We write $\Yars_t \equiv \Yars(\token_t, \xi_t) := U_{t, \token_t}$. \new{
This choice of pivot can be seen as an approximation of the most ``optimistic'' log-likelihood ratio used in an independent hypothesis test for watermark detection (see Appendix \ref{append:gumbel-pivot} for details).} The distribution of this pivot under $H_1$ is a mixture of Beta distributions, which is a classic result in the literature.

%is characterized by the following lemma, which has been derived in  ~\citep{piet2023mark,fernandez2023three}.

%This is a mixture of Beta distributions, $\Bet(1/P_{t,\token}, 1)$, with weights $P_{t, \token}$.

%\wjs{derived or only related? if it was derived by other people, we may need to add their names in the lemma like \begin{lem}[\citep{piet2023mark,fernandez2023three}]\end{lem} \lx{ Already derived}}

\begin{lem}[\citep{piet2023mark,fernandez2023three}]\label{lem:dis-r}
%Under $H_0$, $\Yars_t$'s are i.i.d. copies of $\mu_0 = \UM(0, 1)$.
Under $H_1$, the distribution of $\Yars_t$ given $\bP_t$ obeys % \wjs{Xiang, do you want to use $y$ or $r$? }\lx{use $r$}
\begin{equation*}
%F_{1,\bP_t}(r)  = 
\PB_{H_1}( \Yars_t \le r\mid\bP_t) = \sum_{\token\in \Voca } P_{t,\token} r^{1/P_{t,\token}}~~\text{for $r \in [0, 1]$}.
\end{equation*}
%Its PDF is given by $f_{1,\bP_t}(r) = \sum_{\token\in \Voca }r^{1/P_{t,\token}-1}$.
\end{lem}

\begin{figure}[!t]
\centering
\includegraphics[width=\textwidth]{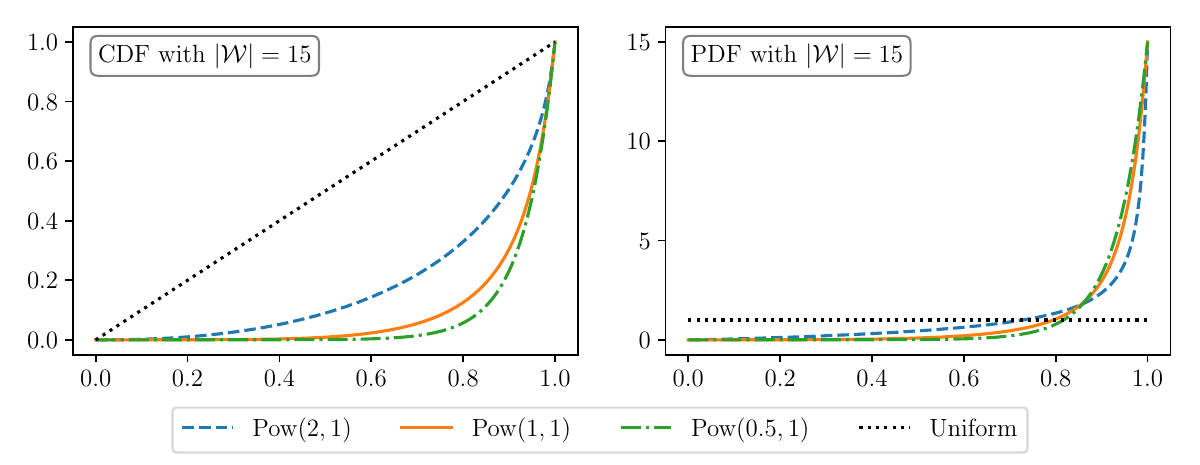}
\caption{Illustration of the distribution of $\Yars_t$.
We plot the CDF $F_{1,\bP_t}$ and the PDF $f_{\bP_t}$ by assuming $P_{t,\token} \propto (\token+b)^{-a}, \token \in \Voca$ with $|\Voca| = 15$ and different choices of $(a, b)$. The resulting distribution is denoted by $\mathrm{Pow}(a, b)$. $F_{1, \bP_t}$ deviates from the uniform distribution.
}
\label{fig:dist}
\end{figure}

The fact that this distribution is stochastically larger than $U(0,1)$ can be gleaned from the following inequality: $\sum_{\token\in \Voca } P_{t,\token} r^{1/P_{t,\token}} = r \sum_{\token\in \Voca } P_{t,\token} r^{1/P_{t,\token}-1} \le r \sum_{\token\in \Voca } P_{t,\token} = r.$
% \[
% \sum_{\token\in \Voca } P_{t,\token} r^{1/P_{t,\token}} = r \sum_{\token\in \Voca } P_{t,\token} r^{1/P_{t,\token}-1} \le r \sum_{\token\in \Voca } P_{t,\token} = r.
% \]
Also, see Figure~\ref{fig:dist} for an illustration: there is a higher accumulation of probability mass around one under $H_1$ compared to the distribution observed under $H_0$.
From this observation, any score function $h$ aware of the distributional differences could be used for detecting the Gumbel-max watermarks.
%\wjs{what is close-to-one tendency? \lx{The PDF for points around one is much larger than those for the rest points. See the right panel of Figure~\ref{fig:dist}.}} could be used for detecting the Gumbel-max watermarks.
For example, \citep{scott2023watermarking} uses $\hars(r)=-\log(1-r)$, while \citep{kuditipudi2023robust,fernandez2023three} propose $\hlog(r)=\log r$. 
All of these score functions are increasing functions that assign large values to points near one.
While all these choices control the Type I error, a crucial question is on what basis we compare score functions and which is the optimal one.

%For any $\bP_t$, the Beta mixture distribution is stochastically larger than the uniform distribution on $[0, 1]$; see Figure~\ref{fig:dist} for an illustration. Consequentially, $\Yars_t$ is more likely taking larger values under $H_1$ than under $H_0$. 

% Under the null hypothesis $H_0$, the independence between $\xi_t$ and $\token_t$ makes the summary statistic $\Yars_t$'s pivotal, as they are i.i.d. copies of $\UM(0, 1)$.
% Under the alternative hypothesis $H_1$, $\token_t$ is significantly influenced by the randomness variable $\xi_t$, given that it represents the optimal solution to the problem in \eqref{eq:it}.
% In fact, this dependence would cause $\xi_{t,\token_t}$ to concentrate at one rather than uniformly distributed on $[0, 1]$ (see Lemma \ref{lem:dis-r}).

To proceed under the framework, we consider the class-dependent efficiency with distribution class $\FPM$ for $\Delta \in [0, 1-|\Voca|^{-1}]$. As one of the main findings of this paper, we obtain the score function with the highest $\FPM$-efficiency rate. Its proof constitutes the subject of Section~\ref{sec:proofthm32}. Below, $\lfloor x\rfloor$ denotes the greatest integer that is less than or equal to $x$.

%% continue to work under the framework by focusing on the As a major finding of this paper, the following theorem uses the class-dependent efficiency with $\PM=\FPM$ as the criterion and identifies the optimal score function $h$. 
%that achieves the largest $\FPM$-dependence efficiency. 

\begin{thm}\label{thm:optimal_score}
%Recall that $R_{\FPM}(h)$ is the efficiency exponent in \eqref{eq:efficiency-exponent} with the prior set given by $\PM=\FPM$. 

Let $\hoptars: [0,1] \to \RB$ be defined as %\wjs{is using $z$ in place of $r$ better? this is because $z$ and $\zeta$ are related.}
\begin{equation}\label{eq:optimal-h-gum}
\hoptars(r)=\log\biggl(\biggl\lfloor\frac{1}{1-\Delta}\biggr\rfloor r^{\frac{\Delta}{1-\Delta}}+r^{\frac{\widetilde{\Delta}}{1-\widetilde{\Delta}}}\biggr)~~\text{with}~~\widetilde{\Delta}=(1-\Delta)\biggl\lfloor\frac{1}{1-\Delta}\biggr\rfloor.
\end{equation}
This function gives the optimal $\FPM$-efficiency rates in the sense that $R_{\FPM}(\hoptars)\ge R_{\FPM}(h)$
% \[
% R_{\FPM}(\hoptars)\ge R_{\FPM}(h)
% \]
for any measurable function $h$. Moreover, $R_{\FPM}(\hoptars)$ is attained at the following least-favorable NTP distribution in $\FPM$:
% , where % In this case, the optimal $\FPM$-efficiency rate is 
% \[
% R_{\FPM}(\hopt)= \KL(\mu_0, \mu_{1, \bP_{\Delta}^{\star}})
% \]
% where $\bP_{\Delta}^{\star}$ is the least-favorable NTP distribution defined by
\begin{equation}\label{eq:optimal-P}
\bP_{\Delta}^{\star} =\biggl(\underbrace{1-\Delta, \ldots, 1-\Delta}_{\floor{\frac{1}{1-\Delta}}\ \textnormal{times}}, 1-(1-\Delta)\cdot \left\lfloor\frac{1}{1-\Delta}\right\rfloor, 0, \ldots\biggr).
\end{equation}
\end{thm}

\begin{rem}
To clear up any confusion, the second conclusion amounts to saying that $(\hoptars, \bP_{\Delta}^{\star})$ is a saddle point solution to the minimax problem \eqref{eq:minmax}. Explicitly, $\bP$ maximizes $L(\hoptars, \bP)$ in \eqref{eq:minmax} if and only if $\bP$ results from rearranging $\bP_{\Delta}^{\star}$'s coordinates.
%can be obtained by permuting the coordinates of $\bP_{\Delta}^{\star}$.

%\wjs{$\PM_{\Delta} \subset \{\bP \in \Simplex(\Voca): \Ent(\bP) \ge C \}$? \lx{Yes, you can set $C=\Delta$ according to \eqref{eq:entropy}}.}
% In fact, $\bP_{\Delta}^{\star}$ have the smallest entropy under the constraint $\max P_{\token} \le 1 - \Delta$.
% \wjs{following Feng's question, is $\bP_{\Delta}^{\star}$ up to permutation the only vertex on $\PM_{\Delta}$? \lx{Yes. In fact any invariant convex $\PM$ will deduce the same $\bP_{\Delta}^{\star}$. As an concrete example, even we use $\Ent(\bP) \ge C$, the optimal $\bP$ is still (19).}} \wjs{todo: call it the least-favorable NTP distribution.}
\end{rem}

%Theorem \ref{thm:optimal_score} shows that when using $\Yars_t$ as the summary statistic, $\hopt$ is the optimal score function in terms of  $\FPM$-dependence efficiency.
%As a recap, a least-favorable distribution is where the supremum is attained in \eqref{eq:efficiency-exponent} with $h = \hoptars$. 
%%Furthermore, any invariant convex set $\PM$ will deduce the same $\bP_{\Delta}^{\star}$.

\begin{figure}[!t]
\centering
\includegraphics[width=\textwidth]{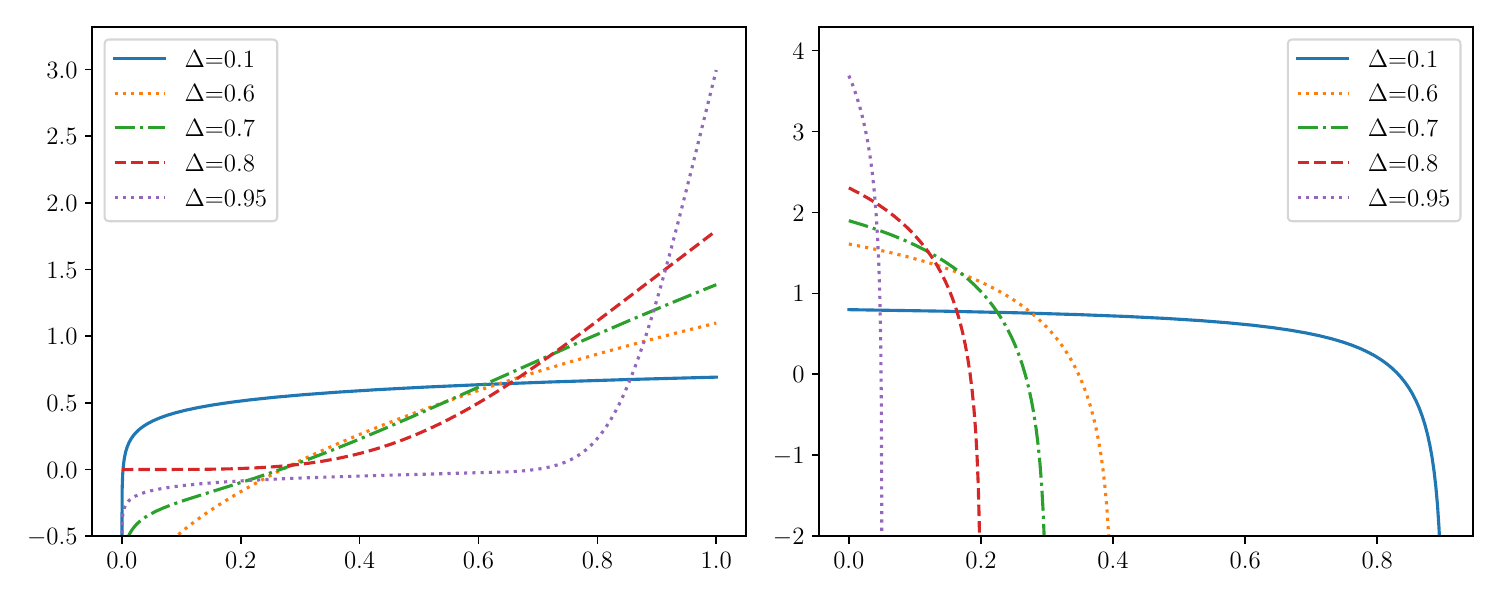}
\caption{Illustration of $\hoptars$ (left) and $\hoptdif$ (right) for different values of $\Delta$. 
}
\label{fig:optimal-score}
\end{figure}

The significance of Theorem~\ref{thm:optimal_score} lies partly in the closed-form nature of the optimal score function, thereby making it easy to use in detecting LLM-generated text. The left panel of Figure \ref{fig:optimal-score} shows $\hoptars$ for some values of $\Delta$. By the Neyman--Pearson lemma, the optimal score function and least-favorable distribution are related via $\hoptars = \log\bigl(\rd \mu_{1,\bP}/\rd \mu_0\bigr)\big|_{\bP=\bP_{\Delta}^{\star}}$. The left panel of Figure~\ref{fig:inhomo-efficiency} shows the efficiency $R_{\FPM}(\hoptars)$ as a function of $\Delta$. This function is not smooth when $\Delta = \frac{k}{k+1}$ for integer $k \ge 1$ because the support of $\bP^\star_{\Delta}$ jumps at these values.

%While our technical proofs are in the next subsection, we can appreciate it by recognizing that $\hoptars = \log\frac{\rd \mu_{1,\bP}}{\rd \mu_0}\big|_{\bP=\bP_{\Delta}^{\star}}$ corresponds to the log-likelihood ratio when the NTP distribution is set to be $\bP_{\Delta}^{\star}$.

%As an interesting fact, the least-favorable distribution $\bP_{\Delta}^{\star}$ and its permuted counterparts form \textit{all} the vertices of $\PM_{\Delta}$. Indeed, this is true for any permutation-invariant convex set of $\Simplex(\Voca)$.\wjs{have we proved this? \lx{Not yet.}} For example, all vertices of $\{\bP \in \Simplex(\Voca): \Ent(\bP) \ge C \}$ are given by \eqref{eq:optimal-P} up to permutation, with $\Delta$ depending on $C$. Moreover, as we will see in the proof of Theorem~\ref{thm:optimal_score} in Section~\ref{sec:proofthm32}, the least-favorable distribution must correspond to the vertices of the permutation-invariant convex distribution class. 

As an interesting fact, the least-favorable distribution $\bP_{\Delta}^{\star}$ and its permuted counterparts form \textit{all} the vertices of $\FPM$ (see Lemma~\ref{lem:extremal-points} in Section~\ref{sec:proofthm32}). Moreover, the vertices are the closest NTP distributions in $\FPM$ to singular distributions. Roughly speaking, the closer to singular distributions, the more difficult it is to tease apart the two hypotheses in \eqref{eq:H-surrogate}. To see this, note that the testing problem becomes most difficult when the NTP distribution is singular, which makes the null and alternative distributions of the pivot indistinguishable.

%As an interesting fact, the least-favorable distribution $\bP_{\Delta}^{\star}$ and its permuted counterparts form \textit{all} the vertices of $\PM_{\Delta}$. Moreover, the vertices are the closest NTP distributions in $\PM_{\Delta}$ to singular distributions. Roughly speaking, the closer to singular distributions, the more difficult it is to test apart the two hypotheses in \eqref{eq:H-surrogate}. Recall that the testing problem became most difficult when the NTP distribution is singular, which makes the null and alternative distributions of the pivot indistinguishable.

\paragraph*{Comparison with other rules}
While it is now known that $\hoptars$ is optimal in the sense of $\FPM$-efficiency, it remains interesting to analyze the relative efficiency of other existing rules. In addition to the aforementioned $\hars$ and $\hlog$, we consider $\hind(r)=\1_{\{r \ge \delta\}}$ for a threshold $0 < \delta < 1$, say, $\delta = \re^{-1}$.\footnote{According to the $\Rlimit$-efficiency defined in \eqref{eqn:limit-measure}, $\delta=\re^{-1}$ achieves the largest $\Rlimit$-efficiency when $\Delta \to 1$. See Lemma \ref{lem:ind_optim} for details.} This indicator function is included as it is perhaps the simplest score function detecting the distributional differences between $H_0$ and $H_1$.

%In the following discussion, we set $\Delta = \re^{-1}$ as the optimal threshold value.

By making use of Theorem~\ref{thm:inhomo-type-II} within our framework, we compare these four detection rules in terms of $\FPM$-efficiency in the theorem below.

%The following theorem compares these four detection rules in terms of $\FPM$-efficiency. 

\begin{figure}[t!]
\centering
\includegraphics[width=\textwidth]{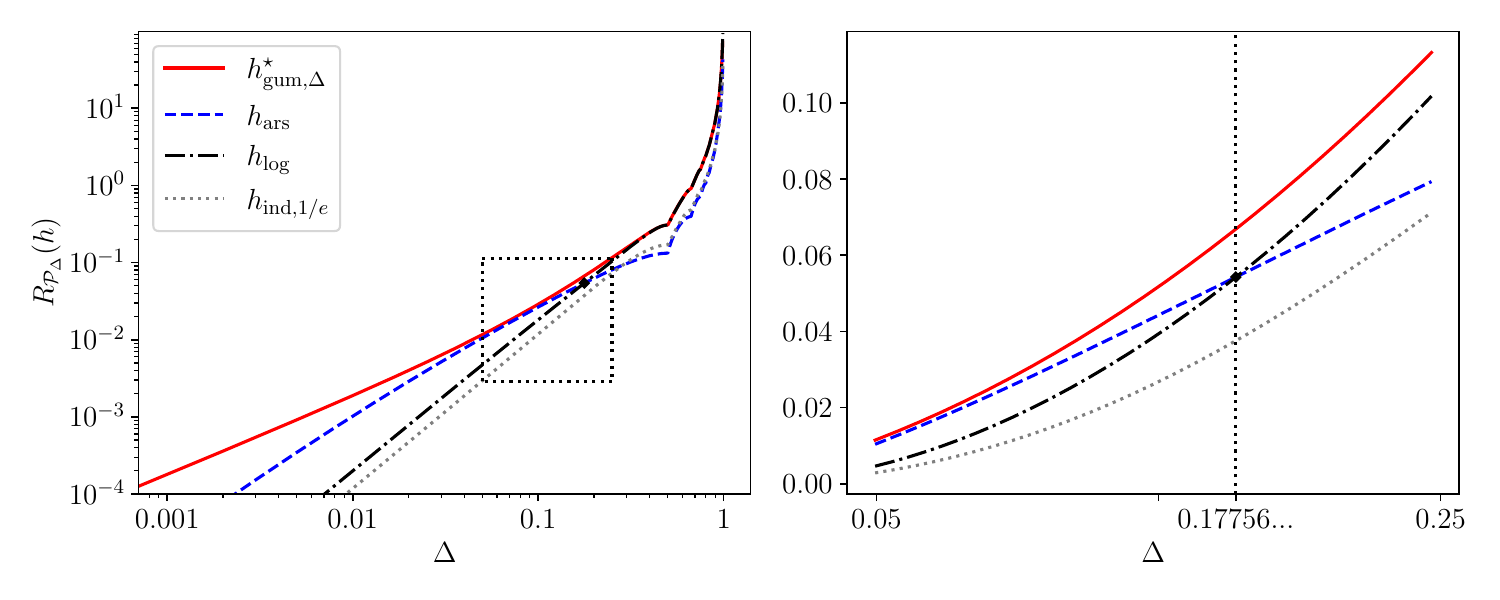}
\caption{
$\FPM$-efficiency rates of detection methods for Gumbel-max watermarks. 
Note that $R_{\FPM}(h)$ has non-smooth points when $\Delta = \frac12, \frac23, \frac34, \ldots$. 
% \wjs{some curves are not smooth when Delta is large. Is it normal? \lx{I think it is normal, because the support of $\bP_{\Delta}^{\star}$ would increase by one when $\Delta$ is larger than $\frac{k-1}{k}$ for integer $k$.}} 
% \wjs{Good. Please add some text in the caption to explain this; the red and orange curves are almost identical when delta is large?  \lx{Yes. They have the same rate converging to infinity} are the red and blue curves tangent to each other around $\Delta \approx 0.009$? \lx{No, they are just quite close.} for the x axis, use $0.1$ instead of $10^{-1}$. \lx{Done}}
}
\label{fig:inhomo-efficiency}
\end{figure}

\begin{thm}\label{thm:main-efficiency-gumbel-informal}
% \wjs{correct to say? There exists an absolute constant $\Delta^\star = 0.177\cdots$ such that the following two statements hold:}
% \lx{In fact, we can't rigorously prove $\Delta^\star$ is the transition point. We can prove that there exists $\Delta^\star_1 < \Delta^\star_2$ such that  $\Delta \in (0.001, \Delta^\star_1)$ is small,... and  $\Delta \in (\Delta^\star_2, 0.99)...$. From the figure, we conjecture that $\Delta^\star_1=\Delta^\star_2\approx 0.177$....The numerical solution (i.e. to solve the root $\Delta$ that makes the two $R_{\FPM}$'s equal) also suggests that $\Delta^\star_1=\Delta^\star_2$ should be equal.  }
There exists an absolute constant $\Delta^\star \approx 0.17756$ such that the following two statements hold:
\begin{itemize}
\item[(a)] When $0.001 < \Delta < \Delta^\star$, $\hars$ has higher $\FPM$-efficiency than both $\hlog$ and $\hindo$, that is, 
\[
\max\left\{R_{\FPM}(\hlog), R_{\FPM}(\hindo)\right\} 
< R_{\FPM}(\hars) < R_{\FPM}(\hoptars).
\]

\item[(b)] When $\Delta^\star < \Delta < 0.99$, $\hlog$ has higher $\FPM$-efficiency than both $\hars$ and $\hindo$, that is,
\[
\max\left\{R_{\FPM}(\hars), R_{\FPM}(\hindo)\right\} 
< R_{\FPM}(\hlog) < R_{\FPM}(\hoptars).
\]
\end{itemize}
\end{thm}

\begin{rem}
The constant $\Delta^\star$ is the unique root of the equation $R_{\FPM}(\hars) 
= R_{\FPM}(\hlog)$ for $\Delta > 0.001$. Numerical evaluation shows $\Delta^\star = 0.17756080525\cdots$. A proof of Theorem~\ref{thm:main-efficiency-gumbel-informal} can be found in Appendix \ref{proof:comparison}.
\end{rem}

\iffalse
\wjs{can we show via experiments that $\hopt$ with $\Delta = 0.2$ still outperforms gumbel if in reality $\Delta = 0.1$? \lx{Yes. But in my experience, it is $\hopt$ with $\Delta = 0.005$ outperform Gumbel if in reality $\Delta = 0.001$}  } \wjs{in the exp sec 5, can we say that our optimal rule's performance is still good if $\Delta$ is chosen with uncertainty? \lx{Yes. The reality $\Delta$ is randomly selected, but our $\Delta$-fixed optimal score function still performs well. }}
\fi

%\wjs{get we can more digits of $\Delta^\star$? 0.17756080525215662}

Figure~\ref{fig:inhomo-efficiency} illustrates Theorem~\ref{thm:main-efficiency-gumbel-informal} by presenting the comparisons under our statistical framework. Notably, for small values of $\Delta$, $\max_{\token \in \Voca} P_{t,\token}$ can be close to 1. In this regime, $\mu_{1, \bP_t}$ is close to $\mu_0$ in distribution, making the detection efficiency relatively low regardless of the rule used. Conversely, a large value of $\Delta$ implies all the token probabilities tend to be small, leading to a significant discrepancy between $\mu_{1, \bP_t}$ and $\mu_0$. Accordingly, all detection rules have efficiency rates tending to infinity. 

%is much more apparent, and thus the problem becomes so easier that both $\hars$ and $\hlog$ have efficiency rates increasing to infinity.

While $\hoptars$ achieves the highest $\FPM$-efficiency rate across the entire range of $\Delta$, the relative performance of $\hars$ and $\hlog$ depends on the value of $\Delta$. From an empirical perspective, $\hars$ is observed to outperform $\hlog$ in common settings of LLMs \citep{kuditipudi2023robust,fernandez2023three}. This is because the largest token probability in NTP distributions is, by and large, close to 1. To appreciate this fact, we experiment on ChatGPT-3.5-turbo with twenty prompts (see Appendix \ref{detail:empirical-delta} for details) and track the largest token probabilities across the generation of token sequences. The results are presented in Figure \ref{fig:empirical-delta}, which shows that, for example, 56.85\% of the NTP distributions have the largest token probability above $1 - \Delta^\star \approx 0.82$. This is the regime where $\hars$ is superior to $\hlog$, as predicted by our Theorem~\ref{thm:main-efficiency-gumbel-informal}.

%$81.36\%$ of NTP distributions are $0.001$-regular, $70.28\%$ are $0.01$-regular, and $58.35\%$ is $0.05$-regular.
% In practice, practitioners also use temperature parameters to tune the uncertainty of generated texts: increasing the temperature parameter makes the NTP distribution more uniform and thus increases the $\Delta$ values.
%This observation suggests that in practice, we should believe $\Delta$ is a small number, which means that some LLM-generated tokens are close to deterministic. From Theorem \ref{thm:main-efficiency-gumbel-informal}, $\hars$ has larger efficiency when $\Delta$ is small. This investigation helps explain why Aaronson's score function $\hars$ is effective in practice.\wjs{but the optimal score function is even better when delta is small, right? \lx{Yes}}

% Nevertheless, $\Delta$ can be large when the LLM sets a high temperature.
Nevertheless, $\Delta$ can be large when the LLM generation is highly stochastic. 
In this regime, $\hindo$ slightly outperforms $\hars$, but the gap soon diminishes to zero as $\Delta$ tends to 1, as seen from Figure~\ref{fig:inhomo-efficiency}. Similarly, $\hlog$ has a diminishing suboptimality compared with the optimal $\hoptars$ in the sense that $\lim\limits_{\Delta \to 1} R_{\FPM}(\hlog)/R_{\FPM}(\hoptars) \to 1$. We prove the limiting behaviors of these score functions in Appendix \ref{appen:effiency-gap}.

\subsection{Proof of Theorem \ref{thm:optimal_score}}
\label{sec:proofthm32}

%\wjs{some text moved here: The key property in finding this optimality is that the MGF of $-h(\Yars_t)$ is a convex function on $\bP$ as long as $h$ is non-decreasing.
%This property, together with, the fact that $R_{\FPM}(h)$ is invariant to the order of $\bP$, implies that the supreme over $\FPM$ in \eqref{eq:efficiency-exponent} is achieved at $\bP = \bP_{\Delta}^{\star}$.
%Without the troublesome supreme over $\FPM$, the problem in \eqref{eq:efficiency-exponent} can be solved analytically using Donsker-Varadhan representation %\citep{donsker1983asymptotic}. }
%\lx{Perhaps we don't need to mention this anymore since the following is a proof sketch. This paragraph is just an overview of the proof.}

% Thus, under $H_0$, the conditional distribution of $\Yars_t$ given $\bP_t$ is uniform over $[0, 1]$, while under $H_1$, it is a mixture of Beta distribution $\Bet(1/P_{t,\token}, 1)$ with weights $P_{t, \token}$. For any $\bP_t$, the Beta mixture distribution is stochastically larger than the uniform distribution on $[0, 1]$; see Figure~\ref{fig:dist} for an illustration. Consequentially, $\Yars_t$ is more likely taking larger values under $H_1$ than under $H_0$. Intuitively, this suggests that a good detection rule $h$ using $\Yars_t$ should be monotonically increasing. 

% In this section, we prove Theorem \ref{thm:optimal_score}.

% \hyw{self-reminder: use $r$ in distribution.}

The proof of Theorem~\ref{thm:optimal_score} relies on the minimax formulation \eqref{eq:minmax} in our framework to solve for the optimal score function. As mentioned earlier, however, this task is generally difficult since $L(h, \bP) =
\EB_0 h(Y) + \log \EB_{1, \bP}\mathrm{e}^{-h(Y)}$ is not convex-concave. We circumvent this obstacle by introducing a simple yet remarkable result that characterizes the MGF \eqref{eq:class-dependent-MGF} associated with the Gumbel-max watermark.

%The crucial property is that $F_{1,\bP}(r)$ is convex in $\bP$ for any $r \in [0,1]$.
%It implies that any MGF is also convex in $\bP$ by integration by parts.
\begin{lem}[Convexity Lemma]
\label{lem:convex-MGF}
For any non-decreasing function $h$, the functional\footnote{We consider a slightly different version of the MGF by setting $\theta = 1$ and taking as input $\bP$ in \eqref{eq:class-dependent-MGF}.} $  \bP \mapsto \phi_{h}(\bP) := \EB_{1,\bP} \re^{- h(\Yars)} 
$
% \begin{equation*}
%     \bP \mapsto \phi_{h}(\bP) := \EB_{1,\bP} \re^{- h(\Yars)} 
% \end{equation*}
is convex in $ \bP \in \Simplex(\Voca)$.
\end{lem}

This convexity lemma is the key to the proof of Theorem \ref{thm:optimal_score} and is a contribution of independent interest to future work on LLM watermarks. Roughly speaking, it helps reduce the ``max'' part in the minimax problem \eqref{eq:minmax} to the problem of identifying vertices of the distribution class. 
%We expect that Proposition~\ref{lem:convex-MGF} is useful in developing optimal detection rules within other family of distributions $\mathcal{P}$, making it of independent interest.
%This observation is important because it allows us to characterize $R_{\FPM}(h)$ for every monotonically increasing function $h$. 
Technically speaking, the role of the convexity lemma in the proof of Theorem \ref{thm:optimal_score} is through the following lemma.

% an important consequence of the convexity lemma leads to the following result.
% allows 
% Lemma~\ref{lem:worst-MGF} follows from Lemma~\ref{lem:convex-MGF} and standard convex analysis. 

\begin{lem}\label{lem:worst-MGF}
For any non-decreasing function $h$, we have
% for every $\theta \ge 0$: %\wjs{for all theta or only for non negative?} 
\begin{equation}
\label{eqn:real-bound}
    \sup_{\bP \in \FPM}\phi_{h}(\bP) = \phi_{h}(\bP_{\Delta}^{\star}).
\end{equation}

\end{lem}
%With Lemma \ref{lem:worst-MGF}, we are ready to prove Theorem \ref{thm:optimal_score}.

\begin{rem}
A direct consequence of Lemma~\ref{lem:worst-MGF} is that $ R_{\FPM}(h) = -\inf_{\theta\ge 0}L(\theta h,\bP_{\Delta}^{\star}).$
% \begin{equation*}
%     R_{\FPM}(h) = -\inf_{\theta\ge 0}L(\theta h,\bP_{\Delta}^{\star}). 
% \end{equation*} 
\end{rem}

Now we prove Lemmas~\ref{lem:convex-MGF} and \ref{lem:worst-MGF} in order.

%\begin{rem}
% For the Gumbel-max watermark, Proposition~\ref{lem:convex-MGF} ensures the convexity of $\bP \mapsto \EB_{1,\bP} e^{- h(Y)}$, which, along with convex analytic tools, facilitates identifying the global optimum of the maximization component within the minimax formulation~\eqref{eq:minmax}. See Lemma~\ref{lem:worst-MGF} below.

%\wjs{we cannot say $L(h, \bP)$ is convex in $\bP$, right? because $L$ takes log on $\phi_{\bP,h}(\theta)$}

%This result is used as a lemma for the proof of the main results in this section. However, we prefer to state it as a proposition in recognition of its fundamental role in delineating the efficiency for the Gumbel-max watermark, which is of independent interest for future work on Gumble-max-style watermarks.
%\end{rem}

%\wjs{say that L is not convex concave.}

\begin{proof}[Proof of Lemma \ref{lem:convex-MGF}]

By Lemma~\ref{lem:dis-r}, the CDF of $\Yars$ with NTP distribution $\bP$ takes the form: $F_{1,\bP}(r) = \sum_{\token\in \Voca } P_{\token} r^{1/P_{\token}}$ for $r \in [0,1]$.
% \begin{equation*}
%     F_{1,\bP}(r) = \sum_{\token\in \Voca } P_{\token} r^{1/P_{\token}}~~\text{for $r \in [0,1]$}.
% \end{equation*}
This function relates to $\phi_{h}(\bP)$ since
\[
\phi_{h}(\bP) = \EB_{1,\bP} \re^{- h(\Yars)} = \int_0^1 \re^{- h(r)}  F_{1, \bP}(\rd r).
\]

First, we show that $\bP \mapsto F_{1,\bP}(r)$ is a convex function for any given $r \in [0, 1]$. This is demonstrated by showing that the Hessian matrix of $\bP \mapsto F_{1,\bP}(r)$ is positive semidefinite for any $\bP$ within the interior of its domain $\Simplex(\Voca)$:  
\begin{equation*}
    \nabla_{\bP}^2 F_{1,\bP}(r) = 
    \begin{bmatrix}
        r^{1/P_{1}} \frac{\log^2 r}{P_{1}^3} & 0 & \ldots & 0 \\
        0 & r^{1/P_{2}} \frac{\log^2 r}{P_{2}^3} & \ldots & 0 \\
        \ldots & \ldots & \ldots & \ldots \\
        0 & 0 & \ldots & r^{1/P_{|\Voca|}} \frac{\log^2 r}{P_{|\Voca|}^3} 
    \end{bmatrix}
        \succeq 0.
\end{equation*}
Note the Hessian is diagonal with all non-negative entries, confirming the convexity of $\bP \mapsto F_{1,\bP}(r)$.

Second, we examine the functional form of $\phi_{h}(\bP)$ through integration by parts, which yields:%\footnote{Here, we interpret the non-decreasing function $h$ as a signed measure on $[0, 1]$, defined by $h((a, b]) = h(b)-h(a)$, with a slight abuse of notation. The same notation is also applied in Lemma \ref{lem:monotone-bar-L}.}
\begin{align*}
\begin{split}
\phi_{h}(\bP)&=F_{1,\bP}(r) \re^{- h(r)}\bigg|_0^1+\int_0^1 F_{1,\bP}(r) \re^{-h(r)}h(\rd r) 
=\re^{- h(1)}+\int_0^1F_{1,\bP}(r)\re^{-h(r)}h(\rd r).
\end{split}
\end{align*}
This equation expresses $\phi_{h}(\bP)$ as a nonnegative weighted sum of $F_{1, \bP}(r)$ evaluated over $r \in [0, 1]$. 

Given the established convexity of $\bP \mapsto F_{1,\bP}(r)$, and considering that a nonnegative weighted sum of convex functions remains convex, it follows directly that $\bP \mapsto \phi_{h}(\bP)$ is convex. This completes the proof of Lemma~\ref{lem:convex-MGF}.
\end{proof}

%\begin{rem}
%This does not apply to the inverse.
%In fact, $F_{1, \bP}$ for $\Ycov$ is not convex in $\bP$.
%We will use figures to show this. 
%\end{rem}

\begin{proof}[Proof of Lemma \ref{lem:worst-MGF}]
Our proof is based on a fundamental principle in convex analysis: the supremum 
of a convex function over a compact convex set in a Euclidean space is necessarily attained at an extreme point of the convex set \citep{HiriartJeLe96}.

Fix a non-decreasing function $h$. The mapping $\bP\mapsto \phi_{h}(\bP)$ is shown to be convex in $\bP$ by Lemma~\ref{lem:worst-MGF}. Additionally,  $\FPM$ is convex. Thus, the supremum of the convex function $\bP \mapsto \phi_{h}(\bP)$ must be necessarily attained at an extreme point of the convex constraint set $\FPM$. By Lemma \ref{lem:extremal-points}, any extreme point of the set $\FPM$ must be a permuted $\bP_{\Delta}^{\star}$.

\begin{lem}\label{lem:extremal-points}
The set of extreme points of $\FPM$, denoted by $\Ext(\FPM)$, is given by 
\begin{equation*}
    \Ext(\FPM) = \left\{\pi(\bP_{\Delta}^{\star}): \pi~\text{is a permutation on $\{1, 2, \ldots, |\Voca|\}$} \right\},
\end{equation*}
where $\pi(\bP)$ denote the permuted NTP distribution whose $i$th coordinate is $P_{\pi(i)}$.
\end{lem}

% In Appendix xxx, we identify the set of extremal points of $\FPM$, for which we describe. 
% Recall \begin{equation*}%\label{eq:optimal-P}
% \bP_{\Delta}^{\star} =\biggl(\underbrace{1-\Delta, \ldots, 1-\Delta}_{\floor{\frac{1}{1-\Delta}}\ \textnormal{times}}, 1-(1-\Delta)\cdot \left\lfloor\frac{1}{1-\Delta}\right\rfloor, 0, \ldots\biggr).
% \end{equation*}
% For a vector $\x \in \R^{|\Voca|}$, let $\pi(\x)$ denote a vector whose $i$th coordinate is given by $\x_{\pi(i)}$. 
% By Lemma xxx in Appendix xxx, the set of extremal points of $\FPM$, denoted by $\Ext(\FPM)$, is given by \wjs{this fact is earlier brought up in the discussion following Thm 3.2. If there is such a lemma in the appendix, please add a pointer to the lemma in the discussion following thm 3.2.}
% \begin{equation*}
%     \Ext(\FPM) = \left\{\pi(\bP_{\Delta}^{\star}): \pi~\text{is a permutation on $\{1, 2, \ldots, |\Voca|\}$} \right\}.
% \end{equation*}

Note there is a permutation invariance in the objective: $\phi_{h}(\bP) = \phi_{h}(\pi(\bP))$ holds for every permutation $\pi$ on $\{1, 2, \ldots, |\Voca|\}$ and every vector $\bP$. As a consequence, we can explicitly compute the supremum of $\phi_{h}(\bP)$ on the convex set $\FPM$ as follows: 
\begin{equation*}
    \sup_{\bP \in \FPM}\phi_{h}(\bP)
        = \sup_{\bP \in \Ext(\FPM)}\phi_{h}(\bP)
        = \phi_{h}(\bP_{\Delta}^{\star}).
\end{equation*}
This proves equation~\eqref{eqn:real-bound}. The remaining statement in Lemma~\ref{lem:worst-MGF} then follows.
\end{proof}

We conclude this section by proving Theorem~\ref{thm:optimal_score}.

\begin{proof}[Proof of Theorem \ref{thm:optimal_score}]

The second part of this theorem follows directly from the convexity lemma (Lemma \ref{lem:convex-MGF}) and Lemma~\ref{lem:extremal-points}.

To prove the first part, we work on the minimax formulation \eqref{eq:minmax} and begin by recognizing that
\begin{equation}\label{eq:gum_proof}
\begin{aligned}
\min_{h}\max_{\bP \in \FPM} L(h, \bP) &=
\min_{h}\max_{\bP \in \FPM} \left( \EB_0 h(\Yars) + \log \phi_{h}(\bP) \right)\\
& \ge \min_{h} \left( \EB_0 h(\Yars) + \log \phi_{h}(\bP_\Delta^{\star}) \right) = -\KL(\mu_0, \mu_{1, \bP_\Delta^{\star}}),
\end{aligned}
\end{equation}
where the final equality is due to the Donsker--Varadhan representation~\citep{donsker1983asymptotic} and $\KL$ denotes the Kullback--Leibler divergence. Note that Lemma~\ref{lem:worst-MGF} can not be applied here because the minimization function $h$ is not necessarily non-decreasing. % at the moment. 

Due to the uniqueness of the Donsker--Varadhan representation, $\EB_0 h(\Yars) + \log \phi_{h}(\bP_\Delta^{\star})$ is strictly larger than $-\KL(\mu_0, \mu_{1, \bP_\Delta^{\star}})$ unless we take the log-likelihood ratio 
\begin{equation*}
    \hoptars(r) = \log \frac{\rd\mu_{1, \bP_\Delta^{\star}}}{\rd\mu_0}(r) = \log \biggl(\biggl\lfloor\frac{1}{1-\Delta}\biggr\rfloor r^{\frac{\Delta}{1-\Delta}}+r^{\frac{\widetilde{\Delta}}{1-\widetilde{\Delta}}}\biggr),
\end{equation*}
which is non-decreasing in $r$. In this case, we get
\begin{equation}\label{eq:gum_proof2}
\begin{aligned}
\max_{\bP \in \FPM} L(\hoptars, \bP) &= \max_{\bP \in \FPM} \left( \EB_0 \hoptars(\Yars) + \log \phi_{\hoptars}(\bP) \right)\\
& = \EB_0 \hoptars(\Yars) + \log \phi_{\hoptars}(\bP_\Delta^{\star})  = -\KL(\mu_0, \mu_{1, \bP_\Delta^{\star}}),
\end{aligned}
\end{equation}
where the final equality follows from Lemma~\ref{lem:worst-MGF} and the third equality is ensured by, again, the Donsker--Varadhan representation.

Taken together, \eqref{eq:gum_proof} and \eqref{eq:gum_proof2} show that $\hoptars$ is the unique score function that solves the minimax problem \eqref{eq:minmax} with $\bP_\Delta^{\star}$.\qedhere

\end{proof}

\section{Application to the Inverse Transform Watermark}
\label{sec:inverse}

In this section, we apply the framework to the 
inverse transform watermark \citep{kuditipudi2023robust}. Without loss of generality, below we take $\Voca = \{1, \ldots, |\Voca|\}$.
Recall that its decoder is defined as 
\begin{equation*}
\token_t = \SMinv(\bP_t, \xi_t) :=\pi_t^{-1}(F^{-1}(U_t; \pi_t)),
\end{equation*}
where $\zeta_t = (\pi_t, U_t)$ with $U_t \sim U(0, 1)$ and $\pi_t$ being sampled uniformly at random from all permutations on $\Voca$. Following the strategy in \cite{piet2023mark,kuditipudi2023robust}, $(\pi_t, U_t)$'s are jointly independent across the token sequence. For a given permutation $\pi$, $F(\cdot, \pi)$ above denotes the CDF under permutation: $F(x; \pi) = \sum_{\token' \in \Voca} P_{\token'} \cdot {\mathbf{1}}_{\{\pi({\token'})\le x\}}.$
% \begin{equation*}
% F(x; \pi) = \sum_{\token' \in \Voca} P_{\token'} \cdot {\mathbf{1}}_{\{\pi({\token'})\le x\}}.
% \end{equation*}
By construction, the inverse transform watermark is unbiased for sampling from the NTP distribution $\bP$. To see the unbiasedness directly, for any token $w$, note that $\SMinv(\bP, \xi) = \token$ if and only if 
\begin{equation}\label{eq:inv_dependence}
\sum_{\token' \in \Voca} P_{\token'} \cdot {\mathbf{1}}_{\{\pi({\token'}) < \pi({\token})\}} \le U \le \sum_{\token' \in \Voca} P_{\token'} \cdot
{\mathbf{1}}_{\{\pi({\token'})\le\pi({\token})\}},
\end{equation}
which occurs with probability $P_\token$ since the interval above has length $P_\token$.

\subsection{Main Results}
\begin{figure}[t!]
\centering
\includegraphics[width=\textwidth]{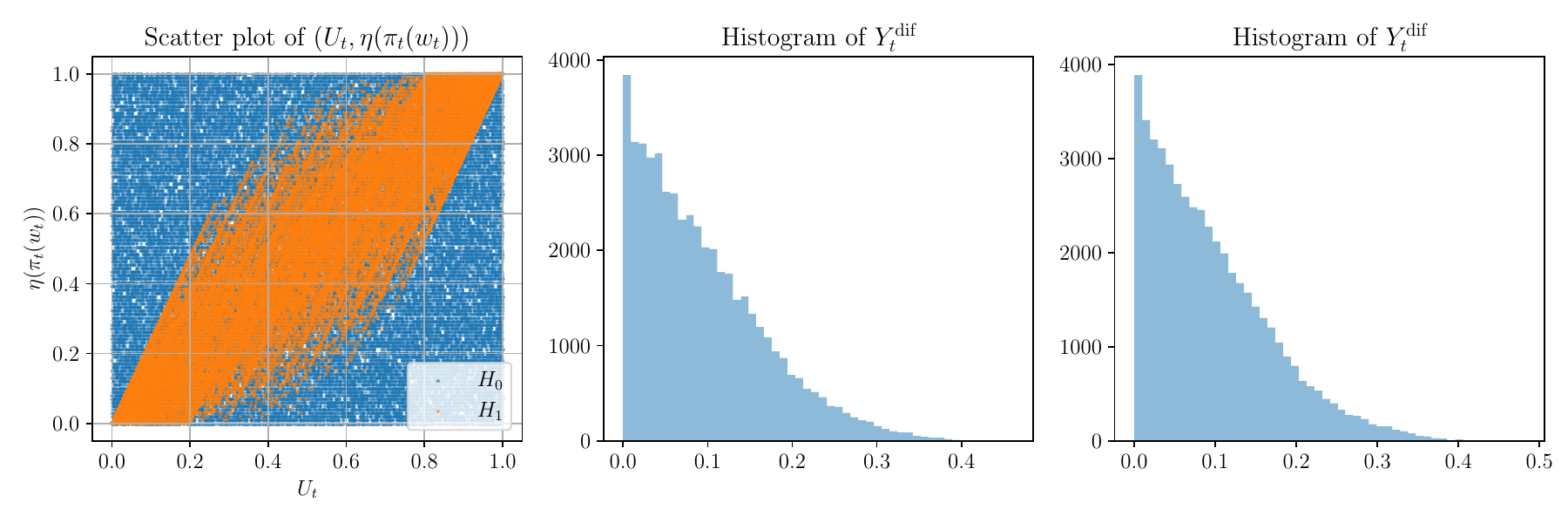}
\caption{
Left: the scatter plot of $5\times 10^4$ i.i.d.~observations from $(U_t, \eta(\pi_{t}(\token_t)))$ when $|\Voca| = 100$ and the top five probabilities in $\bP_t$ are fixed as $(0.2, 0.2, 0.1, 0.05, 0.05)$ and the rest probabilities are uniformly distributed.
Middle: histogram of $10^4$ i.i.d. observations of $\Ydif_t$ when the underlying NTP distribution is $\bP_t$.
Right: the same setting as the middle panel except that the probabilities except for the top five of $\bP_t$ follow normalized $\mathrm{Pow}(0.5, 1)$ values.
}
\label{fig:inverse-distribution}
\end{figure}

To utilize the framework in Section \ref{sec:overview} for the inverse transform watermark, we need to construct a pivotal statistic $Y_t=Y(U_t,\pi_t)$ such that its distribution is known under the null, and meanwhile should capture the dependence between $U_t$ and $\pi_t$ under the alternative for any NTP distribution. 
Under $H_0$, $U_t$ is independent of the human-written token $\token_t$, and $\pi_{t}(\token_t)$ is uniformly distributed over the vocabulary $\Voca$. Under $H_1$, in contrast, a larger $U_t$ value suggests a larger value of $\pi_{t}(\token_t)$ in distribution for the watermarked text, which can be gleaned from \eqref{eq:inv_dependence}. See the left panel of Figure \ref{fig:inverse-distribution} for a simulation that illustrates the distinct behaviors of $(U_t, \eta(\pi_t(\token_t))$ under $H_0$ and $H_1$. 

This motivates the use of the following statistic for watermark detection~\citep{kuditipudi2023robust,piet2023mark}:
\begin{equation*}
\Ydif_t = |U_t-\eta(\pi_t(\token_t))|,~~~\text{$\eta(i) := \frac{i-1}{|\Voca|-1}$},
\end{equation*}
where $\eta$ normalizes an integer token to $[0,1]$ for direct comparison with $U \in [0, 1]$. As is clear, this statistic is a pivot for our framework.

\paragraph*{Technical challenge}
To continue under our framework, we are faced with a difficulty in characterizing the distribution of the $\Ydif_t$ under the alternative. Although it is clear from Figure \ref{fig:inverse-distribution} that $\Ydif_t$ tends to be small under the alternative, the complex dependence of $\Ydif_t$'s distribution on the unknown NTP $\bP_t$ under $H_1$ renders the evaluation of class-dependent efficiency in the framework generally intractable.

Formally, this technical challenge can be elucidated by Lemma~\ref{lem:exact-CDF-dif} and its following remark.

\begin{lem}
\label{lem:exact-CDF-dif}
Under $H_0$, the CDF of $Y_t^{\dif}$ is
\[
F_{0}^\dif(r):= \PB_{H_0}(Y_t^{\dif} \le r)= \frac{1}{|\Voca|}\sum_{i=1}^{|\Voca|} \left[ 
\min\{\eta(i)+r, 1\}-\max\{\eta(i)-r, 0\}
\right].
\]
Under $H_1$, the CDF of $Y_t^{\dif}$ is 
\[
F_{1, \bP_t}^\dif(r) := \PB_{H_1}(Y_t^{\dif}\le r|\bP_t)=\frac{1}{|\Voca|!}\sum_{\pi \in \Pi} \sum_{i=1}^{|\Voca|} |(a_{\pi,i-1}, a_{\pi,i}]\cap B(\eta(i), r)|,
\]
where $\Pi$ collects all permutations on $\Voca$, $a_{\pi,i}=\sum_{j=1}^{i}P_{t,\pi(j)}$, $B(v, r)=\{x \in [0, 1]: |x-v| \le r\}$ and $|\cdot|$ represents the length of an interval. 
\end{lem}

In Lemma~\ref{lem:exact-CDF-dif}, the first part in particular shows that $Y_t^{\dif}$ is a pivotal statistic under $H_0$. The second part provides the explicit CDF of $Y_t^{\dif}$ under $H_1$ conditional on $\bP_t$. 

Nonetheless, Lemma~\ref{lem:exact-CDF-dif} also shows that the distribution formula for $\Ydif_t$ is by no means simple under the alternative hypothesis $H_1$: the dependence on $\bP_t$ is intricate. The intricacy arises due to the nature of $\SM$, whose definition involves all permutations $\Pi$ of $\Voca$, which also complicates the distribution formula for $ \Ydif_t = |U_t-\eta(\pi_{t}(\token_t))|$ where $\token_t = \SM(P_t, (U_t, \pi_t))$. Such complexity in the distribution formula poses significant technical challenges in evaluating the effectiveness of any test procedure under the alternative hypotheses $H_1$.

Considering these intricacies, our goal of deriving an optimal score function $h$ based on $\Ydif_t$ seems challenging at first glance. 

\paragraph*{Asymptotic distributions}

Fortunately, we discover a surprising result that simplifies the understanding of the distribution of $\Ydif_t$ under the alternative $H_1$, which facilitates later derivation of the optimal score function $h$. Roughly speaking, the conditional distribution of $\Ydif_t$ given $\bP_t$ under $H_1$  primarily depends on the largest values among $\bP_t$'s coordinates, especially when the vocabulary size $|\Voca|$ is large.  
For illustration, histograms in Figure~\ref{fig:inverse-distribution} compare $\Ydif_t$ under $H_1$ for different $\bP_t$ and $\bP_t'$ scenarios where $|\Voca| = 100$. Despite differences in $\bP_t$ and $\bP_t'$, each sharing the same top five coordinates but differing elsewhere, the conditional distribution of $\Ydif_t$ remains remarkably similar.

We now mathematically formalize this empirical observation under a special case where there is one token that takes predominantly high probability in $\bP_t$. This scenario, being the most basic case in theory, offers practical insights as empirically, many LLMs often concentrate most of $\bP_t$'s probability mass on a single token~\citep{gehrmann2019gltr}. Through our analysis, we derive the limiting distribution of $\Ydif_t$ given $\bP_t$ when the vocabulary size $|\Voca| \to \infty$. Formally, this means that we are effectively examining a sequence of $\bP_t$, which is indexed by $|\Voca|$, and investigate the limiting conditional distribution of $\Ydif_t$ given $\bP_t$ as we move along that sequence. Yet when the context makes it clear, we suppress the dependence of $\bP_t$ on the index $|\Voca|$ for notation simplicity. 

In below, we always use $P_{t,(i)}$ to refer to the $i$-th largest coordinates of $\bP_t$ for every vector $\bP_t$ and integer $i$ between $1$ and $|\Voca|$.

\begin{thm}
\label{thm:asymptotic-dif}
Under $H_0$, 
\begin{equation*}
\lim_{|\Voca| \to \infty}  \PB_{H_0}(\Ydif_t \le r)=1-(1-r)^2
~~\text{for any~~$r \in [0, 1]$}.
\end{equation*} 
Under $H_1$, assuming that $\lim\limits_{|\Voca| \to \infty} P_{t,(1)} = 1-\Delta$ and $\lim\limits_{|\Voca| \to \infty} \log|\Voca| \cdot P_{t, (2)} = 0$ hold, then 
\begin{equation*}
\lim_{|\Voca| \to \infty} \PB_{H_1}( \Ydif_t \le r \mid \bP_t) =1-\left(1 - \frac{r}{1-\Delta}\right)^2~~\text{for any~~$r \in [0, 1-\Delta]$}.
\end{equation*}
\end{thm}

%\begin{rem}
Theorem~\ref{thm:asymptotic-dif} formalizes this observed phenomenon, providing a simplified distribution characterization of the statistics $\Ydif_t$ under $H_1$, assuming that the token probability $\bP_t$ is highly concentrated at a single token, under the large vocabulary size limit $|\Voca| \to \infty$. The limiting distribution is determined by a single scalar $1-\Delta$ that corresponds to the top token probability. 

Notably, in this limit $|\Voca| \to \infty$, the statistics $\Ydif_t$ has a different support under $H_0$ and under $H_1$. 
By exploiting this distinction, we proceed to identify the optimal score $h$ under our framework, which achieves infinite power in the limit $|\Voca| \to \infty$ in distinguishing the hypotheses $H_0$ and $H_1$.

\paragraph*{Optimal score function} 
Following our previous discussions, we focus on scenarios where an LLM's probability distribution, $\bP_t$, is primarily concentrated on a single token. 

We model $\bP_t$ using a belief class $\SPM$, a subset of the $\Delta$-regular class $\FPM$. Unlike $\FPM$, which only restricts the highest token probability to $1-\Delta$, the belief class $\SPM$ also sets a limit on the second highest token probability, using a threshold $\eps_{|\Voca|}$:\footnote{%For simplicity, we suppress the notation dependence of $\SPM$ on $\eps_{|\Voca|}$ when the context is clear.
For simplicity, we omit the notation indicating the dependence of $\SPM$ on $\eps_{|\Voca|}$ when the context makes this clear.
}
\begin{equation}
\label{eq:PM-inv}
\SPM= \left\{ \bP: P_{(1)} \le 1- \Delta,~ \Psecond \le \eps_{|\Voca|} \right\}.
\end{equation}
In our subsequent asymptotic analysis $|\Voca| \to \infty$, we assume that $\log |\Voca|\cdot \eps_{|\Voca|} \to 0$, from which we can utilize Theorem~\ref{thm:asymptotic-dif} to simplify our distributional characterization of $\Ydif_t$.

This leads to the definition of a limit efficiency measure, $\Rlimit$, whose definition is based on our efficiency measure $R_{\SPM}$ under our framework in Section~\ref{sec:overview}: 
\begin{equation} 
\label{eqn:limit-measure}
\Rlimit(h) = \liminf  R_{\SPM}(h).
\end{equation}
where $\liminf$ denotes the large vocabulary size limit where $|\Voca| \to \infty$ and $\log |\Voca|\cdot \eps_{|\Voca|} \to 0$. In plain words, $\Rlimit(h)$ quantifies the efficiency of any score function $h$ in the limit of large vocabulary size $|\Voca| \to \infty$, assuming that the distribution of the underlying LLM predominantly focuses on a single token. 

We are ready to describe the optimal score function $h$ that maximizes this efficiency measure $\Rlimit(h)$. 
This optimal procedure is formally stated in Theorem~\ref{thm:optimal-score-inv-dif}.

\begin{thm}
\label{thm:optimal-score-inv-dif}
Fix $\Delta \in (0, 1]$. Let $\hoptdif: [0, 1]\to \RB$ denote 
\begin{equation*}
% \label{eq:optimal_score_dif}
\hoptdif(r) = \log \frac{f_{\dif, \Delta}(r)}{f_{\dif, 0}(r)}
~~\text{where}~~f_{\dif, \Delta}(r) = \frac{2}{1-\Delta}\cdot \max \left\{1 - \frac{r}{1-\Delta}, 0\right\} 
\end{equation*}
Then, $\lim_{M \to \infty} \Rlimit\left([\hoptdif]_{[-M, M]}\right) = \infty$ where $[x]_{[-M, M]} := \min\{\max\{x, -M\}, M\}$ denotes the clipping operator that restricts $x$ to the interval $[-M, M]$.
% \begin{equation}
% \label{eqn:main-result-inverse-two}
% \lim_{M \to \infty} \Rlimit\left([\hoptdif]_{[-M, M]}\right) = \infty.
% \end{equation}
\end{thm}

Theorem~\ref{thm:optimal-score-inv-dif} identifies the score function $\hoptdif$  that achieves infinite power in distinguishing $H_1$ from $H_0$ in the large vocabulary limit $|\Voca| \to\infty$, assuming that the NTP distribution $\bP_t$ focuses on a single token, formally described by $\SPM$ with $\log|\Voca| \cdot \eps_{|\Voca|} \to 0$. The right panel of Figure \ref{fig:optimal-score} shows $\hoptdif$ for different values of $\Delta$. 

Our derivation of $\hoptdif$ is based on calculating a log-likelihood ratio between the hardest alternative within $\SPM$ and the null, using the asymptotic distribution formula described in Theorem~\ref{thm:asymptotic-dif}. The log-likelihood ratio $\hoptdif$ achieves infinite power due to $\Ydif$ having distinct supports under $H_0$ and $H_1$ in the limit, as demonstrated in Theorem~\ref{thm:asymptotic-dif}. 

Finally, the truncation in $[\hoptdif]_{[-M, M]}$ is mainly for technical reasons, stemming from a proof artifact.

\subsection{Proof of Theorem \ref{thm:optimal-score-inv-dif}}

Fix $\Delta \in (0, 1]$. 
To prove Theorem~\ref{thm:optimal-score-inv-dif}, 
our goal is to establish the limit: $\lim_{M \to \infty} \Rlimit\left([\hoptdif]_{[-M, M]}\right) = \infty.$
% \begin{equation*}
% \lim_{M \to \infty} \Rlimit\left([\hoptdif]_{[-M, M]}\right) = \infty.
% \end{equation*}
Our strategy involves establishing an effective lower bound for $\Rlimit(h)$, as detailed
in Lemma~\ref{lem:key-prop-inverse}. This involves introducing a function 
$F_{\dif, \Delta}$ for each $\Delta \in [0, 1]$, defined as: $F_{\dif, \Delta} (r) = 
    1-\max\left\{1 - \frac{r}{1-\Delta}, 0\right\}^2$ for any $r \in [0, 1]$.
% \begin{equation*}
% F_{\dif, \Delta} (r) = 
%     1-\max\left\{1 - \frac{r}{1-\Delta}, 0\right\}^2~~\text{for any~~$r \in [0, 1]$}.
% \end{equation*}
By Theorem~\ref{thm:asymptotic-dif}, this corresponds to the limiting 
CDF of $\Ydif$ given $\bP_t$, assuming that $\bP_t \in \SPM$. 

\begin{lem} 
\label{lem:key-prop-inverse}
For any function $h$ that is non-increasing, Lipschitz-continuous on $[0, 1]$, there is the lower bound: 
\begin{equation}
\label{eqn:key-lower-bound-Rlimit}
\Rlimit(h) \ge 
- \left[  \int_{0}^1 h(r) F_{\dif, 0} (\rd r) + 
\log \int_{0}^1 \re^{- h(r)} F_{\dif, \Delta} (\rd r) \right].
\end{equation}
\end{lem}

The proof of Lemma~\ref{lem:key-prop-inverse} is deferred to 
Section~\ref{sec:proof-propp-key-prop-inverse}.

Now we finish the proof of Theorem~\ref{thm:optimal-score-inv-dif}. Let $G(h)$ denote the RHS of inequality~\eqref{eqn:key-lower-bound-Rlimit}. 
Also, let $\mu_{\dif, \Delta}$ denote the probability measure associated 
with the CDF $F_{\dif, \Delta}$.
By Donsker--Varadhan representation \citep{donsker1983asymptotic}, we deduce:  
\begin{equation*}
\sup_h G(h) = \sup_h  \left[-\int_{0}^1 h(r) F_{\dif, 0} (\rd r) - 
\log \int_{0}^1 \re^{- h(r)} F_{\dif, \Delta} (\rd r)\right]
= D_{{\rm KL}}(\mu_{\dif, 0}, \mu_{\dif, \Delta}) = \infty
\end{equation*}
where the last identity holds because 
the two probability measures $\mu_{\dif, 0}$ and $\mu_{\dif, \Delta}$ have different support, with $\mu_{\dif, 0}$ supported on $[0, 1]$ and $\mu_{\dif, \Delta}$ on $[0, 1-\Delta]$ for $\Delta > 0$.
This supremum $\sup_h G(h)$ is obtained for $h$ as the log-likelihood ratio: 
\begin{equation*}
\hoptdif = \log \frac{\rd\mu_{\dif, \Delta}}{\rd\mu_{\dif, 0}}
= \log \frac{ f_{\dif, \Delta}}{f_{\dif, 0}}.
\end{equation*}
Despite $G(\hoptdif) = \infty$, we are not able to use Lemma~\ref{lem:key-prop-inverse} 
to conclude that $\Rlimit(\hoptdif) = \infty$. This is because the function $\hoptdif$,
despite being non-increasing, is neither uniformly bounded nor Lipschitz-continuous on $\R$,
thus not meeting the requirements of Lemma~\ref{lem:key-prop-inverse}.

Nonetheless, after truncating $\hoptdif$ to $h_M = [\hoptdif]_{[-M, M]}$, the function $h_M$ is 
non-decreasing, Lipschitz-continuous, and uniformly bounded on $\R$ for any 
$M < \infty$. This allows us to use Lemma~\ref{lem:key-prop-inverse} 
to conclude that for every $M < \infty$: $\Rlimit(h_M) \ge G(h_M). $
% \begin{equation*}
% \Rlimit(h_M) \ge G(h_M). 
% \end{equation*}
Applying the limit $M \to \infty$ on both sides, and 
leveraging Lebesgue's dominated convergence theorem along with Fatou's lemma, we deduce: 
\begin{equation*}
\lim_{M \to \infty} \Rlimit\left([\hoptdif]_{[-M, M]}\right) = 
\lim_{M \to \infty} \Rlimit(h_M) \ge \lim_{M \to \infty} G(h_M) 
\ge G(\hoptdif) = \infty.
\end{equation*}
This yields Theorem \ref{thm:optimal-score-inv-dif} as desired.

\subsubsection{Proof of Lemma~\ref{lem:key-prop-inverse}}
\label{sec:proof-propp-key-prop-inverse}
\renewcommand{\TPMo}{\mathcal{Q}_{\Delta'}}
\renewcommand{\TPM}{\mathcal{Q}_{\Delta}}

We are interested in lower bounding $\Rlimit(h)= \liminf R_{\SPM}(h) $. Recall that its definition is given by: 
\begin{equation}
\label{eqn:redo-definition-of-Rlimith}
\begin{split}
\Rlimit(h) %&= \liminf R_{\SPM}(h) 
= -\limsup \inf_{\theta \ge 0}
\sup_{\bP \in \SPM}\left[ \theta \EB_0 h(\Ydif) + \log \EB_{1,\bP}[ \re^{- \theta h(\Ydif)}] \right].
\end{split}
\end{equation}
Our lower bound strategy relies on first connecting $\Rlimit(h)$ to an intermediate quantity that
\begin{equation*}
\overline{L}(h, \Delta'):= \limsup \sup_{\bP \in \TPMo}\left[ \EB_0 h(\Ydif) + \log \EB_{1,\bP}[ \re^{-h(\Ydif)}] \right]
\end{equation*}
where $\TPMo := \left\{ \bP : \Ptop = 1-\Delta', \log |\Voca|\cdot \Psecond \le \varepsilon_{|\Voca|} \right\}$. This connection is given in Lemma~\ref{lemma:technical-lemma-exchange-limit}.

\begin{lem}
\label{lemma:technical-lemma-exchange-limit}
For every function $h$ Lipschitz-continuous on $[0, 1]$, we have
\begin{equation*}
\Rlimit(h) \ge -
\sup_{\Delta' \ge \Delta}\overline{L}(h, \Delta').
\end{equation*}
\end{lem}

%\begin{rem}
Lemma \ref{lemma:technical-lemma-exchange-limit} implies that maximizing $\Rlimit(h)$ is related to the minimax optimization: $\inf_{h } \sup_{\Delta' \ge \Delta} \overline{L}(h, \Delta').
$
% \begin{equation*}
% % \label{eq:new-minimax}
% \inf_{h } \sup_{\Delta' \ge \Delta} \overline{L}(h, \Delta').
% \end{equation*}
This minimax formulation is closely related to our original minimax formulation~\eqref{eq:minmax}, yet it replaces the $|\Voca|$-dimensional variable $\bP$ with a scalar $\Delta'$ after taking asymptotics (Theorem \ref{thm:asymptotic-dif}). %The global 
%By Lemma \ref{lem:closed-form-bar-L}, we can evaluate $\overline{L}(h, \Delta')$ for a Lipschitz-continuous function $h$ and $\Delta' \in [0, 1]$.
%In this way, we can solve the minimax problem~\eqref{eq:new-minimax} within the domain of all Lipschitz-continuous functions.
% Through evaluation of the supremum in \eqref{eqn:new-supremum}, we identify the global optimum of the minimax problem~\eqref{eq:new-minimax}. It turns out that the minimax solution coincides with the optimal score $h$ maximizing $\Rlimit(h)$. This is because the optimal values for both optimizations are equal to $\infty$.
%\end{rem}

\begin{proof}[Proof of Lemma \ref{lemma:technical-lemma-exchange-limit}]

By definition, we have
\begin{equation*}
\inf_{\theta \ge 0}
\sup_{\bP \in \SPM}\left[ \theta \EB_0 h(\Ydif) + \log\EB_{1,\bP}[ \re^{-\theta h(\Ydif)}]  \right]
\le \sup_{\bP \in \SPM}\left[ \EB_0 h(\Ydif) + \log \EB_{1,\bP}[ \re^{-h(\Ydif)}] \right].
\end{equation*}
After we substitute this inequality into~\eqref{eqn:redo-definition-of-Rlimith}, 
we obtain a lower bound
\begin{align}
\Rlimit(h) 
&\ge -\limsup
\sup_{\bP \in \SPM}\left[ \EB_0 h(\Ydif) + \log  \EB_{1,\bP}[ \re^{-h(\Ydif)}] \right] \nonumber\\
&= -  \limsup \sup_{\Delta' \ge \Delta}
\sup_{\bP \in \TPMo}\left[ \EB_0 h(\Ydif) + \log  \EB_{1,\bP}[ \re^{-h(\Ydif)}] \right], \label{eq:lower-bound-Rlimit}
\end{align}
where the second identity holds because $\SPM = \cup_{\Delta' \ge \Delta} \TPMo$ by definition.

We wish to swap the order of $\sup_{\Delta' \ge \Delta}$ and $\limsup$ in~\eqref{eq:lower-bound-Rlimit} to reach the conclusion of Lemma~\ref{lemma:technical-lemma-exchange-limit}. Lemma~\ref{lemma:super-technical-exchange-order} justifies such a swap as valid. %, whose proof is deferred to Appendix \ref{proof:optimal-score-dif}.

\begin{lem}
\label{lemma:super-technical-exchange-order}
For any Lipschitz-continuous function $h$ on $[0,1]$, we have
\begin{equation*}
\begin{split} 
&\limsup \sup_{\Delta' \ge \Delta}
\sup_{\bP \in \TPMo}\left[ \EB_0 h(\Ydif) + \log  \EB_{1,\bP}[ \re^{-h(\Ydif)}] \right] \\
=&\sup_{\Delta' \ge \Delta} \limsup 
\sup_{\bP \in \TPMo}\left[ \EB_0 h(\Ydif) + \log  \EB_{1,\bP}[ \re^{-h(\Ydif)}] \right].
\end{split}
\end{equation*}
\end{lem}
By swapping $\sup_{\Delta' \ge \Delta}$ and $\limsup$ in \eqref{eq:lower-bound-Rlimit}, we obtain 
\begin{equation*}
\Rlimit(h) \ge - \sup_{\Delta' \ge \Delta}\limsup 
\sup_{\bP \in \TPMo}\left[ \EB_0 h(\Ydif) + \log  \EB_{1,\bP}[ \re^{-h(\Ydif)}] \right] = -\sup_{\Delta' \ge \Delta}\overline{L}(h, \Delta').
\end{equation*}
This completes the proof of Lemma \ref{lemma:technical-lemma-exchange-limit}.
\end{proof}

Lemma~\ref{lemma:technical-lemma-exchange-limit} reduces the problem of finding a lower bound on $\Rlimit(h)$ to evaluating $\sup_{\Delta' \ge \Delta}\overline{L}(h, \Delta')$. To do so, we first derive an explicit expression for $\overline{L}(h, \Delta')$.

\begin{lem}
\label{lem:closed-form-bar-L}
For any Lipschitz-continuous $h$ on $[0,1]$, we have
\begin{equation*}
\overline{L}(h, \Delta')= \int_{0}^1  h(r) F_{\dif, 0} (\rd r) + 
\log \int_{0}^1 \re^{- h(r)} F_{\dif, \Delta'} (\rd r).
\end{equation*}
\end{lem}

The proof of Lemma~\ref{lem:closed-form-bar-L} is deferred to Appendix \ref{proof:optimal-score-dif}. Here we give some intuition behind its proof. Theorem~\ref{thm:asymptotic-dif} states that as $|\Voca| \to \infty$, for any sequence of $\bP \in \TPMo$ indexed by $\Voca$, the CDF of $\Ydif$ under $H_0$ converges to $F_{\dif, 0}$, while its CDF under $H_1$ given $\bP$ converges to $F_{\dif, \Delta'}$. This results in a convergence statement that holds for any sequence of $\bP \in \TPMo$, and for any continuous function $h$ uniformly bounded on $\R$: 
\begin{equation*}
\lim \left[ \EB_0 h(\Ydif) + \log \EB_{1,\bP}[\re^{-h(\Ydif)}] \right]
=\int_{0}^1  h(r) F_{\dif, 0} (\rd r) + 
\log \int_{0}^1 \re^{- h(r)} F_{\dif, \Delta'} (\rd r).
\end{equation*}
Lemma~\ref{lem:closed-form-bar-L} strengthens this by ensuring such convergence can be made uniform across the sequence:
\begin{equation*}
\begin{split}
\overline{L}(h, \Delta') &= 
\limsup \sup_{\bP \in \TPMo}
\left[ \EB_0 h(\Ydif) + \log \EB_{1,\bP}[ \re^{-h(\Ydif)}] \right] \\
&=\int_{0}^1  h(r) F_{\dif, 0} (\rd r) + 
\log \int_{0}^1 \re^{- h(r)} F_{\dif, \Delta'} (\rd r).
\end{split}
\end{equation*}
The caveat is that the above uniform limit applies only to functions $h$ that are Lipschitz-continuous and uniformly bounded on $\R$, as required in Lemma~\ref{lem:closed-form-bar-L}.

Finally, with Lemma~\ref{lem:closed-form-bar-L} we are able to derive an explicit 
expression of the supremum $\sup_{\Delta' \ge \Delta}\overline{L}(h, \Delta')$. 

\begin{lem}\label{lem:monotone-bar-L}
For any non-increasing and Lipschitz-continuous $h$ on $[0,1]$,
\[
\sup_{\Delta' \ge \Delta}\overline{L}(h, \Delta') = \overline{L}(h, \Delta)
= \int_{0}^1  h(r) F_{\dif, 0} (\rd r) + 
\log \int_{0}^1 \re^{- h(r)} F_{\dif, \Delta} (\rd r).
\]
\end{lem}
\begin{proof}[Proof of Lemma \ref{lem:monotone-bar-L}]
Using integration by parts, we first obtain the equation:
\[
\int_{0}^1 \re^{- h(r)} F_{\dif, \Delta'}(\rd r)
=\re^{- h(1)} - \int_0^1 F_{\dif, \Delta'}(r) \re^{- h(r)} (-h)(\rd r).
\]
As a consequence, the above integral's value is non-increasing as $\Delta'$ increases because $F_{\dif, \Delta'}(r)$ is non-decreasing in $\Delta'$ for any $r$ within the interval $[0, 1]$.

By Lemma~\ref{lem:closed-form-bar-L}, this further implies that $\overline{L}(h, \Delta')$ is 
non-decreasing in $\Delta'$. Thus, $\sup_{\Delta' \ge \Delta}\overline{L}(h, \Delta') 
= \overline{L}(h, \Delta)$. The conclusion of Lemma~\ref{lem:monotone-bar-L} then follows. 
\end{proof}

We now finish the proof of Lemma~\ref{lem:key-prop-inverse}. By Lemma~\ref{lemma:technical-lemma-exchange-limit} and Lemma~\ref{lem:monotone-bar-L}, we obtain 
\begin{equation*}
\Rlimit(h) \ge -
\sup_{\Delta' \ge \Delta}\overline{L}(h, \Delta') 
= -  \int_{0}^1  h(r) F_{\dif, 0} (\rd r) -
\log \int_{0}^1 \re^{- h(r)} F_{\dif, \Delta} (\rd r).
\end{equation*}
This yields the lower bound as desired.

\section{Experiments}
This section highlights our framework's effectiveness through synthetic and real-data experiments, showcasing the practical utility of our proposed methods for detecting watermarks.\footnote{The experiment code is collected in the GitHub repository \url{https://github.com/lx10077/WatermarkFramework}.}

\subsection{Synthetic Studies}
\label{sec:simulation}

In our simulation, we generate a vocabulary $\Voca$ of size $1,000$ and assess Type I and II errors in watermark detection methods for generated token sequences. These methods include  $\hars(\Yars)$, $\hlog(\Yars)$, $\hindo(\Yars)$, $\hoptars(\Yars)$ for Gumbel-max watermark and $\hneg(\Ydif)$, $\hoptdif(\Ydif)$ for the inverse transform watermark.

\begin{figure}[t!]
\centering
\includegraphics[width=1.0\textwidth]{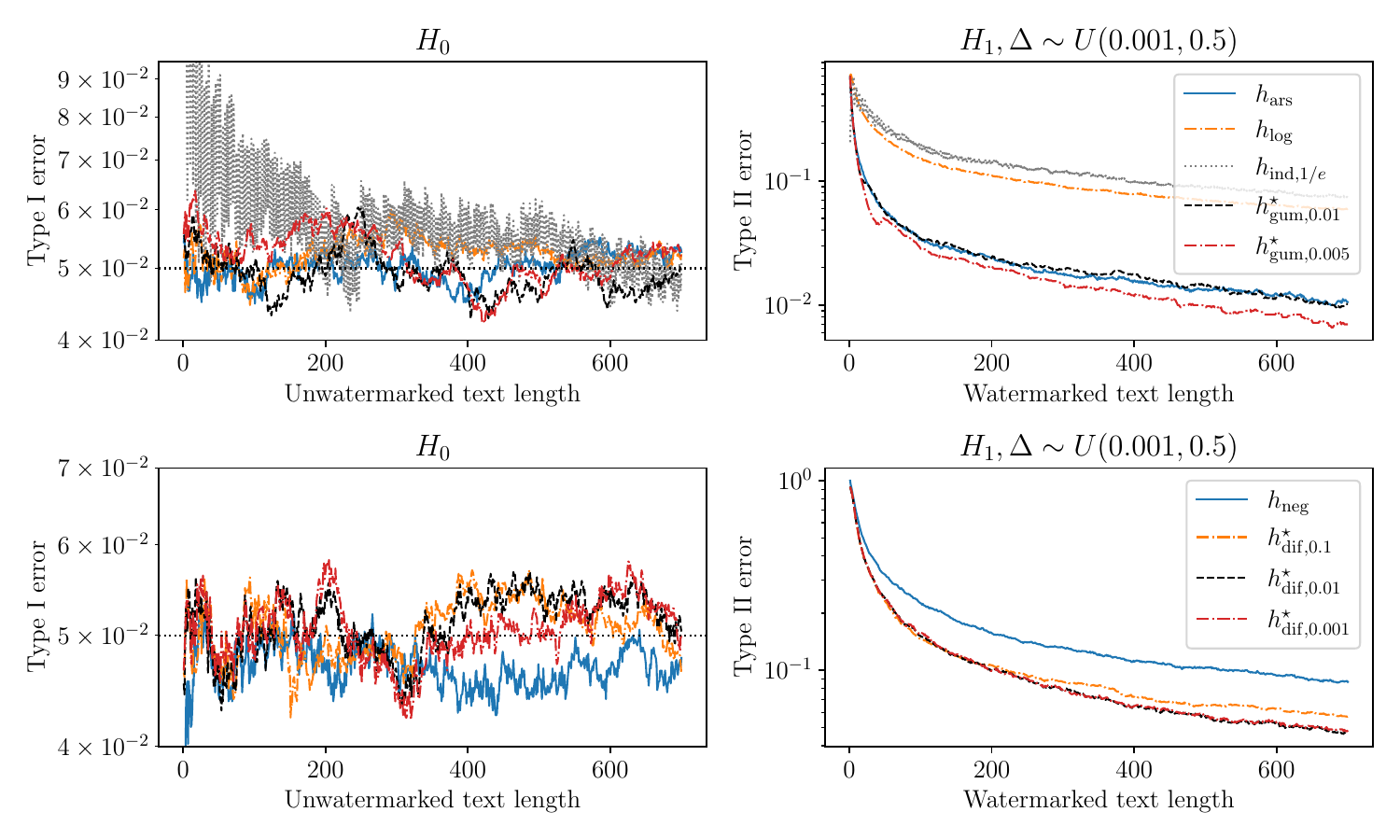}
\caption{
Average Type I and II errors (on a log scale) vs. text length on synthetic datasets for Gumbel-max (top) and inverse transform watermarks (bottom).
}
\label{fig:inhomo-simu}
\end{figure}

Our initial investigation is on Type I error control for finite token sequences, recognizing that 
the guarantees for Type I error control follow from the central limit theorem and hold asymptotically when the sequence length $T \to \infty$. For a given text length $T$, we generate $5,000$ samples of \emph{unwatermarked} word tokens sequences. We uniformly sample each unwatermarked token from the vocabulary $\Voca$ in our experiments. Although different sampling strategies could be employed, due to the pivotal property of our test statistics, we anticipate that they would lead to similar Type I error rates. 

Throughout $5,000$ repeated experiments, we compute the average Type I error, with the findings presented in the leftmost column of Figure~\ref{fig:inhomo-simu}. These results reveal that empirically, the Type I errors generally align with the nominal level of $0.05$, fluctuating between $0.04$ and $0.06$. This performance confirms our theoretical expectations regarding Type I error control across most scenarios. The only exception is the score function $\hind(\Yars)$, achieving close to the nominal $0.05$ level only for token sequences over 300 in length, where asymptotic effects emerge.

Next, we examine the effectiveness of these methods in controlling Type II errors on \emph{watermarked} sequences. Performing this examination involves specifying how each $\bP_t$ is generated under $H_1$.
In our simulation, we model LLM's NTP ($\bP_t$) distributions as spike distributions: we set its largest probability as $P_{t, (1)} = 1-\Delta_t$, where $\Delta_t > 0$ is i.i.d. sampled from $[a,b]$ with $a=10^{-3}$,\footnote{This small number $10^{-3}$ is used to avoid $H_1$ merging with $H_0$.} and $b=0.5$, and uniformly distribute the remaining probabilities so that $P_{t, (k)} = \Delta_t/(|\Voca|-1)$ for $k > 1$. The choice of this setup is based on our theoretical development. We also tested alternative $b$ values of $0.1$, $0.3$, and $0.7$, but given the results are similar, we will include those details in the appendix.

 Throughout $5,000$ repeated experiments, we compute the average Type II error and present the results in the rightmost column of Figure~\ref{fig:inhomo-simu}. 
Designed using the optimality criterion through our framework, both $\hoptars(\Yars)$ for Gumbel-max and $\hoptdif(\Ydif)$ for inverse transform outperform other watermark detection techniques in reducing Type II errors.
 Additionally, empirical performance rankings are consistent with our theoretical prediction in Theorem \ref{thm:main-efficiency-gumbel-informal}: $\hars(\Yars)$ has lower Type II error, followed by $\hlog(\Yars)$, with $\hindo(\Yars)$ having the highest Type II error.

\subsection{Real-World Examples}

\begin{figure}[t!]
\centering
\includegraphics[width=1.0\textwidth]{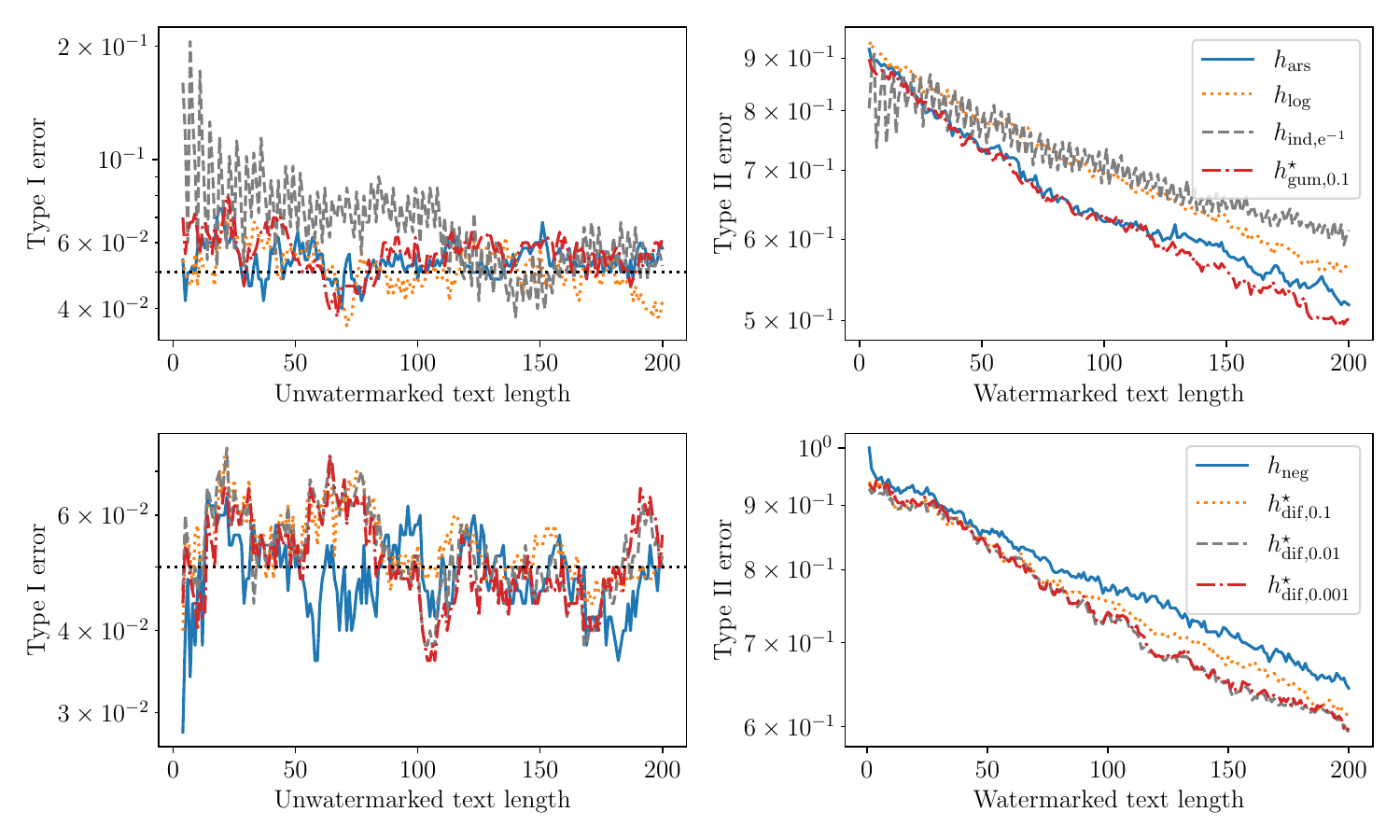}
\caption{
Average Type I and II errors (on a log scale) vs. text length on C4 datasets \citep{raffel2020exploring} and the OPT-1.3B model for Gumbel-max (top) and inverse transform watermarks (bottom).
}
\label{fig:c4-1.3B}
\end{figure}

In this section, we conduct an empirical analysis of watermark detection methods for text sequences from large token generation models such as OPT-1.3B \citep{zhang2022opt} and Sheared-LLaMA-2.7B \citep{xia2023sheared}. Specifically, we evaluate Type I errors on the \emph{unwatermarked} texts directly from large token generation models and Type II errors with texts from the same models, but with watermarks incorporated. This experimental setup essentially follows the approach described in the existing literature \citep{kirchenbauer2023watermark} and is further detailed in Appendix \ref{sec:real-data-detail} for completeness.

We present our numerical results with the OPT-1.3B models in Figure \ref{fig:c4-1.3B} (similar results for Sheared-LLaMA-2.7B are documented in Appendix \ref{appen:extened-results}). Notably, both the average Type I error (leftmost column of Figure \ref{fig:c4-1.3B}) and Type II error (rightmost column of Figure \ref{fig:c4-1.3B}) curves are computed based on $500$ repeated experiments. Though the setup (mainly the way each $\bP_t$ is generated) is slightly different from our simulation studies, the conclusions are quite similar. 

First, in most scenarios, all the detection methods maintain Type I errors within $0.04$ to $0.06$, aligning well with the expected $0.05$ level, particularly in texts over $100$ tokens long, as observed in our numerical experiments. Second, our methods, $\hoptars(\Yars)$ for Gumbel-max and $\hoptdif(\Ydif)$ for inverse transform, excel in minimizing Type II errors over other score functions. For Gumbel-max watermarks, the performance rankings follow our Theorem \ref{thm:main-efficiency-gumbel-informal}'s predictions: $\hars(\Yars)$ records a lower Type II error, then $\hlog(\Yars)$, with $\hindo(\Yars)$ the highest type II error.

\section{Discussion}\label{sec:discussion}

In this paper, we have introduced a statistical framework for reasoning about watermarks for LLMs from a hypothesis testing viewpoint. Our framework offers techniques and concepts for analyzing the asymptotic test efficiency of watermark detection by evaluating Type I and Type II errors of the problem of watermark detection. Specifically, to provide a formal approach to comparing different detection rules, we have introduced class-dependent efficiency, which further reduces the comparison problem to a minimax optimization problem. We have applied this framework to two representative unbiased watermarks, namely, the Gumbel-max watermark and the inverse transform watermark. Our results include the derivation of optimal detection rules, which are shown to have competitive or superior performance compared to existing detection rules.

Our work opens up several research directions in this emerging line of statistical research on watermarks for LLMs. Perhaps the most immediate question is how to choose $\Delta$ for the distribution class $\FPM$ when the knowledge of the underlying token distributions is limited. 
\new{For instance, it would be valuable to explore methods for enhancing detection power when $\Delta$ is known in a distributional sense (see the discussion in Appendix \ref{appen:Delta-selection}).
%In the case where we have a partial knowledge that a $\gamma(\Delta)$-fraction of $\bP_1, \ldots, \bP_n$ belongs to $\FPM$, but with no information about the remaining NTP distributions, we explore the overall efficiency rate in Appendix \ref{appen:Delta-selection}, and investigate its dependence on $\Delta$. 
%We argue that when prior knowledge about $\Delta$ is limited, selecting its value involves a trade-off. Specifically, if we assume a $\gamma(\Delta)$-fraction of $\bP_1, \ldots, \bP_n$ belongs to $\FPM$ and lack information about the remaining NTP distributions, Proposition \ref{thm:inhomo-type-II-general} suggests that $\FPM$-efficiency should be $\gamma(\Delta) R_{\FPM}(h)$. In general, $\gamma(\Delta)$ decreases as $\Delta$ increases, while $R_{\FPM}(h)$ increases with $\Delta$. Consequently, the overall efficiency rate $\gamma(\Delta)R_{\FPM}(h)$ is a complex function of $\Delta$. We explore this further in Appendix \ref{appen:Delta-selection}.
}
Another approach is to extend Proposition~\ref{thm:inhomo-type-II-general} to multiple distribution classes in recognition of the heterogeneity of NTP distributions across the token sequence. This can be done, for example, by adaptively clustering the empirical distributions of $\bP_t$ into several clusters and using these to represent the distribution classes. Moreover, it would be a welcome advance to offer guidelines on how to choose pivotal statistics, which perhaps would allow the framework to maximize its utility. In this paper, we fix the pivotal statistics from the beginning.
\new{
However, we show that it is fundamentally difficult to identify optimal pivotal statistics that maximize our efficiency notion. The challenge lies in the \emph{nonconvex} nature of the set of pivotal statistics, making the selection problem a \emph{nonconvex} optimization problem that is currently beyond our reach  (see Appendix \ref{appen:difficulties}). 
} %However, we present a heuristic argument in Appendix \ref{append:gumbel-pivot}, showing that the pivot $\Yars$ for the Gumbel-max watermark is sufficiently good, as it effectively approximates the most-favorable log-likelihood ratio in an independence test.}
Finally, given that an unbiased watermark is probabilistically similar to sampling from a multinomial distribution, it would be of interest to consider other sampling schemes such as the alias method \cite{walker1974new, walker1977efficient} to develop new categories of watermarks.

More broadly, we wish to remark on several possible extensions of our framework for improving watermark detection and design. Our framework would be enhanced if different watermarks with their own detection rules could be compared on a fair and comprehensive basis. Different watermarks might work best in different regimes and, therefore, directly comparing their class-dependent efficiency rates would be very sensitive to the choice of the distribution class. Another direction to extend our framework is to take distribution classes that are adaptive to NTP distributions in the real-world deployment of LLMs. This adaptivity has the potential to make the minimax formulation \eqref{eq:minmax} better capture the empirical performance. A possible approach is to model the empirical distribution of the NTP distributions $\bP_t$ using Zipf's law \citep{zipf2016human} or other laws. A challenge here, however, might arise from the fact that the distribution class is no longer permutation invariant. Finally, an interesting extension is to have varying score functions across the token sequence, as opposed to having the same $h$ in the test statistic \eqref{eq:Th}. This flexibility might enhance the power of watermark detection by taking into account the heterogeneity of the NTP distributions from token to token in generation.

{\small
\section*{Acknowledgments}
We would like to thank two anonymous referees for their constructive comments that helped improve the presentation of the paper.
This work was supported in part by NIH grants, RF1AG063481 and U01CA274576, NSF DMS-2310679, a Meta Faculty Research Award, and Wharton AI for Business. The content is solely the responsibility of the authors and does not necessarily represent the official views of the NIH.

%\clear
\bibliographystyle{abbrv}
\bibliography{bib/chatgpt,bib/privacy,bib/stat}

}

\appendix

\begin{appendix}
\onecolumn
\begin{center}
{\huge {Supplementary Material}}
\end{center}

This Supplementary Material contains the remaining proofs and technical details.
The proof \new{that supports the general framework} is collected in Section \ref{proof:main}.
The proofs about the Gumbel-max watermark are presented in Section \ref{append:gumbel}.
Section \ref{append:inverse} includes the proofs of results for the inverse transform watermark.
Section \ref{append:experiments} contains experiment details.

\new{
\section{Proof Supporting the General Framework}
\label{proof:main}}

\subsection{Proof of Theorem \ref{thm:inhomo-type-II}}

\begin{proof}[Proof of Theorem \ref{thm:inhomo-type-II}]
We will use the Chernoff bound to analyze the type II error of $T_h$.
Recall that the test introduced by $h$ is
\[
T_h(Y_{1:n}) = 
\left\{ \begin{array}{ll}
1 &~\text{if}~ \sum_{t=1}^n   h(Y_t) \ge \gamma_{n, \alpha }\\
0 &~\text{if}~ \sum_{t=1}^n   h(Y_t) < \gamma_{n, \alpha }.\\
\end{array} \right. 
\]
By the strong law of large numbers, we know that \new{as long as $\alpha \in (0, 1)$},
\new{
\begin{equation}
\label{eq:limit-critical-value}
\limsup_{n \to \infty} \frac{\gamma_{n, \alpha }}{n} = \EB_0 h(Y).
\end{equation}
}
We will prove this equation \eqref{eq:limit-critical-value} at the end of the proof.

Let $P_t$ denote the probability outputted by $\MM$ at iteration $t$.
By the condition that $h$ is regular, we have that for any feasible $\theta \ge 0$, 
\begin{align*}
\EB_{1, \bP_t} \exp(-\theta h(Y_t))
= \exp( \log \phi_{P_t, h}(\theta) )
\le \exp(\log \phi_{\PM, h}(\theta)).
\end{align*}
Using the last inequality, we have that for any feasible $\theta \ge 0$, 
\begin{align}
1 - \EB_{1} T_h(Y_{1:n})
&= \PB_{1}\left( \sum_{t=1}^n -h(Y_t) \ge - \gamma_{n, \alpha}  \right) \nonumber \\ 
&\le \exp(\gamma_{n, \alpha}  \theta)  \cdot \EB_{H_1}\exp\left( \sum_{t=1}^n -\theta h(Y_t)\right) \label{eq:help-inhomo-type-II-1} \\ 
&\le \exp(\gamma_{n, \alpha}  \theta)  \cdot \exp(n \cdot \log \phi_{\PM, h}(\theta)). \nonumber
\end{align}
Therefore,
\begin{align*}
\limsup_{n \to \infty} \left[1 - \EB_{1} T_h(Y_{1:n})\right]^{1/n}
&\le \limsup_{n \to \infty} \inf_{\theta \ge 0}
\exp(\theta\gamma_{n, \alpha}/n )  \cdot \exp( \log \phi_{\PM, h}(\theta)) \\
&\le  \inf_{\theta \ge 0} \limsup_{n \to \infty} \exp(\theta\gamma_{n, \alpha}/n )  \cdot \exp( \log \phi_{\PM, h}(\theta)) \\
&=  \inf_{\theta \ge 0} 
\exp(\theta \EB_0 h(Y)  +\log \phi_{\PM, h}(\theta)).
\end{align*}
\end{proof}

\new{
\begin{proof}[Proof of \eqref{eq:limit-critical-value}]
Under the null hypothesis, $Y_t$ are i.i.d. copies of $\mu_0$. By the strong law of large numbers, it follows that
\[
\frac{1}{n_k}\sum_{t=1}^{n_k} h(Y_t) \to \EB_0 h(Y)
\]
for any diverging subsequence $\{n_k\}_{k \geq 1}$ satisfying $\lim_{k \to \infty} n_k = \infty$. 

By definition, $\gamma_{n, \alpha}$ is the theoretical critical value that makes the detection rule $T_h$ have a given Type I error of $\alpha$, that is,
\[
\PB_{H_0}\left(\sum_{t=1}^n h(Y_t) \ge \gamma_{n, \alpha}\right) = \alpha.
\]

Assume, for contradiction, that $\limsup_{n \to \infty} \frac{\gamma_{n, \alpha}}{n} > \mathbb{E}_0[h(Y)]$. Then we could find a small positive number $\epsilon > 0$ and a diverging subsequence $\{n_k\}_{k \geq 1}$ such that $\frac{\gamma_{n_k, \alpha}}{n_k} \geq \mathbb{E}_0[h(Y)] + \eps$ for sufficiently large $k$. This leads to a contradiction, as follows:
\begin{align*}
\alpha 
&= \lim_{k \to \infty}\PB_{H_0}\left(\sum_{t=1}^{n_k} h(Y_t) \ge \gamma_{n_k, \alpha}\right) \\
&=\lim_{k \to \infty}\PB_{H_0}\left(\frac{1}{n_k}\sum_{t=1}^{n_k} h(Y_t) \ge \frac{\gamma_{n_k, \alpha}}{n_k}\right) \\
&\le \lim_{k \to \infty}\PB_{H_0}\left(\frac{1}{n_k}\sum_{t=1}^{n_k} h(Y_t) \ge \EB_0 h(Y) + \eps\right) \\
&= \PB_{H_0}\left(\EB_0 h(Y) \ge \EB_0 h(Y) + \eps\right)  = 0.
\end{align*}

Similarly, if $\liminf_{n \to \infty} \frac{\gamma_{n, \alpha}}{n} < \mathbb{E}_0[h(Y)]$, we could find a small positive number $\epsilon > 0$ and a diverging subsequence $\{n_k\}_{k \geq 1}$ such that $\frac{\gamma_{n_k, \alpha }}{n_k} \le \EB_0 h (Y) - \eps$ for sufficiently large $k$.
This also leads to a contradiction:
\begin{align*}
\alpha 
&= \lim_{k \to \infty}\PB_{H_0}\left(\sum_{t=1}^{n_k} h(Y_t) \ge \gamma_{n_k, \alpha}\right) \\
&=\lim_{k \to \infty}\PB_{H_0}\left(\frac{1}{n_k}\sum_{t=1}^{n_k} h(Y_t) \ge \frac{\gamma_{n_k, \alpha}}{n_k}\right) \\
&\ge \lim_{k \to \infty}\PB_{H_0}\left(\frac{1}{n_k}\sum_{t=1}^{n_k} h(Y_t) \ge \EB_0 h(Y) - \eps\right) \\
&= \PB_{H_0}\left(\EB_0 h(Y) \ge \EB_0 h(Y) - \eps\right) = 1.
\end{align*}
\end{proof}
}

\begin{rem}
We detail the tightness mentioned in Remark \ref{rem:tight} in the following lemma.

\begin{lem}
\label{lem:tight}
If there exists some $\bP$ in the closure of $\PM$ which maximizes $\phi_{\bP, h}(\theta)$ simultaneously for all $\theta \ge 0$, then \eqref{eq:inequality_Eff} is tight in the sense that there exists $\bP^\star$ in $\PM$ such that, if $\bP_t = \bP^\star$ for all $t$, then
\[
\PB_{H_1}(T_h(Y_{1:n}) = 0) \ge  \mathrm{e}^{-(R_{\PM}(h) + \varepsilon)n}
\]
for any positive $\varepsilon$ and sufficiently large $n$.
\end{lem}

This condition regularizes $\PM, h$ and $\mu_{1, \bP}$ simultaneously.
As we already see, it is satisfied by both the Gumbel-max and inverse transform watermarks (see Lemma \ref{lem:worst-MGF} and \ref{lem:monotone-bar-L} for the details).
It implies our rates are tight in the least-favorable case.

\begin{proof}[Proof of Lemma \ref{lem:tight}]
To demonstrate the tightness of our results, we resort to the following lemma.

\begin{lem}[Exchangeability]
\label{lem:exchangeability}
Let $\mathrm{cl}(\PM)$ denote the closure of $\PM$.
For any $\theta \ge 0$ and $\bP \in \mathrm{cl}(\PM)$, we define
\[
Q(\theta, \bP) = \theta \EB_0 h(Y) + \log \EB_{1, \bP} \re^{-\theta h(Y)}.
\]
If there exists some $\bP^* \in \mathrm{cl}(\PM)$ such that $Q(\theta, \bP^\star) =\sup_{\bP \in \mathrm{cl}(\PM)}  Q(\theta, \bP)$ for any $\theta \ge 0$, then
\[
\sup_{\bP \in \mathrm{cl}(\PM)} \inf_{\theta \ge 0} Q(\theta, \bP)
=\inf_{\theta \ge 0} \sup_{\bP \in \mathrm{cl}(\PM)}  Q(\theta, \bP).
\]
\end{lem}
\begin{proof}[Proof of Lemma \ref{lem:exchangeability}]
First of all, by definition, we have
\[
\sup_{\bP \in \mathrm{cl}(\PM)} \inf_{\theta \ge 0} Q(\theta, \bP)
\le \inf_{\theta \ge 0} \sup_{\bP \in \mathrm{cl}(\PM)}  Q(\theta, \bP).
\]
We then focus on the other direction.
Because there exists some $\bP^* \in \mathrm{cl}(\PM)$ such that $Q(\theta, \bP^\star) =\sup_{\bP \in \mathrm{cl}(\PM)}  Q(\theta, \bP)$ for any $\theta \ge 0$, 
\begin{align*}
\sup_{\bP \in \mathrm{cl}(\PM)} \inf_{\theta \ge 0} Q(\theta, \bP)
\ge \inf_{\theta \ge 0} Q(\theta, \bP^\star) = \inf_{\theta \ge 0} \sup_{\bP \in \mathrm{cl}(\PM)}  Q(\theta, \bP).
\end{align*}
\end{proof}

By Lemma \ref{lem:exchangeability}, given that 
\[
R_{\PM}(h) = -\inf_{\theta \ge 0} \sup_{\bP \in \mathrm{cl}(\PM)}  Q(\theta, \bP),
\]
we know that $R_{\PM}(h)$ can also be written as
\[
R_{\PM}(h)=-\sup_{\bP \in \mathrm{cl}(\PM)} \inf_{\theta \ge 0} Q(\theta, \bP).
\]
As a result, we can find a sequence of $\bP_k \subset \PM$ such that 
\[
\inf_{\theta \ge 0} Q(\theta, \bP_k) \to \inf_{\theta \ge 0} \sup_{\bP \in \mathrm{cl}(\PM)}  Q(\theta, \bP)
\]
as $k \to \infty$; that is, $R_{\bP_k}(h) \to R_{\PM}(h)$ as $k \to \infty$.

By standard i.i.d. results from the large deviation theory \citep{dembo2009large,van2000asymptotic}, if we set all the NTP distribution as $\bP_k$, then
\begin{equation}
\label{eq:limit_lower}
\lim_{n\to\infty}\PB_{H_1}(T_h(Y_{1:n}) = 0)^{1/n}=\mathrm{e}^{-R_{\bP_k}(h)}
\ge\mathrm{e}^{-(R_{\PM}(h)+\varepsilon_k^{(1)})},
\end{equation}
where $\varepsilon_k^{(1)} \to 0$ as $k \to \infty$.
It then follows that
\[
\PB_{H_1}(T_h(Y_{1:n}) = 0)\ge\mathrm{e}^{-n(R_{\PM}(h)+\varepsilon_k^{(1)}+\varepsilon_{n}^{(2)})},
\]
where $\varepsilon_{n}^{(2)}$ denotes a sequence of positive numbers approaching to zero as $n\to\infty$ to ensure \eqref{eq:limit_lower} to hold. We conclude that the lower bound holds by choosing a sufficiently large $k$.
\end{proof}

\end{rem}

\subsection{Proof of Proposition~\ref{thm:inhomo-type-II-general}}

In this subsection, we present the proof of Proposition~\ref{thm:inhomo-type-II-general} and establish the equation that $R_{\Simplex(\Voca)}(h)=0$.

\begin{proof}[Proof of Proposition~\ref{thm:inhomo-type-II-general}]
We start from the inequality \eqref{eq:help-inhomo-type-II-1}.
Since $\gamma$-fraction of $\bP_1, \ldots, \bP_n$ belongs to $\PM_1$ while the rest belongs to $\PM_2$, it follows that
\[
1 - \EB_{1} T_h(Y_{1:n})
\le \exp(\gamma_{n, \alpha}  \theta)  \cdot 
\exp(\gamma n \cdot \log \phi_{\PM_1, h}(\theta))+
(1-\gamma) n \cdot \log \phi_{\PM_2, h}(\theta)).
\]
Hence,
\begin{align*}
  \limsup_{n \to \infty} \left[1 - \EB_{1} T_h(Y_{1:n})\right]^{1/n}
&\le  \inf_{\theta \in [0, \theta_{\PM, h})} 
\exp(\theta \EB_0 h(Y)  +\gamma \log \phi_{\PM_1, h}(\theta) + (1-\gamma)\log \phi_{\PM_2, h}(\theta) )  \\
&\le \exp(-[\gamma R_{\PM_1, h} +(1-\gamma) R_{\PM_2, h}]).
\end{align*}
\end{proof}

\begin{lem}
\label{lem:no-efficiency}
If $\mu_0 \in \{\mu_{1, P}: P \in \Simplex(\Voca) \}$,
\[
R_{\Simplex(\Voca)}(h)=0.
\]
\end{lem}

\begin{proof}[Proof of Lemma \ref{lem:no-efficiency}]

On one hand, we always have $R_{\PM}(h) \ge 0$ for any $\PM \subset \Simplex(\Voca)$.
The reason is given as follows.
Recall that $R_{\PM}(h) = -\inf\limits_{\theta \ge 0} f(\theta)$ where $f(\theta) := \theta  \EB_0 h(Y) +  \log  \phi_{\PM, h}(\theta)$.
Clearly, $f(\theta)$ is convex in $\theta$ with $f(0)=0$.
Hence, we always have
\[
R_{\PM, h} \ge -f(0) = 0.
\]
On the other hand, we have $R_{ \Simplex(\Voca), h} \le 0$.
The reason is given as follows.
Note that $R_{\PM, h}$ is monotone in $\PM$: if $\PM_1 \subset \PM_2$, then $R_{\PM_2, h} \le R_{\PM_1, h}$.
If $\mu_0 \in \{\mu_{1, P}: P \in \Simplex(\Voca) \}$, without loss of generality, we assume $\mu_0 = \mu_{1, \bP_0}$.
Since we must have $\{\bP_0\} \subset \Simplex(\Voca)$, it follows that
\[
R_{ \Simplex(\Voca), h} \le R_{\bP_0, h} = 0.
\]
\end{proof}

\subsection{Difficulties of Finding the Optimal Pivot Statistic}
\label{appen:difficulties}

\new{
At the end of this section, we illustrate that searching for the optimal pivotal statistic is not generally a convex problem.
Overcoming these challenges would be significant and could be addressed in future research endeavors.

First and foremost, from an optimization perspective, the problem of finding the optimal pivotal statistic is in general \emph{non-convex}. We will illustrate this \emph{non-convexity} through concrete examples.

To set up the stage, we introduce some notation. 
For a given pivotal statistic $Y$, we first denote $J(Y) = \sup_{h} R_{\PM}(h)$ to be the $\PM$-efficiency for that pivotal statistic. 
We consider the simplest case where $\PM$ contains only a single element (so that we can exchange the order of $\sup_{\bP \in \PM}$ and $\inf_h$).
We deduce that
\begin{align*}
J(Y) = \sup_{h} R_{\PM}(h)  
&= -\inf_{h} \sup_{\bP \in \PM}\left(\EB_0 h(Y) +  \log \EB_{1, \bP}\mathrm{e}^{-h(Y)} \right)\\
&=- \sup_{\bP \in \PM}\inf_{h}\left(  \EB_0 h(Y) + \log \EB_{1, \bP}\mathrm{e}^{-h(Y)} \right)\\
&= \inf_{\bP \in \PM}\sup_{h}\left(-\theta  \EB_0 h(Y) - \log \EB_{1, \bP}\mathrm{e}^{-h(Y)} \right) \\
&=\inf_{\bP \in \PM} \KL(\mu_{0}(Y)\|\mu_{1, \bP}(Y)).
\end{align*}
The problem of finding the optimal pivotal statistic can be formally written down as: 
\begin{equation}
\label{eqn:optimizing-the-pivotal-statistics}
    {\rm maximize}~~ J(Y)~~~\text{subject to}~~~\text{$Y = Y(\token, \xi)$ is a pivotal statistics}.
\end{equation}
Here, we recall the notation that $\token$ refers to the token, and $\xi$ refers to the pseudorandom numbers. We give examples to demonstrate that (i) the functional $J$ is neither convex nor concave, and (ii) the constraint set consisting of all $Y$ that form a pivotal statistic is not convex.

Consider the Gumbel-max watermark with two tokens, say, $\Voca = \{1, 2\}$ (our examples extend to the general case $|\Voca| \ge 2$ by restricting the support of $\bP$ to two).
We denote the pseudorandom number $\zeta = (U_1, U_2)$ where $U_1, U_2 \sim \UM(0, 1)$. Our examples take place in the simplest setting where $\PM = \{\bP\}$ is a singleton. In this setting, $J(Y)$ satisfies:  
\[
J(Y) = \KL(\mu_{0}(Y)\|\mu_{1, \bP}(Y)).
\]

Our first proposition states that \( J(Y) \) is not necessarily convex or concave. 
\begin{prop}[The objective function is neither convex nor concave]
\label{prop:nonconvex}
There is a singleton set $\PM = \{\bP\}$ with \( Y_1, Y_2, Y_3, Y_4 \), each of which is a function of $(\token, \zeta)$, such that
\begin{gather*}
J\left( \frac{Y_1+Y_2}{2}\right) < \frac{J(Y_1)+J(Y_2)}{2}
~~\text{and}~~
J\left( \frac{Y_3+Y_4}{2}\right) > \frac{J(Y_3)+J(Y_4)}{2}.
\end{gather*}
\end{prop}

Our second proposition demonstrates that the set of pivot statistics is not a convex set.
\begin{prop}[The constraint is not convex]
\label{prop:nonconvex-constraint}
There exist two pivotal statistics, \( Y_1 \) and \( Y_2 \), such that \( \frac{Y_1 + Y_2}{2} \) is not a pivotal statistic.
\end{prop}

These two factors contribute to the significant challenges in solving the underlying optimization problem.
We provide the proofs for the above two propositions in the following.

\begin{proof}[Proof of Proposition \ref{prop:nonconvex}]
We introduce the following functions and parameters:
\begin{gather*}
Y_1 = U_{\token}, Y_2 = U_1, Y_3 = U_{\token} + U_1, Y_4 = |U_{\token}-U_1|,~\text{and}~P_1=P_2=1/2.
\end{gather*}
Table \ref{tab:densities} presents some important calculation results that complete the proof.
\end{proof}

\begin{table}[h!]
\centering
\resizebox{\columnwidth}{!}{
 \begin{tabular}{c| c c c} 
 \hline
 Function $Y$ & $F_0(r)$ & $F_1(r)$ & $J(Y) = \KL(F_0\|F_1) $ \\ [0.5ex] 
 \hline\hline
 $Y_1 = U_{\token}$ & $r$ & $r^2$ & $1-\log 2 \approx 0.306$ \\ [0.5ex] 
 \hline
 $Y_2 = U_1$ & $r$ & $r$ & 0 \\[0.5ex] 
 \hline
 $\frac{Y_1+Y_2}{2}= \frac{U_{\token} + U_1}{2} $ & \multirow{2}{*}{$\begin{cases}
     \frac{r}{2} + r^2~\text{if}~r \in [0, 1/2] \\
    \frac{5r-1}{2} - r^2~\text{if}~ r \in [1/2, 1]
 \end{cases}$} & 
\multirow{2}{*}{$\begin{cases}
     3r~\text{if}~r \in [0, 1/2] \\
   2-r~\text{if}~ r \in [1/2, 1]
 \end{cases}$} &  
 \multirow{2}{*}{$\approx 0.0450295$} \\[0.5ex] 
 or $Y_3 = U_{\token} + U_1$ &  &  & 
  \\[0.5ex] 
 \hline
 $Y_4 = |U_{\token}-U_1|$ & $1-\frac{(1-r)^2}{2}$ & $1-\frac{(1-r)^2}{2}$ & 0 \\[0.5ex] 
 \hline
  $\frac{Y_3+Y_4}{2} = \max\{U_{\token}, U_1\}$ & $\frac{r+r^2}{2}$ & $r^2$ & $\approx 0.0996444$ \\ [0.5ex] 
 \hline
 \end{tabular}}
 \caption{CDFs under different hypotheses and the efficiency rates of the considered functions. \( F_i \) represents the CDF under hypothesis \( H_i \). The domain for all the CDFs is \( r \in [0, 1] \). Here $\token=1$ if and only if $\frac{\log U_1}{P_1} \ge \frac{\log U_2}{P_2}$. Recall that we use $P_1=P_2=1/2$ in the proof of Proposition \ref{prop:nonconvex}.}
 \label{tab:densities}
\end{table}

\begin{proof}[Proof of Proposition \ref{prop:nonconvex-constraint}]
Define $Y_1 = U_{\token}$ and $Y_2 = U_1$.
It is clear that $Y_1$ and $Y_2$ are pivotal statistics, as their distributions are independent of the distribution of $\token$ under the null hypothesis $H_0$. In the following, we aim to compute the null CDF for $\frac{Y_1+Y_2}{2}$.
Under $H_0$, we have $w \sim \bP = (P_1, P_2)$ and $w \perp (U_1, U_2)$.
As a result, it follows that for any $r \in [0, 1]$, 
\begin{align*}
\PB_{H_0}\left(\frac{Y_1+Y_2}{2} \le r\right)
&= \PB_{H_0}\left(\frac{U_w+U_1}{2} \le r\right)\\
&= P_1 \cdot \PB(U_1 \le r) + P_2 \cdot \PB(U_1+U_2 \le 2r)\\
&= P_1\cdot r + P_2 \cdot \begin{cases}
2r^2~\text{if}~ r \in [0, 1/2],\\
1-2(1-r)^2~\text{if}~ r \in [1/2, 1].
\end{cases}
\end{align*}
We conclude that $\frac{Y_1+Y_2}{2}$ is not a pivotal statistic because its null CDF depends on the distribution of $\token$, that is, $\bP=(P_1, P_2)$. 
\end{proof}
}

\section{Proof for Gumbel-max Watermarks}
\label{append:gumbel}
\subsection{Analysis for the pivotal statistic \texorpdfstring{$\Yars$}{Yars}
}
% \subsection{Proof of Lemma~\ref{lem:dis-r}}
\label{append:gumbel-pivot}
The distribution of $\Yars_t$ under the alternative hypothesis $H_1$ has been studied by some existing works \citep{piet2023mark,fernandez2023three}. 
We include a proof for completeness.

\begin{proof}[Proof of Lemma~\ref{lem:dis-r}]
We assume $\xi = (U_\token)_{\token \in \Voca}$ with each $U_{\token}$ i.i.d. from $\UM(0, 1)$.
We split the probability $\PB(U_{\token_t} \le r)$ into the sum of probability of $|\Voca|$ disjoint events as follows:
\[
\PB(U_{\token_t} \le r) = \sum_{\token \in \Voca}\PB(U_{\token_t} \le r,\token_t = \token)=\sum_{\token \in \Voca}\PB(U_{\token} \le r,\token_t = \token).
\]
By definition in~\eqref{eq:it}, we have that 
\begin{align*}
\PB(U_{\token} \le r,\token_t = \token)
&=\PB(U_{\token} \le r, \token =\arg\max_{j \in \Voca}  U_{j}^{1/P_{t,j}} ) \\
&=\PB(U_{\token} \le r,  U_{j}\le  U_{\token}^{P_{t,j}/P_{t,\token}}~\text{for any}~j \neq \token ) \\
&= \int_0^r \rd u_{\token} \int_{[0, 1]^{|\Voca|-1}} \1_{\{ u_{j} \le u_{\token}^{P_{t,j}/P_{t,\token}}~\text{for any}~j \neq \token \}} \rd u_1 \ldots \rd u_{\token-1} \rd u_{\token+1} \ldots \rd u_{|\Voca|} \\
&=\int_0^r  \prod_{j \neq \token} u_{\token}^{P_{t,j}/P_{t,\token}} \rd u_\token  \\
&=\int_0^r  u_{\token}^{1/P_{t,\token}-1} \rd u_k = P_{t,\token} \cdot r^{1/P_{t,\token}},
\end{align*}
where the third equality follows from the independence of $U_{1}, \ldots, U_{|\Voca|}$.
\end{proof}

\new{At the end of this subsection, we argue that the pivotal statistic $ \Yars_t = Y(\token_t, \xi_t)=U_{t, \token_t}$ in our manuscript is a good pivotal statistic. Our justification is that it
approximates well the most optimistic likelihood ratio test statistic for an independent hypothesis test designed for watermark detection.

Suppose we have one generated token ($n=1$). Fix $t$ (and we drop the notation dependence on $t$). Consider the independence test regarding the token $\token$ and the pseudorandom variable $\xi$. Under $H_0$, $ \token$ and $\xi $ are independent. Under $H_1$, $ \token$ and $\xi $ are related by the decoder function $\SM$. More precisely, we have the hypothesis testing problem:
\[
H_0:~\token \perp \xi
~~\text{versus}~~
H_1:~\token = \SM(\bP, \xi)~\text{for some $\bP$}.
\]
This independent test is designed for watermark detection. 
We now compute the likelihood ratio test statistic. Let $\rho$ denote the probability density of $\xi$. 
\begin{itemize}
\item Under $H_0$, the joint probability density of $(\token, \xi)$ is given by $ P_{\token} \cdot \rho(\xi) $. 
\item Under $H_1$, the joint probability density of $(\token, \xi)$ is given by $\1_{\token = \SM(\bP, \xi)} \rho(\xi)$. 
\end{itemize}
The most optimistic log-likelihood ratio test statistic for this independent test is given by:
\[
\Lambda(\token, \xi) = \sup_{\bP} \log \frac{\1_{\token = \SM(\bP, \xi)} \rho(\xi)}{P_{\token} \cdot \rho(\xi)} =\sup_{\bP} \log \frac{\1_{\token = \SM(\bP, \xi)}}{P_{\token}}.
\]
We refer to \(\Lambda(\token, \xi)\) as the \textit{most optimistic} log-likelihood ratio test statistic because we take the supremum over all possible \(\bP\) to eliminate its dependence. Generally, we would reject \(H_0\) if the log-likelihood ratio is sufficiently large. By using this most optimistic log-likelihood ratio, we adopt a more liberal approach when deciding to reject the null hypothesis.

We compute $\Lambda(\token, \xi)$ for the Gumbel-max watermark. Recall the notation $\zeta = (U_\token)_{\token \in \Voca}$. We note that $\token = \SMmax(\bP, \xi)$ holds if and only if:
\[
\frac{\log \frac{1}{U_\token}}{P_{\token}} \le \frac{\log \frac{1}{U_k}}{P_{k}} \quad \text{for any} \quad k \in \Voca.
\]

As a result, if $\token = \SMmax(\bP, \xi)$ holds, then:
\[
1 = \sum_{k \in \Voca} P_k \le \sum_{k \in \Voca} \frac{\log \frac{1}{U_k}}{\log \frac{1}{U_\token}} P_\token \implies P_{\token} \ge \frac{\log \frac{1}{U_\token}}{\sum_{k \in \Voca} \log \frac{1}{U_k}}.
\]
Using this relation, we deduce that 
\[
\Lambda(\token, \xi) = \log \frac{\sum_{k \in \Voca} \log \frac{1}{U_k}}{\log \frac{1}{U_{\token}}}.
\]
We argue that in the regime where $|\Voca|$ is large, this likelihood ratio statistic can be effectively approximated by
\[
\Lambda(\token, \xi) \approx \log \frac{|\Voca|}{\log \frac{1}{U_\token}} %=: \widetilde{\Lambda}(\token, \xi).
\]
as there is the law of large numbers where $\frac{1}{|\Voca|} \sum_{k \in \Voca} \log \frac{1}{U_k} \to 1$ for i.i.d. $U_i \sim U(0, 1)$.
Note, on the RHS, $\log \frac{|\Voca|}{\log \frac{1}{U_\token}}$ is an increasing function of $U_{\token}$, which is exactly the current pivotal statistic for the Gumbel-max watermark.

To conclude, we have shown that the pivotal statistic $\Yars_t$ in our manuscript, up to a monotone transform, approximates well the most optimistic likelihood ratio test statistic for the independent hypothesis test.  This supports the effectiveness of $\Yars_t$ as a pivotal statistic for watermark detection. Thus, one contribution of our work is identifying the optimal transform $h$ based on $\Yars_t$ that achieves the highest possible statistical efficiency. 
}

\subsection{Proof of Theorem~\ref{prop:unique}}
\label{proof:unique}
\begin{proof}[Proof of Theorem~\ref{prop:unique}]
For every $P\in [0, 1]$, let us define  
\begin{equation*}
G_P(t) = \PB(V(P, U) \le t).
\end{equation*}
Since $\max\{V(P_1, U_1), V(P_2, U_2)\} 
\stackrel{d}{=} V(P_1+P_2, U)$ holds, this then implies 
for every $t \in \R$: 
\begin{equation}\label{eq:mono}
\begin{split}
G_{P_1+P_2}(t) &= \PB(V(P_1+P_2, U) \le t) \\
&=  \PB(\max\{V(P_1, U_1), V(P_2, U_2)\} \le t) = G_{P_1}(t) G_{P_2}(t)
\end{split}
\end{equation}
must hold for every non-negative $P_1, P_2$ with $P_1 + P_2 \le 1$.

Hence, if we define $h_t(P) = -\log G_P(t)$, then $h_t$ is a solution to 
the following variant of Cauchy's functional equation 
\begin{equation*}
f(P_1+ P_2) = f(P_1) + f(P_2)~~\forall P_1, P_2 \ge 0, P_1 + P_2 \le 1.
\end{equation*}
Furthermore, by \eqref{eq:mono}, for any $0\le P_1<P_2\le 1$, 
\[
G_{P_2}(t)=G_{P_1}(t)G_{P_2-P_1}(t)\ge G_{P_1}(t),
\]
and thus $h_t(P)$ is monotone in $P$.
Fix $t > 0$. 
By Lemma~\ref{lem:monotone} and using the fact that $h_t$ is measurable, we obtain that for every $P > 0$
\begin{equation*}
h_t(P) = \nu(t) P
\end{equation*}
must hold for some $\nu(t) \in \overline{\R}$. This means 
for every $P \in (0, 1]$ and $t \in \R$ 
\begin{equation*}
\PB(V(P, U) \le t)= G_P(t) = \re^{\nu(t) P}.
\end{equation*}
Since $t \mapsto G_P(t)$ is monotonically increasing and right continuous, 
$t \mapsto \nu(t)$ must be increasing and right continuous. Let 
us introduce 
\begin{equation*}
\nu^{-1}(t) = \sup\{y: \nu(y) < t\}.
\end{equation*}
Then $\nu^{-1}$ is increasing. Using the fact that
$\nu^{-1} (t) \le x$ is equivalent to $t \le \nu(x)$, the random variable
$\nu^{-1}(\log U/P)$
has its distribution function given by $G_P$, meaning that for $P \in (0, 1]$
\begin{equation*}
V(P, U) \stackrel{d}{=} \nu^{-1}(\log U/P).
\end{equation*}

Lemma~\ref{lem:strict} shows that the function $\nu^{-1}$ must be strictly increasing in order for
\begin{equation*}
\PB(\nu^{-1}(\log U_1/P_1) \le \nu^{-1}(\log U_2/P_2)) = P_1/(P_1+P_2)
\end{equation*}
to hold. Note also $V(0, U) = -\infty$ must hold almost surely. 
The conclusion is, by picking $g$ with $g(z) = \nu^{-1}(z)$ when
$z > -\infty$ and $g(-\infty) = -\infty$, 
$V(P, U) \stackrel{d}{=} g(\log U/P)$ holds for every $P\in [0, 1]$.
\
\end{proof}

\begin{lem}
\label{lem:monotone}
If a monotone function $f: [0, 1] \to \R$ satisfies that $f(P_1+P_2)=f(P_1)+f(P_2)$ for any $P_1,P_2\ge0$, $P_1+P_2\le1$, then for any $P\in [0, 1]$,
\[
f(P) = Pf(1).
\]
\end{lem}
\begin{proof}[Proof of Lemma~\ref{lem:monotone}]
For any $P_1,P_2\ge0$, $P_1+P_2\le1$, $f(P_1+P_2)=f(P_1)+f(P_2)$ implies that
\[
f(0)=0,\quad f\biggl(\frac{j}{k}\biggr)=\frac{j}{k}f(1)
\]
for any $j,k\in\mathbb{N}$ and $j\le k$. This gives that $f(p)=pf(1)$ for any rational number $p\in[0,1]$.  Suppose that $f(r_0)=r_0 f(1)$ is not true for some irrational $r_0 \in [0, 1]$. Then we have either $f(r_0)>\eta f(1)$ or $f(r_0)<r_0 f(1)$. In the former case we can find rational numbers $r_1, r_2$ such that $r_1<r_0<r_2<f(r_0)/f(1)$ and consequently
\[
f(r_1)=r_1f(1)<r_0 f(1)<f(r_0)\quad\text{and}\quad f(r_2)=r_2f(1)<f(r_0).
\]
It immediately follows that $\{f(r_0)-f(r_1)\}\{f(r_0)-f(r_2)\}>0$, which is, however, contradictory to the monotonicity of $f$. The latter case of $f(r_0)<r_0 f(1)$ can be proved using similar arguments. We then conclude that $f(r_0)=r_0 f(1)$ holds for all irrational number $r_0 \in [0, 1]$, and hence $f(p)=pf(1)$ for all $p\in[0,1]$. 
\end{proof}

\begin{lem}
\label{lem:strict}
Assume $\nu: \R \to \overline{\R}$ is non-decreasing, right continuous, and satisfies $\PB(\nu^{-1}(\log U_1/P_1) \le \nu^{-1}(\log U_2/P_2)) = P_1/(P_1+P_2)$ for $U_1, U_2 \overset{i.i.d.}{\sim} \UM(0, 1)$ and any $ P_1, P_2 > 2$.
Then $\nu^{-1}$ is strictly increasing, that is, we have $\nu^{-1}(t_1)<\nu^{-1}(t_2)$ for any $t_1 < t_2$.
\end{lem}
\begin{proof}[Proof of Lemma~\ref{lem:strict}]
If $\nu^{-1}$ is not strictly increasing, then there exists some $t_1 < t_2$ such that $\nu^{-1}(t_1) = \nu^{-1}(t_2)$.
%	We denote that value by $q = \nu(t_1)$.
%	Then we have that $\nu^{-1}(q) = t_2$ and $\nu^{-1}(q-) = t_1$.
Hence, we have that
\begin{align*}
&\PB\left(\nu^{-1}\left( \frac{\log U_1}{P_1} \right) \le \nu^{-1}
\left( \frac{\log U_2}{P_2} \right) \right) \\
&\quad=\PB\left( \frac{\log U_1}{P_1} \le \frac{\log U_2}{P_2} \right)+ \PB\left(\nu^{-1}\left( \frac{\log U_1}{P_1} \right) \le \nu^{-1}
\left( \frac{\log U_2}{P_2} \right), \frac{\log U_1}{P_1} > \frac{\log U_2}{P_2}\right) \\
&\quad\ge \PB\left( \frac{\log U_1}{P_1} \le \frac{\log U_2}{P_2} \right)
+ \PB\left( t_2\ge \frac{\log U_1}{P_1} > \frac{\log U_2}{P_2} \ge t_1 \right)> \frac{P_1}{P_1+P_2},
\end{align*}
which is contradictory to the condition that $\PB\left(\nu^{-1}\left( \frac{\log U_1}{P_1} \right) \le \nu^{-1}
\left( \frac{\log U_2}{P_2} \right) \right) = \frac{P_1}{P_1+P_2}$.
\end{proof}

\subsection{Analysis of Different Score Functions}
In this subsection, we analyze different score functions and provide all the necessary information for computing the vanilla efficiency rate $R_{\bP}(h)$.
We consider four score functions: the log-likelihood ratio $\hllr$, Aaronson's score function $\hars$, log function $\hlog$, and the indicator function $(\delta=0.5)$.
Recall that from Lemma \ref{lem:dis-r}, $f_{\bP}$ is the PDF of $\Yars_t$ if the underlying NTP distribution is $\bP_t = \bP$.

\begin{rem}
\label{rem:hlog}
$\hllr$ is the log-likelihood ratio and $\hlog$ can be viewed as a relaxed log-likelihood ratio.
This is because $\hlog$ can be derived by applying Jensen's inequality to lower bound $\hllr$. 
Indeed, it follows that for any $r \in [0, 1]$,
\[
\log f_{\bP}(r) = \log \sum_{\token\in \Voca} r^{1/P_{\token}-1} \ge \sum_{\token\in \Voca } P_{\token} \log \left( \frac{r^{1/P_{\token}-1}}{P_{\token}} \right)= (|\Voca|-1) \log r + \Ent(\bP).
\]
If we replace $\log f_{1,\bP}(r)$ with $(|\Voca|-1) \log r + \Ent(\bP)$ and ensure the resulting test is of $\alpha$-level, we obtain the test introduced by $\hlog$.
Hence, $\hlog$ is the probability-agnostic relaxation of $\hllr$.
\end{rem}

\paragraph*{A simple case study}

To visualize the differences among the tests induced by these $h$'s, we examine a two-sample two-token case (that is, $n=2$ and $|\Voca|=2$).
The left panel in Figure~\ref{fig:efficiency} depicts the $\alpha$-level rejection region for different $h$'s.
The rejection region defined by $\hllr, \hlog, \hind$ exhibits a more concentrated pattern around the corner $(1, 1)$, in contrast to the elongated and narrow shape defined by $\hars$.
Furthermore, the rejection regions of $\hllr$ and $\hlog$ share almost the same shape.

\begin{figure}[t!]
\centering
\includegraphics[width=1.0\textwidth]{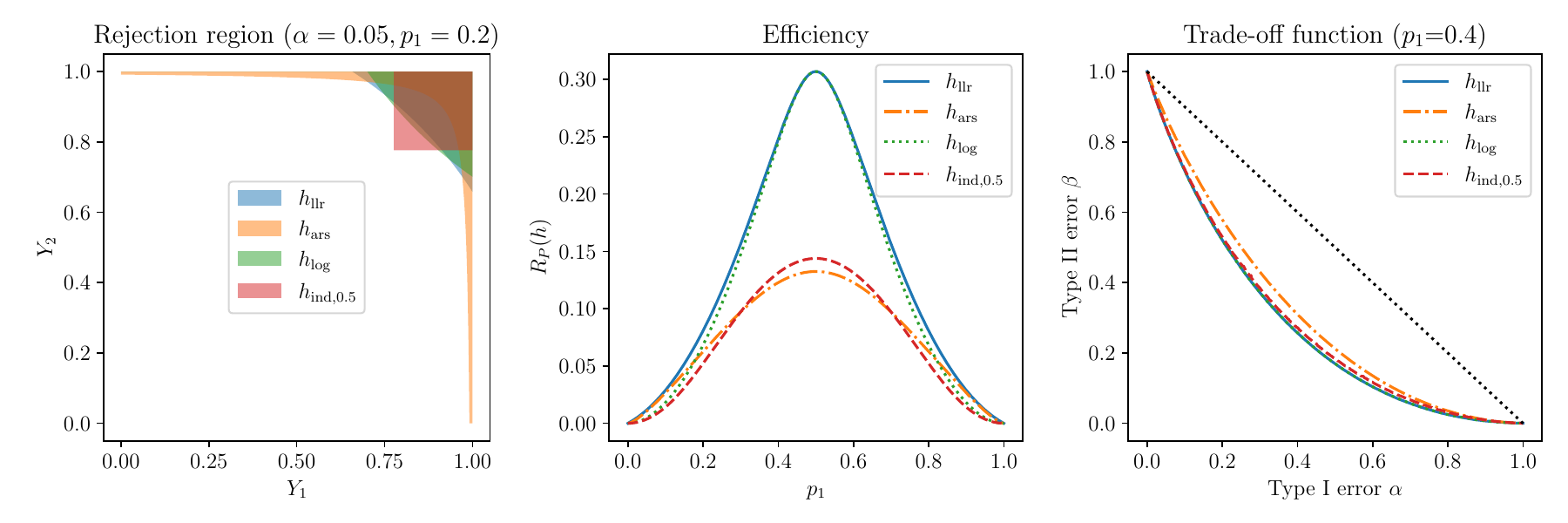}
\caption{The rejection region (left), test efficiency (middle), and trade-off function (right) of different score functions when $|\Voca|=2$ and $n=2$.
}
\label{fig:efficiency}
\end{figure}

\begin{table}[t!]
\centering
\begin{tabularx}{\columnwidth}{c c c  c  } 
\toprule
Score function $h$& $\EB_0 h(Y)$&  $\phi_{\bP, h}(\theta)$& $R_{\bP}(h)$ \\ [0.5ex] 
\midrule
$\hllr(y) = \log f_{\bP}(y)$ & $-\KL(\mu_0, \mu_{1,\bP})$   &  $\int_0^1 f_{\bP}^{1-\theta}(y) \rd y $ & $\KL(\mu_0, \mu_{1,\bP})$ \\
$\hars(y) = -\log(1-y)$ & 1   &  $ \sum_{\token \in \Voca} B(1/P_{\token}, \theta+1)$ & $-\inf\limits_{\theta \ge 0} \left[ \theta + \log \phi_{\bP,\hars}(\theta) \right]$ \\
$\hlog(y) = \log y$ & $-1$   & 
$\sum_{\token \in \Voca} \frac{P_{\token}}{1-P_{\token} \theta}$
& $\sup\limits_{\theta \in [0, \frac{1}{P_{(1)}} )} [\theta - \log \phi_{\bP,\hlog}(\theta)]$ \\
$\hind(y) = \1_{\{y \ge \delta\}}$ & $1-\delta$  &  $F(\delta) + \re^{-\theta}(1-F(\delta))$ & $ \delta \log \frac{\delta}{F(\delta)} + (1-\delta) \log \frac{1-\delta}{1-F(\delta)}$ \\
\bottomrule
\end{tabularx}
\caption{The efficiency component $R_{\bP}(h)$ of different function $h$'s. Here $B(x, y) := \int_0^1 r^{x-1}(1-r)^{y-1} \rd r$ is the Beta function.}
\label{tab:homo-efficiency}
\end{table}

\paragraph*{Formula to compute MGFs}
The following proposition provides all necessary information for computing $R_{\bP}(h)$ which we summarize in Table \ref{tab:homo-efficiency}.
Note that they are all non-decreasing, so Lemma \ref{lem:worst-MGF} could apply.

\begin{prop}
\label{prop:efficiency}
\begin{itemize}
\item For $\hllr(y) = \log f_{\bP}(y)$, $-\EB_0 \hllr(Y) = R_{\bP}(\hllr) = \KL(\mu_0, \mu_{1, \bP})$.
\item For $\hars(y) = - \log(1-y)$, $\EB_0 \hars(Y) = 1$ and $\phi_{\bP, \hars}(\theta) = \sum_{\token \in \Voca} B(1/P_{\token}, \theta+1)$ where $B(a, b)=\int_0^1 y^{a-1}(1-y)^{b-1} \rd y$ is the beta function. $\phi_{\bP, \hars}(\theta)$ is well-defined on $[0, \infty)$ and is strictly decreasing with $\phi_{\bP, \hars}(0) = \infty$ and $\phi_{\bP}(\hars)(\infty) = 0$.
Hence, $R_{\bP}(\hars) =-\theta_1^{\star} - \log \phi_{\hars}(\theta_1^{\star} )$ where $\theta_1^{\star} $ is the unique solution of the equation $\phi_{\bP,\hars}(\theta)  + \phi_{\hars}'(\theta) =0$.
\item For $\hlog(y) = \log y$, $\EB_0 \hlog(Y) = -1$ and $\phi_{\bP, \hlog}(\theta) = \sum_{\token \in \Voca} \frac{P_{\token}}{1-P_{\token} \theta}$ which is well-defined on $[0, 1/P_{(1)})$. 	Hence, $R_{\bP}(\hlog) = \theta_2^{\star} - \log \phi_{\bP, \hlog}(\theta_2^{\star} )$ where $\theta_2^{\star} $ is the unique solution of the equation $\phi_{\bP,\hlog}(\theta)  = \phi_{\bP,\hlog}'(\theta) $.
\item For $\hind(y) = \1_{\{y \ge \delta\}}$ for some $\delta \in (0, 1)$, $\EB_0 \hind(Y) = 1-\delta$ and $\phi_{\bP, \hind}(\theta) = F(\delta) + \re^{-\theta}(1-F(\delta))$ which is well-defined on $[0, \infty)$.
Hence, $R_{P}(\hind) = \delta \log \frac{\delta}{F(\delta)} + (1-\delta) \log \frac{1-\delta}{1-F(\delta)}$.
\end{itemize}
\end{prop}

\subsection{Proof of Lemma~\ref{lem:extremal-points}}
This lemma can be proved by directly applying Lemma \ref{lem:convex-combination} using the concept of majorization. Below, we provide a constructive proof as an alternative.

\begin{proof}[Proof of Lemma~\ref{lem:extremal-points}]
The set of extremal points of $\FPM$ is denoted by $\Ext(\FPM)$ where, by definition, an element does not lie between any two distinct points in $\FPM$.
Recall that $\bP_{\Delta}^{\star}$ is defined as:
\[
\bP_{\Delta}^{\star} =\biggl(\underbrace{1-\Delta, \ldots, 1-\Delta}_{\floor{\frac{1}{1-\Delta}}\ \textnormal{times}}, 1-(1-\Delta)\cdot \left\lfloor\frac{1}{1-\Delta}\right\rfloor, 0, \ldots\biggr).
\]
We define $\Pi(\bP_{\Delta}^{\star})$ as the set of all NTP distributions that can be obtained by permuting $\bP_{\Delta}^{\star}$, where:
\[
\Pi(\bP_{\Delta}^{\star}) := \left\{\pi(\bP_{\Delta}^{\star}): \pi~\text{is a permutation on $\{1, 2, \ldots, |\Voca|\}$} \right\},
\]
and $\pi(\bP)$ denotes the permuted NTP distribution with the $i$th coordinate given by $P_{\pi(i)}$. Our goal is to prove that 
\[
\Ext(\FPM) = \Pi(\bP_{\Delta}^{\star}).
\]

To start, we show that $\Pi(\bP_{\Delta}^{\star}) \subseteq \Ext(\FPM)$, meaning every permuted NTP distribution of $\bP_{\Delta}^{\star}$ is an extremal point. Suppose the contrary, that there exists $\bP_{\Delta}^{\star} = \theta \bP_1 + (1-\theta) \bP_2$ for some distinct $\bP_1, \bP_2 \in \FPM$ and $\theta \in (0, 1)$. Given the first $\floor{\frac{1}{1-\Delta}}$ coordinates of $\bP_{\Delta}^{\star}$ are maximal, both $\bP_1$ and $\bP_2$ must reach the maximum in these coordinates, leading to a contradiction with the assumption that $\bP_1 \neq \bP_2$.

Next, we prove $\Ext(\FPM) \subseteq \Pi(\bP_{\Delta}^{\star})$, indicating every extremal point is obtained by permuting $\bP_{\Delta}^{\star}$. It suffices to demonstrate that any point in $\FPM$ is in the convex hull of $\Pi(\bP_{\Delta}^{\star})$, denoted by $\mathrm{Conv}(\Pi(\bP_{\Delta}^{\star}))$. 
Without loss of generality, we pick up any $\bP \in \FPM$ with non-increasing order $P_1 \ge P_2 \ge \cdots \ge P_{|\Voca|}$ and $m$ coordinates equal to $1-\Delta$.

In the following, we use induction to show:
\begin{center}
\textit{Any NTP distribution in $\FPM$ with exactly $m$ coordinates equal to $1-\Delta$ is in $\mathrm{Conv}(\Pi(\bP_{\Delta}^{\star}))$.}
\end{center}

\begin{itemize}
\item If $m = \floor{\frac{1}{1-\Delta}}$, then $\bP$ must assign the remaining probability mass (which is $1-m(1-\Delta)$) on some of the rest coordinates.
We assume these coordinates are contained in the set $\IM$ satisfying $\sum_{i \in \IM} P_i = 1-m(1-\Delta)$.
It is easy to find that $\bP$ is the convex interpolation of $|\IM|$ points in $\Pi(\bP_{\Delta}^{\star})$ whose first $N$ coordinates are the same as $1-\Delta$ and the $i$th coordinate is $1-N (1-\Delta)$.

\item Suppose that we already show the statement holds for $m \le \floor{\frac{1}{1-\Delta}}$, we now need to show that it also holds for $m-1$.
For notation simplicity, we let $\PM_m \subset \FPM$ denote the set of NTP distributions which having exact $m$ coordinates equal to $1-\Delta$.
Let $\bP \in \PM_{m-1}$ be the NTP distribution under consideration. 
Given its coordinates are decreasingly ranked, we must have $1-\Delta > P_{m} \ge P_{m+1} \ge \cdots \ge P_{|\Voca|}$.

To proceed with the proof, we discuss two cases.
\begin{itemize}
\item [(a)] If $P_{m}+P_{m+1} > 1-\Delta$, we assert we must have $\bP \in \mathrm{Conv}(\PM_{m})$. To see this, we construct two NTP distributions $\bP_1 \neq \bP_2 \in \PM_m$ such that $\bP$ lies in the segment between the two points.
These two points differ from $\bP$ in exactly two coordinates:
\begin{gather*}
   P_{1,i} = P_{2,i} = P_i ~~\text{for}~~i \notin \{m, m+1\},\\
   P_{1,m} = 1-\Delta ~~\text{and}~~P_{1,m+1} = P_{m}+P_{m+1} - (1-\Delta),\\
   P_{2,m} = P_{1,m-1}, P_{2, m+1} = P_{1,m}.
\end{gather*}
By the hypothesis, we already have $\PM_m \subset \mathrm{Conv}(\Pi(\bP_{\Delta}^{\star}))$ so that $\bP \in \mathrm{Conv}(\Pi(\bP_{\Delta}^{\star}))$.

\item [(b)] If $P_{m}+P_{m+1} \le 1-\Delta$, we similarly construct two NTP distributions $\bP_1 \neq \bP_2 \in \PM_{m-1}$ such that $\bP$ lies in the segment between the two points.
These two points differ from $\bP$ in exactly two coordinates:
\begin{gather*}
P_{1,i} = P_{2,i} = P_i ~~\text{for}~~i \notin \{m, m+1\},\\
P_{1,m} = P_{m}+P_{m+1} ~~\text{and}~~P_{1,m+1} = 0,\\
P_{2,m} =0, P_{2, m+1} = P_{m}+P_{m+1}.
\end{gather*}
In this case, though $\bP_1$ and $\bP_2$ are still in $\PM_{m-1}$, their support decreases by one when compared with $\bP$.
We then discuss the remaining coordinates of $\bP_1$ and $\bP_2$.
Take $\bP_1$ for example.
We want to determine whether the sum of the largest two non-$(1-\Delta)$ probabilities in $\bP_1$ is larger than $1-\Delta$ or not.
If this is true, we arrive at Case (a) again but with a new instance, namely $\bP_1$, which has already been addressed.
If this is not true, we can further decrease the support of $\bP_1$  by treating it as a new instance in Case (b) and combining these two largest non-$(1-\Delta)$ coordinates.
Because $\bP_1$ has a finite number of non-$(1-\Delta)$ coordinates, there is only a finite number of cases where we have to deal with Case (b).
As a result, by tracing the way we decompose $\bP_1$, we always have that $\bP_1 \in \mathrm{Conv}(\Pi(\bP_{\Delta}^{\star}))$ and similarly $\bP_2 \in \mathrm{Conv}(\Pi(\bP_{\Delta}^{\star}))$.
\end{itemize}
Combining Case (a) and Case (b), we prove the correctness of the statement.
\end{itemize}
\end{proof}

\new{
\subsection{Results for a General Probability Set}

In the following, we consider a new, general distribution class $\PM_{\Delta_1, \Delta_2}$.
For two real numbers $0 \le \Delta_2 < \Delta_1 < 1$, we define the following probability class:
\begin{equation}
\label{eq:linear-set}
\PM_{\Delta_1, \Delta_2} = \left\{\bP \in \Simplex(\Voca) : P_{(1)} \le 1-\Delta_1, P_{(2)} \le 1-\Delta_2
\right\},
\end{equation}
where $P_{(i)}$ is the $i$-th largest probability in the NTP distribution $\bP=(P_1, \cdots, P_{|\Voca|})$.
Notably, $\PM_{\Delta_1, \Delta_2}$ is invariant under permutation, meaning that for any permutation $\pi$, if $\bP \in \PM_{\Delta_1, \Delta_2}$, then $\pi(\bP):= (P_{\pi(1)}, \ldots, P_{\pi(|\Voca|)})$ also belongs to $\PM_{\Delta_1, \Delta_2}$. Furthermore, $\PM_{\Delta_1, \Delta_2}$ is a convex polytope, defined as a compact convex set with a finite number of extreme points. The permutation invariance allows us to precisely capture its set of extreme points.

Specifically, we establish the following lemma:
\begin{lem}
\label{lem:ext-point-linear-set}
For $0 \le \Delta_2 < \Delta_1 < 1$, the set of extreme points of $\PM_{\Delta_1, \Delta_2}$ is
\[
\mathrm{Ext}(\PM_{\Delta_1, \Delta_2}) = \Pi(\bP_{\Delta_1, \Delta_2}^{\star})
\]
where $\Pi(\bP_{\Delta_1, \Delta_2}^{\star}) := \left\{\pi(\bP_{\Delta_1, \Delta_2}^{\star}): \pi~\text{is a permutation on $\{1, 2, \ldots, |\Voca|\}$} \right\}$ and
$\bP_{\Delta_1, \Delta_2}^{\star}$ is the least-favorable NTP distribution in $\PM_{\Delta_1, \Delta_2}$ defined by
\[
\bP_{\Delta_1, \Delta_2}^{\star} = \biggl( 1-\Delta_1, \underbrace{1-\Delta_2, \ldots, 1-\Delta_2}_{\floor{\frac{\Delta_1}{1-\Delta_2}}~\text{times}}, \Delta_1 - (1-\Delta_2)\cdot \left\lfloor\frac{\Delta_1}{1-\Delta_2} \right\rfloor, 0,\ldots   \biggl).
\]
\end{lem}
Following a similar line of reasoning as in our main analysis, we could show that
\[
h_{\mathrm{gum}, \Delta_1,\Delta_2}(r) = \log \frac{\rd\mu_{1, \bP_{\Delta_1, \Delta_2}^{\star}}}{\rd\mu_0}(r)
\]
is the optimal score function for the Gumbel-max watermark when using $\PM_{\Delta_1, \Delta_2}$ as the prior probability set.  

However, this probability set is less practical because it requires specifying two real numbers $\Delta_1, \Delta_2$. Therefore, we include this discussion as an extension, illustrating how our analysis and method can be extended to other permutation-invariant sets $\PM$.
In the main text, we focus on the $\Delta$-regular prior set $\FPM$, which represents a special case of $\PM_{\Delta_1, \Delta_2}$. This case occurs when we set $\Delta_1=\Delta$ and $\Delta_2=0$, making it perhaps the simplest form of this class.
% Even in this simple class, there are many interesting results when we try to identify the optimal detection rules.

\begin{proof}[Proof of Lemma \ref{lem:ext-point-linear-set}]

The proof strategy closely follows the approach used in Lemma \ref{lem:extremal-points}, with a few key distinctions that we detail here. 

First, it is straightforward to see that $\bP_{\Delta_1, \Delta_2}^{\star}$ and its permutations are all extreme points since they reach the boundary conditions $P_{(1)} \le 1-\Delta_1$ and $P_{(2)} \le 1-\Delta_2$. 
This result can be easily verified by using the definition of extreme points so we omit them for simplicity. Consequently, we have $ \Pi(\bP_{\Delta_1, \Delta_2}^{\star}) \subseteq \mathrm{Ext}(\PM_{\Delta_1, \Delta_2})$.

To prove the reverse inclusion, $\mathrm{Ext}(\PM_{\Delta_1, \Delta_2}) \subseteq \Pi(\bP_{\Delta_1, \Delta_2}^{\star})$, it suffices to demonstrate that any point in $\PM_{\Delta_1, \Delta_2}$ lies within the convex hull of $\Pi(\bP_{\Delta_1, \Delta_2}^{\star})$. This can be shown by applying the following lemma. By the terminology introduced in Lemma \ref{lem:convex-combination}, we know that $\bP_{\Delta_1, \Delta_2}^{\star}$ majorizes any point in $\PM_{\Delta_1, \Delta_2}$, implying that any point in $\PM_{\Delta_1, \Delta_2}$ can be expressed as a convex combination of $\bP_{\Delta_1, \Delta_2}^{\star}$ and its permutations. Thus, the proof is complete.

\begin{lem}[Lemma 2.2 in \citep{kim2019convexification}]
\label{lem:convex-combination}
Given two vectors $\vx, \vy \in \RB^d$, we say $\vx$ majorizes $\vy$, a property we denote by $\vx \ge_m \vy$, if $\sum_{i=1}^j x_{(i)} \ge \sum_{i=1}^j y_{(i)}$ for all $j=1, \ldots, d$ and $\sum_{i=1}^d x_{(i)} = \sum_{i=1}^d y_{(i)}$. 
Here $x_{(i)}$ is the $i$-th largest component of $\vx$.
If $\vx \ge_{m} \vy$, then $\vy$ is a convex combination of $\vx$ and its permutations.
\end{lem}
\end{proof}
}

\subsection{Proof of Theorem \ref{thm:main-efficiency-gumbel-informal}}
\label{proof:comparison}
\begin{thm}[Formal version of Theorem \ref{thm:main-efficiency-gumbel-informal}]
\label{thm:main-efficiency-gumbel}
There exist two positive constants $ 0.001 <\Delta_1^{\gum}\le\Delta_2^{\gum} < 1$ such that
\begin{itemize}
\item When $10^{-3} < \Delta <\Delta_1^{\gum}$,
\[
\max\left\{R_{\FPM}(\hlog), R_{\FPM}(\hindo)\right\} 
< R_{\FPM}(\hars).
\]
\item When $\Delta_2^{\gum} < \Delta <1$, 
\[
\max\left\{R_{\FPM}(\hars), R_{\FPM}(\hindo) \right\} 
< R_{\FPM}(\hlog).
\]
\item From Figure \ref{fig:inhomo-efficiency}, we observe that $\Delta_1^{\gum}=\Delta_2^{\gum} \approx 0.17756080525215662$.
\end{itemize}
\end{thm}

\begin{proof}[Proof of Theorem \ref{thm:main-efficiency-gumbel}]
By Lemma \ref{lem:worst-MGF}, for any feasible $\Delta\in[0, \frac{|\Voca|-1}{|\Voca|}]$, the efficiency exponents $R_{\FPM}(h)$ are functions of $\Delta$, with the functional forms being
\begin{equation*}
\begin{aligned}
R_{\FPM}(\hars)&=-\inf\limits_{\theta\ge 0} \left[\theta+\log \phi_{\bP^{\star}_\Delta,\hars}(\theta) \right]:=R_{ \hars}(\Delta),\\
R_{\FPM}(\hlog)&=\sup\limits_{\theta \in [0, \frac{1}{1-\Delta} )} [\theta - \log \phi_{\bP^{\star}_\Delta,\hlog}(\theta)]:=R_{\hlog}(\Delta),\\
R_{\FPM}(\hind)&=\delta \log \frac{\delta}{F_{1, \bP_{\Delta}^{\star}}(\delta)} + (1-\delta) \log \frac{1-\delta}{1-F_{1, \bP_{\Delta}^{\star}}(\delta)}:=R_{\hind}(\Delta),
\end{aligned}
\end{equation*}
where definitions of $\phi_{\bP^{\star}_\Delta,\hars}$, $\phi_{\bP^{\star}_\Delta,\hlog}$, and $\phi_{\bP^{\star}_\Delta,\hind}$ can be found in Table \ref{tab:homo-efficiency} and $\bP_{\Delta}^{\star} $ is the least-favorable NTP distribution defined in \eqref{eq:optimal-P}.
Here, we slightly abuse the notion and rewrite each $R_{\FPM}(h)$ as a function of $\Delta$ for simplicity.

We take the comparison between $R_{\hars}(\Delta)$ and $R_{\hlog}(\Delta)$ as an example since the other comparisons follow similarly.
When setting $\Delta=\frac{|\Voca|-1}{|\Voca|}$, we note that $\bP_\Delta^{\star}$ is the uniform distribution over $\Voca$. 
In this case,
\begin{align*}
R_{\hars}\biggl(\frac{|\Voca|-1}{|\Voca|}\biggr)
&<D_{\rm KL}(\mu_0, \mu_{1, \bP_{\Delta}^{\star}})=R_{\hlog}\biggl(\frac{|\Voca|-1}{|\Voca|}\biggr)<\infty,
\end{align*}
where the inequality follows from the Neyman-Pearson lemma and the equation uses the observation that $\hopt$ is reduced to $\hlog$ if $\Delta=\frac{|\Voca|-1}{|\Voca|}$.
Note that $R_{\hars}$ and $R_{\hlog}$ are continuous with respect to $\Delta$.
By the continuity, we conclude that there exists a constant $ 0 <\Delta_2^{\gum} \le \frac{|\Voca|-1}{|\Voca|}$ such that $R_{\hars}(\Delta) <  R_{\hlog}(\Delta) < \infty$ holds for all $\frac{|\Voca|-1}{|\Voca|} \ge\Delta \ge \Delta_2^{\gum}$. 

For the other direction, when setting $\Delta=0.01$, we numerically find that $R_{\hlog}(0.001)\approx 3 \time 10^{-6}$ and $R_{\hars}(0.001) \approx 2.5 \times 10^{-5}$. By the continuity, there exists a positive constant  $10^{-3} < \Delta_1^{\gum}$ such that $R_{\hlog} (\Delta)
< R_{\hars}(\Delta)$ for all $10^{-3} < \Delta < \Delta_1^{\gum}$.

\begin{rem}
The constant $10^{-3}$ used in the above proof is not essential. 
For a more accurate characterization, we use the python \textsf{SciPy} package to solve the root of $R_{\FPM}(\hars) = R_{\FPM}(\hlog)$, the numerical result suggests that there is only one root whose value is $0.17756080525215662$ when $\Delta \ge 0.001$.
\end{rem}
\end{proof}

\subsection{Analysis of Efficiency Gap}
\label{appen:effiency-gap}

\begin{lem}
\label{lem:efficiency-gap}
For the probability vector $\bP$, let $|\mathrm{supp}(\bP)|$ denote the size of the support of $\bP$ where $\mathrm{supp}(\bP):= \{ \token \in \Voca: P_{\token} \neq 0\}$.
If $P_{(1)} < 1$, then
\[
1 - \min \left\{  \frac{P_{(1)}}{1-P_{(1)}} \log \frac{1}{P_{(1)}},\frac{ |\mathrm{supp}(\bP)|-\frac{1}{P_{(1)}}}{|\mathrm{supp}(\bP)| - 1 - \log |\mathrm{supp}(\bP)|}  
\right\}\le \frac{R_{\bP}(\hlog)}{R_{\bP}(\hllr)} \le 1.
\]
\end{lem}

\begin{cor}[Ignorable suboptimality]
\label{cor:ignorable-suboptimality}
\[
\lim_{\Delta \to 1} \frac{R_{\FPM}(\hlog)}{R_{\FPM}(\hoptars)} \to 1.
\]
\end{cor} 

\begin{proof}[Proof of Corollary \ref{cor:ignorable-suboptimality}]
By Lemma \ref{lem:worst-MGF}, we have that for non-decreasing function $h$,
\[
R_{\FPM}(h) = R_{\bP_{\Delta}^{\star}}(h)
\]
where $\bP_{\Delta}^{\star}$ is the least-favorable NTP distribution defined in \eqref{eq:optimal-P}.
This corollary directly follows from Lemma \ref{lem:efficiency-gap} by noting the largest probability in  $\bP_{\Delta}^{\star}$ is $1-\Delta$ which converges to zero if $\Delta \to 1$. 
\end{proof}

At the end, we provide the proof of Lemma~\ref{lem:efficiency-gap}.
\begin{proof}[Proof of Lemma~\ref{lem:efficiency-gap}]
Let $P_{(1)}$ denote the largest probability in $\bP$.
By definition, we have $R_{\bP}(h) \le \KL(\mu_0, \mu_{1,\bP})$ for any score function $h$.
We then focus on the other direction.
On one hand, it follows that
\begin{align*}
    R_{\bP}(\hlog)
    &=\sup_{\theta \in [0, 1/P_{(1)})} \left( \theta - \log \sum_{\token \in \Voca} \frac{P_{\token}}{1-P_{\token}\theta} \right)\\
    &\ge \sup_{\theta \in [0, 1/P_{(1)})} \left( \theta  + \log(1-P_{(1)}\theta)  \right)\\
    &= \frac{1}{P_{(1)}} - 1-\log \frac{1}{P_{(1)}}.
\end{align*}
The last equality uses the condition $P_{(1)}< 1$ so that the supreme is obtained within $[0, 1/P_{(1)})$.
On the other hand, note that $\lim\limits_{p \to 0}r^{1/p-1} = 0$ for any $r \in (0,1)$ and thus
\begin{align*}
    \KL(\mu_0, \mu_{1,\bP})
    &= -\int_0^1\log \sum_{\token \in \Voca} r^{1/P_{\token}-1}\rd r\\
    &\le \min \left\{ -\int_0^1\log r^{1/P_{(1)}-1}\rd r, -\int_0^1(|\mathrm{supp}(\bP)|-1) \log r \rd r - \log |\mathrm{supp}(\bP)| \right\}\\
    &= \min\left\{ 1/P_{(1)}-1, |\mathrm{supp}(\bP)|-1 - \log |\mathrm{supp}(\bP)| \right\}.
\end{align*}
The particular inequality uses the following lemma.
This lemma can be proven by computing the derivative of $\KL(\mu_0, \mu_1)$ with respect to the probability vector $P = (p_1, \cdots, P_{\token})$.
\begin{lem}
    \label{lem:K}
    Let $|\mathrm{supp}(\bP)|$ denote the cardinality of the support of $\bP$, then
    \[
    0 \le 
    \KL(\mu_0, \mu_{1,\bP}) = - \int_0^1 \log \left( \sum_{\token \in \Voca} r^{1/P_{\token}-1} \right) \rd r \le|\mathrm{supp}(\bP)|-1 - \log|\mathrm{supp}(\bP)|.
    \]
    The left inequality is achieved when $P_{\token}=1$ for some $\token \in \Voca$ and the right inequality is achieved when $P_{\token} \equiv 1/|\mathrm{supp}(\bP)|$ for $\token \in \mathrm{supp}(\bP)$.
\end{lem}
\begin{proof}[Proof of Lemma~\ref{lem:K}]
The left inequality is obvious due to the non-negativeness of  KL divergence.
The right inequality follows from Jensen’s inequality. See Remark~\ref{rem:hlog} for the details.
\end{proof}
We complete the proof by noting
\begin{align*}
    \frac{\frac{1}{P_{(1)}} - 1-\log \frac{1}{P_{(1)}}}{|\mathrm{supp}(\bP)|-1 - \log |\mathrm{supp}(\bP)|} 
    \ge 
    1- \frac{ |\mathrm{supp}(\bP)|-\frac{1}{P_{(1)}}}{|\mathrm{supp}(\bP)|-1 - \log |\mathrm{supp}(\bP)|}.
\end{align*}	
\end{proof}

\subsection{
\texorpdfstring{$\FPM$}{P}-Efficiency for the Baby Watermark}
\label{proof:baby-watermark}
\begin{lem}\label{lem:baby_wtm}
For the baby watermark in \eqref{eq:baby_wmk},  the summary function $Y(\token_t,\xi_t)=(2\token_t-1)(2\xi_t-1)$ and the score function $\hid(y)=y$, the efficiency exponent for the belief class $\FPM=\{\bP=(P_0,P_1):\Delta\le\min(P_0,P_1)\le1/2\}$ is
\[
R_{\FPM}(\hid)=-\inf_{\theta\ge0}\log\biggl[\frac{1}{\theta}\biggl\{\frac{\re^{\theta(1-2\Delta)}+\re^{-\theta(1-2\Delta)}}{2}-\re^{-\theta}\biggr\}\biggr].
\]
\end{lem}
\begin{proof}[Proof of Lemma \ref{lem:baby_wtm}]
With the choice of $(Y,h)$, we obtain the detection rule as in \eqref{eq:Th}. By Proposition \ref{thm:inhomo-type-II}, the efficiency of this detection rule is measured by $R_{\PM,h}$ defined in \eqref{eq:efficiency-exponent}. Note that by the independence of $\token_t$ and $\xi_t$ under the null, we have $\EB_0h(Y)=0$. Thus, the efficiency exponent can be written as
\[
R_{\PM}(hid)=-\inf_{\theta\ge0}\sup_{\bP\in\PM}\log\EB_{1,\bP}\left(\re^{-\theta Y(\token,\xi)}\right).
\]
To calculate $R_{\PM,h}$, we need first to obtain the distribution of $Y(\token,\xi)$.
\begin{lem}\label{lem:dis_baby}
Under $H_1$, the distribution of $Y(\token,\xi)$ is
\[
\PB_{H_1}(Y\le y)=\left\{\begin{array}{cc}
0&y\le2P_{\min}-1,\\
(y+1)/2-P_{\min}&2P_{\min}-1<y\le1-2P_{\min},\\
y&1-2P_{\min}<y\le1,\\
1&y>1,
\end{array}\right.
\]
where $P_{\min}=\min\{P_0,P_1\}$, $P_1=\PB(\token=1)$ and $P_0=1-P_1$.
\end{lem}
\begin{proof}[Proof of Lemma \ref{lem:dis_baby}]
By the Bayes rule, we have
\begin{align*}
\PB(Y\le y)&=\PB(Y\le y\mid\token=1)\PB(\token=1)+\PB(Y\le y\mid\token=0)\PB(\token=1)\\
&=\PB\biggl(\xi\le\frac{y+1}{2}\bigg|\token=1\biggr)(1-P_0)+\PB\biggl(\xi\ge\frac{1-y}{2}\bigg|\token=0\biggr)P_0.
\end{align*}
Note that for $y<-1$, $(y+1)/2<0$ and $(1-y)/2>1$ and for $y>1$, $(y+1)/2>1$ and $(1-y)/2<0$. Thus, $\PB(Y\le y)=0$ for $y<-1$ and $\PB(Y\le y)=1$ for $y>1$. For $-1\le y\le1$, it suffices to calculate $\PB(\xi\le a\mid\token=1)$ and $\PB(\xi\le a\mid\token=0)$ for any $-1\le a\le1$. By the Bayes rule,
\begin{align*}
\PB(\xi\le a\mid\token=1)=\frac{\PB(\token=1\mid\xi\le a)\PB(\xi\le a)}{\PB(\token=1)}=\frac{\max(0,a-P_0)}{a}\frac{a}{1-P_0}=\max\biggl(0,\frac{a-P_0}{1-P_0}\biggr).
\end{align*}
Similarly, 
\[
\PB(\xi\le a\mid\token=0)=\min\biggl(1,\frac{a}{P_0}\biggr).
\]
Combining these pieces, we obtain
\[
\PB_{H_1}(Y\le y)=\max\biggl(0,\frac{y+1}{2}-P_0\biggr)+\max\biggl(0,\frac{y-1}{2}+P_0\biggr).
\]
If $P_0\le1/2$, then the distribution can be further simplified as
\begin{equation}\label{eq:P_0}
\PB_{H_1}(Y\le y)=y \cdot \1_{\{1\ge y>1-2P_0\}}y+\biggl(\frac{y+1}{2}-P_0\biggr) \cdot \1_{\{1-2P_0\ge y\ge2P_0-1]\}}
\end{equation}
for $-1\le y\le1$. If $P_0\ge1/2$, \eqref{eq:P_0} holds with $P_0$ being replaced by $P_1$. We then complete the proof.
\end{proof}
By Lemma \ref{lem:dis_baby}, when $\Delta \le 0.5$, the MGF is given by
\[
\phi_{\bP,h}(\theta)=\EB_{1, \bP}\re^{-\theta Y}=\frac{1}{\theta}\biggl\{\frac{\re^{\theta(1-2P_{\min})}+\re^{-\theta(1-2P_{\min})}}{2}-\re^{-\theta}\biggr\}.
\]
Note that $\phi_{\bP,h}(\theta)$ is convex with respect to $\bP$ for any $\theta \ge 0$. Moreover, for $\FPM=\{\bP=(P_0,P_1):\Delta\le\min(P_0,P_1)\le1/2\}$, it can be easily verified that $\phi_{\bP,h}(\theta)$ attains its maximum as a function of  $P_{\min}$ when $P_{\min}=\Delta$, in which case
\begin{align*}
R_{\FPM}(\hid)&=-\inf_{\theta\ge0}\log\biggl[\frac{1}{\theta}\biggl\{\frac{\re^{\theta(1-2\Delta)}+\re^{-\theta(1-2\Delta)}}{2}-\re^{-\theta}\biggr\}\biggr].
\end{align*}
\end{proof}

\section{Proof for Inverse Transform Watermarks}
\label{append:inverse}
\subsection{Distribution Characterization}

The first step in applying the framework of class-dependent efficiency is to characterize the distributions of $\Ydif_t$ under both null and alternative hypotheses.
Note that it is a bivariate function of $(U_t, \pi_t(\token_t))$.
We first focus on the joint distribution of $(U_t, \pi_t(\token_t))$ instead.

\begin{lem}\label{lem:inverse_dis_shift}
Let $r\in[0, 1]$ and $\token \in \Voca$.
Under $H_0$, the joint CDF of $(U_t, \pi_t(\token_t))$ is:
\[
\PB(U_t \le r, \pi_t(\token_t) = \token |H_0)
= \frac{r}{|\Voca|}.
\]
Under $H_1$, the joint CDF of $(U_t, \eta(\pi_t(\token_t)))$ conditioning on $\bP_t$ is:
\begin{align*}
\PB(U_t \le r, \pi_t(\token_t) =\token|\bP_t, H_1)
&= \frac{1}{|\Voca|!}\sum_{\pi \in\Pi}\PB\left(U_t\in \left(
a_{\pi, \token-1}, a_{\pi, \token}  \right]  \cap [0, r]\right),
\end{align*}
where $a_{\pi,\token}=\sum_{j=1}^{\token}P_{t,\pi(j)}$ is the sum of the first $\token$ probabilities of $\bP_t$ under the permutation $\pi$.
\end{lem}

Lemma~\ref{lem:inverse_dis_shift} provides the explicit formulation for the distribution of $(U_t, \eta(\pi_t(\token_t)))$.
Using Lemma \ref{lem:inverse_dis_shift}, one can apply the change of variables to compute the exact distribution for $\Ydif_t$ and thus prove Lemma \ref{lem:exact-CDF-dif}.

\begin{proof}[Proof of Lemma \ref{lem:exact-CDF-dif}]

It follows from Lemma \ref{lem:inverse_dis_shift} that
\begin{align*}
\PB(Y_t^{\dif} \le r|H_0)&= \sum_{\token \in \Voca} \PB(|U_t -\eta(\token)| \le r, \pi_t(\token_t) =\token|H_0)\\
&=\sum_{\token \in \Voca} \PB( U_t \in [\eta(\token)-r, \eta(\token)+r] \bigcap [0, 1], \pi_t(\token_t) =\token|H_0)\\
&=\frac{1}{|\Voca|}\sum_{\token \in \Voca} \left[ 
[\eta(\token)+r]_{[0, 1]}-[\eta(\token)-r]_{[0, 1]} \right].
\end{align*}
On the other hand,
\begin{align*}
\PB(Y_t^{\dif}\le r|\bP_t, H_1)
&=\sum_{\token \in \Voca} \PB(|U_t -\eta(\token)| \le r, \pi_t(\token_t) =\token|\bP_t, H_1)\\
&=\frac{1}{|\Voca|!}\sum_{\pi \in \Pi} \sum_{\token \in \Voca} \lambda((a_{\pi,\token-1}, a_{\pi,\token}]\cap B(\eta(\token), r)),
\end{align*}
where $a_{\pi,\token}=\sum_{j=1}^{\token}P_{t,\pi(j)}$, $B(v, r)=\{x \in [0, 1]: |x-v| \le r\}$ and $\lambda$ is uniform measure on $[0, 1]$. 
\end{proof}

At the end, we provide the proof of Lemma \ref{lem:inverse_dis_shift}. 

\begin{proof}[Proof of Lemma \ref{lem:inverse_dis_shift}]
In the following, we denote the probability measure conditioning on $\bP_t$ under $H_1$ by $\PB_1(\cdot) = \PB(\cdot|\bP_t, H_1)$ for simplicity and $\PB_0(\cdot)$ denotes the probability under $H_0$.
For any $r \in \R$ and $\token \in \Voca$, by the law of total probability we have that for $i \in \{0, 1\}$,
\begin{align}
\PB_i(U_t \le r, \pi_t(\token_t) =\token)
&=\PB_i(U_t\le r \mid\pi_t(\token_t)=\token)\PB_i(\pi_t(\token_t)=\token).\label{eq:ex_yt1}
\end{align}

Under $H_0$, the construction of the watermark implies that the random variables $\pi_t$, $\token_t$, and $U_t$ are independent of each other. 
The independence result has two consequences.
\begin{itemize}
\item First, $U_t$ is independent with $\pi_t(\token_t)$ so that $\PB_0(U_t\le r \mid\pi_t(\token_t)=\token)=\PB_0(U_t\le r ) = r_1$.
\item Second, $\pi_t(\token_t)$ is uniformly distributed on $\Voca$ due to
\[
\PB_0(\pi_t(\token_t)=\token) = \sum_{\ell \in \Voca} \PB_0(\pi_t(\token_t)=\token,\token_t=\ell) = \sum_{\ell \in \Voca} \PB_0(\pi_t(\ell)=\token) \PB_0(\token_t=\ell) = \frac{1}{|\Voca|}.
\]
\end{itemize}
As a result,
\[
\PB_0(U_t \le r, \pi_t(\token_t) =\token)=  \frac{r}{|\Voca|}.
\]

Under $H_1$, imagine that we fix $\pi_t$ to be a given permutation $\pi$ (so that $\pi_t$ is not random at all). In this case, $\token_t$ is also generated by the inverse transform sampling and thus its distribution is still $\bP_t$. 
This remains true when $\pi_t$ is allowed to be uniformly random; $\token_t$ continues to follow the distribution $\bP_t$. Consequently, one can show, using Bayes’ formula, that $\pi_t$ is independent of $\token_t$. In particular, we have
\[
\PB_1(\pi_t=\pi\mid\token_t=\token)=\frac{\PB_1(\token_t=\token\mid\pi_t=\pi)\PB_1(\pi_t=\pi)}{\PB_1(\token_t=\token)}=\PB_1(\pi_t=\pi).
\]
Thus, $\pi_t$ is independent of $\token_t$ under $H_1$.
Using this independence,  we obtain that $\PB_1(\pi_t(\token_t)=\token)=|\Voca|^{-1}$. 
It remains to evaluate the distribution of $U_t$ conditional on $\pi_t(\token_t)=\token$. 
First, by the law of total probability and Lemma \ref{lem:p1its}, we have for any measurable set $S$,
\begin{align*}
\PB_1(U_t\in S\mid\token_t=\ell,\pi_t(\ell)=\token)&=\sum_{\pi\in\Pi:\pi(\ell)=k}\PB_1(U_t\in S,\pi_t=\pi\mid\token_t=\ell,\pi_t(\ell)=\token)\\
&=\sum_{\pi\in\Pi:\pi(\ell)=k}\PB_1(U_t\in S\mid\token_t=\ell,\pi_t=\pi)\PB(\pi_t=\pi\mid\token_t=\ell,\pi_t(\ell)=\token)\\
&=\frac{1}{(|\Voca|-1)!}\sum_{\pi\in\Pi:\pi(\ell)=k} \frac{1}{P_{t,\ell}}\PB(U_t\in S_t(\ell,\pi) \cap S),
\end{align*}
where $S_t(\ell, \pi)$ is defined by
\begin{equation}
\label{eq:\token_t}
S_t(\ell,\pi)=\biggl( \sum_{\token \in \Voca}P_{t,\token}\1_{\{\pi(\token)<\pi(\ell)\}},\sum_{\token \in \Voca}P_{t,\token}\1_{\{\pi(\token)\le\pi(\ell)\}}\biggr].
\end{equation}
Again by the law of total probability and the fact that $\token_t$ and $\pi_t$ are independent, we have for any measurable $S \subset [0, 1]$,
\begin{align}
\label{eq:condition-U}
\PB_1(U_t \in S \mid \pi_t(\token_t)=\token)
&=\sum_{\ell \in \Voca}\PB_1(U_t\in S,\token_t=\ell\mid \pi_t(\token_t)=\token) \notag \\
&=\sum_{\ell \in \Voca}\PB_1(U_t\in S\mid\token_t=\ell,\pi_t(\ell)=\token)\PB_1(\token_t=\ell\mid\pi_t(\token_t)=\token) \notag\\
&=\sum_{\ell \in \Voca}\PB_1(U_t\in S\mid\token_t=\ell,\pi_t(\ell)=\token)P_{t,\ell} \notag \\
&=\frac{1}{(|\Voca|-1)!}\sum_{\ell \in \Voca}\sum_{\pi\in\Pi:\pi(\ell)=\token}\PB(U_t\in S_t(\ell,\pi) \cap S)\notag\\
&=\frac{1}{(|\Voca|-1)!}\sum_{\ell \in \Voca}\sum_{\pi\in\Pi}\PB(U_t\in S_t(\ell,\pi) \cap S) \cdot \1_{ \{\pi(\ell) = \token\} }\notag\\
&=\frac{1}{(|\Voca|-1)!}\sum_{\pi\in\Pi}\PB(U_t\in S_t(\pi^{-1}(\token),\pi) \cap S)\notag\\
&=\frac{1}{(|\Voca|-1)!}\sum_{\pi\in\Pi}\PB\left(U_t\in \left(
a_{\pi^{-1}, \token-1}, a_{\pi^{-1}, \token}  \right]  \cap S\right) \notag \\
&=\frac{1}{(|\Voca|-1)!}\sum_{\pi \in\Pi}\PB\left(U_t\in \left(
a_{\pi,\token-1}, a_{\pi,\token}  \right]  \cap S\right),
\end{align}
where $a_{\pi^{-1}, \token} = \sum_{j \le \token} P_{t, \pi^{-1}(j)}$ is the endpoint of $S_t(\pi^{-1}(\token),\pi)$ by definition.
The last equation uses $\pi^{-1}\in\Pi$ is equivalent to $\pi \in \Pi$.

Plugging \eqref{eq:condition-U} into \eqref{eq:ex_yt1}, we have
\begin{align*}
\PB_1(U_t \le r, \pi_t(\token_t) = \token)
&=\frac{1}{|\Voca|} \PB_1(U_t\le r \mid\pi_t(\token_t)=\token)\\
&= \frac{1}{|\Voca|!}\sum_{\pi \in\Pi}\PB\left(U_t\in \left(
a_{\pi,\token-1}, a_{\pi,\token}  \right]  \cap [0, r]\right).
\end{align*}
\end{proof}

\begin{lem}\label{lem:p1its}
Under $H_1$, the probability distribution of $U_t\mid\token_t=\token,\pi_t=\pi$ is
\[
\PB_1(U_t \in S \mid\token_t=\token,\pi_t=\pi)=
\frac{1}{P_{t,\token}}\PB(U_t\in S_t(\token,\pi) \cap S),
\]
where $S$ is any measurable set in the range of $U_t$ and $\token_t$ is defined in~\eqref{eq:\token_t}.
\end{lem}

\begin{proof}[Proof of Lemma \ref{lem:p1its}]
For any measurable set $S\subset[0,1]$,
\begin{align*}
\PB_1(U_t\in S\mid\pi_t=\pi,\token_t=\token)
&=\frac{\PB_1(\pi_t=\pi,\token_t=\token\mid U_t\in S)\PB(U_t\in S)}{\PB_1(\pi_t=\pi,\token_t=\token)}.
\end{align*}
By the definition of $\token_t$, $\token_t=\token$ if and only if
\[
U_t\in\biggl(\sum_{j \in \Voca}P_{t,j}\1_{\{\pi_t(j)<\pi_t(\token)\}},\sum_{j \in \Voca}P_{t,j}\1_{\{\pi_t(j)\le\pi_t(\token)\}}\biggr] =S_t(\token, \pi_t).
\]
Thus, using the independence between $U_t$ and $\pi_t$, 
\begin{align*}
\PB_1(\pi_t=\pi,\token_t=\token\mid U_t\in S)
&=\PB_1(\token_t=\token\mid\pi_t=\pi,U_t\in S)\PB_1(\pi_t=\pi)\\
&=\PB_1(U_t\in S_t(\token,\pi_t)\mid\pi_t=\pi,U_t\in S)\PB_1(\pi_t=\pi)\\
&=\PB_1(U_t\in S_t(\token,\pi)\mid\pi_t=\pi,U_t\in S)\PB_1(\pi_t=\pi)\\
&=\PB_1(U_t\in S_t(\token,\pi)\mid U_t\in S)\PB_1(\pi_t=\pi)\\
&=\PB(U_t\in S_t(\token,\pi) \cap S)/\PB(U_t\in S) \cdot \PB_1(\pi_t=\pi).
\end{align*}
Moreover, note that $\pi_t$ is independent with $\token_t$ under $H_1$ so that $\PB_1(\pi_t=\pi,\token_t=\token)=\PB_1(\pi_t=\pi)\PB_1(\token_t=\token)=P_{t,\token}\PB_1(\pi_t=\pi)$.
Combining these pieces completes the proof.
\end{proof}

\subsection{Key Lemma for Asymptotic Analysis}
\label{proof:J}
\newcommand{\tpi}{{\tilde{\pi}}}
\newcommand{\tk}{{\tilde{\token}}}
\newcommand{\tell}{{\tilde{\ell}}}

At the center of our asymptotic analysis is Lemma \ref{lem:J}.

\begin{lem}
\label{lem:J}
Let $\mathcal{J}$ collect any $1$-Lipschitz-continuous function on $[0, 1]^2$.
For any $J \in \mathcal{J}$, $J(U_t, \eta(\pi_t(\token_t)))$ is a bounded random variable where $\eta(\token)=\frac{\token-1}{|\Voca|-1}$.
Let $\bP_t$ be the NTP distribution that $\token_t$ follows. 
We introduce a new belief class, denoted by $\TPM$ with the following definition
\begin{equation*}
% \label{eq:PM-help}
\TPM = \left\{ \bP : P_{(1)} = 1-\Delta, \log |\Voca|\cdot \Psecond \le \varepsilon_{|\Voca|} \right\}.
\end{equation*}
The following equation holds uniformly for $\bP_t \in \TPM$ and $\Delta \in [0, 1-\frac{1}{|\Voca|}]$:
\begin{align}
\begin{split}
\label{eq:J-expectation}
\lim_{|\Voca| \to \infty} \EB_{1, \bP_t}  J(U_t, \eta(\pi_t(\token_t))) 
&= \Delta \int_0^1(1-x) 
J(\Delta x, x)
\rd x
+\Delta \int_0^1x 
J(\Delta x+1-\Delta, x)
\rd x\\
&\qquad +\int_0^1 \rd y
\int_{\Delta y}^{\Delta y + 1-\Delta}  J(x, y)  \rd x.
\end{split}
\end{align}
The uniform convergence holds in the sense that
\[
\lim_{|\Voca| \to \infty} \sup_{J \in \mathcal{J}}\sup_{\Delta \in [0, 1-\frac{1}{|\Voca|}]}\sup_{\bP_t \in \TPM}\left| \EB_{1, \bP_t}  J(U_t, \eta(\pi_t(\token_t)))  - \text{right hand side of \eqref{eq:J-expectation}} \right| = 0.
\]
\end{lem}

Given the flexibility in the choice of the Lipschitz-continuous function $J$, Lemma \ref{lem:J} demonstrates that this random vector $(U_t, \eta(\pi_t(\token_t)))$ weakly converges to a joint bivariate random vector, which we denoted by $(Z_1, Z_2)$. This convergence is substantiated by Theorem 3.9.1 in \citep{durrett2013probability}, which provides the theoretical foundation. Moreover, the expected value of \(J(Z_1, Z_2)\) is explicitly provided by \eqref{eq:J-expectation}, showcasing a straightforward formulation.

By relaxing the requirement for uniform convergence, we can alleviate the condition of the Lipschitz continuity imposed on \(J\). This relaxation leads to the derivation of a corollary that broadens the applicability of the expectation formula in \eqref{eq:J-expectation}.

\begin{cor}
\label{cor:J-expectation}
For any bounded and measurable $J: [0, 1]^2 \to \R$, \eqref{eq:J-expectation} also holds.
\end{cor}
\begin{proof}[Proof of Corollary \ref{cor:J-expectation}]
It suffices to show \eqref{eq:J-expectation} holds for any indicator functions, that is, $J(x, y) = \1_{\{x \le r_1, y \le r_2\}}$ for a given pair $(r_1, r_2) \in [0, 1]^2$.
For the indicator function $x \mapsto \1_{\{x \le r\}}$, we introduce its Lipschitz variant $f_{\eps, r}(x)$ for a small number $\eps > 0$:
\[
f_{\eps, r}(x) := \begin{cases}
1 & ~\text{if}~0 \le x \le r, \\
1-\frac{x-r}{\eps}& ~\text{if}~ r \le x \le r + \eps, \\
0  & ~\text{if} ~r + \eps \le x \le 1. \\
\end{cases}
\]
Utilizing Lemma \ref{lem:J} with the specific function $J_{\eps}(x, y) = f_{\eps, r_1}(x)f_{\eps, r_2}(y)$, we establish the validity of \eqref{eq:J-expectation} for $J_{\eps}$. By progressively reducing $\eps$ to zero and invoking the monotone convergence theorem, we affirm that \eqref{eq:J-expectation} remains applicable for this particular $J$.

\end{proof}
\begin{rem}
This above proof yields an insightful result. Employing $J(x, y) = \1_{\{x \le r_1, y \le r_2\}}$ within \eqref{eq:J-expectation} reveals the asymptotic CDF of $(U_t, \eta(\pi_t(\token_t)))$ to be free of discontinuities across the square $[0, 1]^2$. This continuity is a direct consequence of the right-hand side of \eqref{eq:J-expectation} (substituting $J(x, y) = \1_{\{x \le r_1, y \le r_2\}}$) being continuous in $(r_1, r_2)$.
\end{rem}

\subsubsection{Proof of Lemma \ref{lem:J}}

At the end of this section, we provide the proof of Lemma \ref{lem:J}.

\begin{proof}[Proof of Lemma~\ref{lem:J}]
Since all the following probability is given conditioning on $\bP_t$, for simplicity, we omit the dependence on $\bP_t$ and use $\PB_1(\cdot)$ to denote the conditional probability.
For any $J \in \mathcal{J}$, we know that $J$ is $1$-Lipschitz-continuous and bounded.
For simplicity, we use $O(1)$ to hide universal constants such as the bound of $J$, and its Lipschitz-continuous constant. 

Lemma \ref{lem:inverse_dis_shift} shows that the joint distribution of $(U_t, \pi_t(\token_t))$ under $H_1$ is 
\[
\PB_1(U_t \le r, \pi_t(\token_t)=\token)
= \frac{1}{|\Voca|!} \sum_{\pi \in\Pi}\PB\left(U_t\in \left(
a_{\pi,\token-1}, a_{\pi,\token}  \right]  \cap [0, r_1]\right),
\]
where $a_{\pi,\token}=\sum_{j=1}^{\token}P_{t,\pi(j)}$ is the sum of the first $\token$ probabilities of $\bP_t$ under the permutation $\pi$.
Hence,
\begin{align*}
\EB_{1, \bP_t} J(U_t, \eta(\pi_t(\token_t)))
&=  \frac{1}{|\Voca|!} \sum_{\token \in \Voca}\sum_{\pi \in\Pi}  \int_{a_{\pi,\token-1}}^{a_{\pi,\token}} J(r, \eta(\token)) \rd r.
\end{align*}

Since we fix $t \in [n]$, we drop the dependence of $\bP_t$ on $t$ for simplicity.
We let $P_{(1)} = \max_{\token \in \Voca} P_{\token}$ denote the largest probability in $P$ so that $P_{(1)} = 1-\Delta$.
Now, we discuss the position, denoted by $\ell$, that is mapped to 1 according to the permutation $\pi_t$, that is, $\pi_t(\ell) = 1$ and divide the summation $\sum_{\pi \in\Pi}$ into three parts.
We use $\Pi_{\ell}$ to represent the set of permutations that map $\ell$ to 1. Then $\Pi = \bigcup_{\ell \in \Voca} \Pi_\ell$ and thus
\begin{align*}
\EB_{1, \bP_t}&J(U_t, \eta(\pi_t(\token_t)))
= \frac{1}{|\Voca|!} \sum_{\token \in \Voca}  \sum_{\ell \in \Voca} \sum_{\pi \in\Pi_\ell}  \int_{a_{\pi,\token-1}}^{a_{\pi,\token}}  J(r, \eta(\token)) \rd r \notag \\
&= \frac{1}{|\Voca|!}  \sum_{\token \in \Voca}  \sum_{\ell=1}^{\token-1} \sum_{\pi \in\Pi_\ell} \int_{a_{\pi,\token-1}}^{a_{\pi,\token}} J(r, \eta(\token)) \rd r 
+ \frac{1}{|\Voca|!}  \sum_{\token \in \Voca} \sum_{\ell=\token+1}^{|\Voca|} \sum_{\pi \in\Pi_\ell} \int_{a_{\pi,\token-1}}^{a_{\pi,\token}} J(r, \eta(\token)) \rd r \notag\\
&\quad +  \frac{1}{|\Voca|!} \sum_{\token \in \Voca} \sum_{\pi \in\Pi_{\token}}  \int_{a_{\pi,\token-1}}^{a_{\pi,\token}} J(r, \eta(\token)) \rd r \notag\\
&=: \TM_1 +\TM_2 + \TM_3.
% \label{eq:two-terms-in-MGF}
\end{align*}

In the following, we analyze the three terms $\TM_1, \TM_2$ and $\TM_3$ respectively.

\paragraph*{For the term $\TM_1$}
For simplicity, we let $P_{(2)} = \max\limits_{k: P_{\token} \neq P_{(1)}} P_{\token}$ is the second largest probability in $\bP$.
The Taylor expansion together with the boundedness of $J$ implies that
\[
\int_{a_{\pi,\token-1}}^{a_{\pi,\token}} J(r, \eta(\token)) \rd r
=  J(a_{\pi,\token}, \eta(\token)) P_{\pi(\token)} + O(P_{\pi(\token)}^2).
\]
Note that for any $\ell \neq \token$, $\pi \in \Pi_\ell$ mush satisfy $\pi(\token) \neq 1$ and thus $P_{\pi(\token)} \le P_{(2)}$.
As a result, it follows that
\begin{align}
\TM_1 &= \frac{1}{|\Voca|!}  \sum_{\token \in \Voca}  \sum_{\ell=\token+1}^{|\Voca|} \sum_{\pi \in\Pi_\ell} \left[J(a_{\pi,\token}, \eta(\token)) P_{\pi(\token)} + O(P_{\pi(\token)}^2)\right] \nonumber \\ 
&=\frac{1}{|\Voca|!}  \sum_{\token \in \Voca}  \sum_{\ell=\token+1}^{|\Voca|} \sum_{\pi \in\Pi_\ell} J(a_{\pi,\token}, \eta(\token)) P_{\pi(\token)}  + O(P_{(2)}) \nonumber \\
&=\frac{1}{|\Voca|!}  \sum_{\token \in \Voca}  \sum_{\pi \in\Pi} J(a_{\pi,\token}, \eta(\token)) P_{\pi(\token)} I_{\pi, \token}+ O(P_{(2)}) \nonumber \\
&=\EB_{\bar{\pi}} \sum_{\token \in \Voca} J(a_{\bar{\pi}, \token}, \eta(\token)) P_{\bar{\pi}(\token)} I_{\bar{\pi}, \token} + O(P_{(2)}),
\label{eq:bound-T1-first-part}
\end{align}
where $\bar{\pi}\sim \UM(\Pi)$ is a uniformly distributed permutation and  $I_{\pi, \token} = \1_{\{{\pi}^{-1}(1) >k\} }$ is the indicator function.
We then focus on the main term in \eqref{eq:bound-T1-first-part}.
By Lipschitz-continuity of $J(\cdot, \cdot)$, we have
\begin{align}
&\left| \sum_{\token \in \Voca} \left[J(a_{\bar{\pi}, \token}, \eta(\token)) P_{\bar{\pi}(\token)} I_{\bar{\pi}, \token}-
J\left(\frac{\token\Delta}{|\Voca|-1}, \eta(\token)\right)  \frac{\Delta I_{\bar{\pi}, \token}}{|\Voca|-1}  \right] \right| \nonumber \\
&\quad \le  \left|\sum_{\token \in \Voca}J(a_{\bar{\pi}, \token}, \eta(\token))  
\left( P_{\bar{\pi}(\token)} -   \frac{\Delta}{|\Voca|-1}  \right)I_{\bar{\pi}, \token} \right|  +O(1) \cdot \sum_{\token \in \Voca} \frac{\Delta I_{\bar{\pi}, \token}}{|\Voca|-1}  \cdot  \left| a_{\pi,\token} -\frac{\token\Delta}{|\Voca|-1} \right|.
\label{eq:bound-T1-second-part}
\end{align}
Here the inequality uses the Lipschitz continuity of $J(\cdot, \cdot)$.

We turn to focus on the right-hand side of \eqref{eq:bound-T1-second-part}.
We analyze the first term in \eqref{eq:bound-T1-second-part} using summation by parts.
We define 
\[
A_{\bar{\pi}, \token} =\sum_{j \in \Voca}  \left(P_{\bar{\pi}(j)} -  \frac{\Delta}{|\Voca|-1}\right)I_{\bar{\pi}, j}
\]
and make a convention that $A_{\pi, 0} = 0$ for any $\pi \in \Pi$.
We comment that $A_{\bar{\pi}, \token}$ can be viewed as a centered version of $a_{\bar{\pi}, \token}$.
It then follows that
\begin{align*}
\bigg|\sum_{\token \in \Voca}&J(a_{\bar{\pi}, \token}, \eta(\token))   (A_{\bar{\pi}, \token} -A_{\bar{\pi}, \token-1} ) \bigg| \\
&= \left|\sum_{\token=1}^{|\Voca|-1} \left[J(a_{\bar{\pi}, \token}, \eta(\token)) - J(a_{\bar{\pi}, \token+1}, \eta(\token+1))\right] A_{\bar{\pi}, \token}  + J(a_{\bar{\pi}, \token}, \eta(\token)) A_{\bar{\pi}, \token}\right| \\
&\le  \left[ \sum_{\token=1}^{|\Voca|-1}  \left| J(a_{\bar{\pi}, \token}, \eta(\token)) - J(a_{\bar{\pi}, \token+1}, \eta(\token+1))\right| + |J(a_{\bar{\pi}, \token}, \eta(\token)) | \right] \cdot \sup_{\token \in \Voca} |A_{\bar{\pi}, \token} | \\
&= O(1 )\cdot \sup_{\token \in \Voca} |A_{\bar{\pi}, \token} |.
\end{align*}
The last equation uses the boundedness and Lipschitz-continuity of $J(\cdot, \cdot)$ which implies
\begin{align*}
\sum_{\token=1}^{|\Voca|-1} & \left| J(a_{\bar{\pi}, \token}, \eta(\token)) - J(a_{\bar{\pi}, \token+1}, \eta(\token+1))\right| + |J(a_{\bar{\pi}, \token}, \eta(\token)) |\\
&\le O(1) \left[  \sum_{\token=1}^{|\Voca|-1} |a_{\bar{\pi}, \token}-a_{\bar{\pi}, \token+1}| +
\sum_{\token=1}^{|\Voca|-1}  |\eta(\token)-\eta(\token+1)| + 1
\right] = O(1).
\end{align*}
The second term in \eqref{eq:bound-T1-second-part} is smaller than $O(1) \cdot \sum_{\token \in \Voca} I_{\bar{\pi}} \left| a_{\pi,\token} -\frac{\token\Delta}{|\Voca|-1} \right|$.
Note that $\EB_{\bar{\pi}} I_{\bar{\pi}, \token} = \frac{|\Voca|-\token}{|\Voca|}$ for any $\token \in \Voca$.
Summarizing the above analysis for $\TM_1$, we have that  
\begin{align}
&\left| \EB_{\bar{\pi}} \sum_{\token \in \Voca} J(a_{\bar{\pi}, \token}, \eta(\token)) P_{\bar{\pi}(\token)} I_{\bar{\pi}, \token}-  \sum_{\token \in \Voca} J\left( \frac{\token\Delta}{|\Voca|-1}, \eta(\token)\right) \frac{|\Voca|-\token}{|\Voca|-1} \frac{\Delta}{|\Voca|}  \right| \nonumber \\
&\qquad \le O(1) \cdot \left[ \EB_{\bar{\pi}} \sup_{\token \in \Voca} |a_{\bar{\pi}, \token}|
+ \EB_{\bar{\pi}} \sup_{\token \in \Voca}  I_{\bar{\pi}, \token}\left|a_{\bar{\pi}, \token}  -  \frac{\token\Delta}{|\Voca|-1}\right| 
\right] \nonumber \\
&\qquad \le O(1) \cdot \left( \frac{1}{|\Voca|} + \sqrt{P_{(2)} \log |\Voca| } \right), \label{eq:bound-T1-third-part}
\end{align}
where the last inequality uses the following lemma.
\begin{lem}
\label{lem:inhomo-help1}
Define $I_{\pi, \token} = \1_{\{{\pi}^{-1}(1) > \token\} }$ and $A_{\bar{\pi}, \token} =\sum_{j \in \Voca}  \left(P_{\bar{\pi}(j)} -  \frac{\Delta}{|\Voca|-1}\right)I_{\bar{\pi}, j}$.
Then 
\begin{gather*}
\max\left\{\EB_{\bar{\pi}} \sup_{\token \in \Voca}  I_{\bar{\pi}, \token}\left|a_{\bar{\pi}, \token}  -  \frac{\token\Delta}{|\Voca|-1}\right|,\EB_{\bar{\pi}} \sup_{\token \in \Voca} \left|A_{\bar{\pi}, \token}\right|  \right\}
\le \sup_{\ell \in \Voca} 	\EB_{\bar{\pi}}  \left[ \sup_{\token \le \ell-1}  \left| a_{\bar{\pi}, \token} - \frac{\token\Delta}{|\Voca|-1}
\right| 
\bigg| \bar{\pi}(\ell) = 1
\right], \\
\sup_{\ell \in \Voca} 	\EB_{\bar{\pi}}  \left[ \sup_{\token \le \ell-1}  \left| a_{\bar{\pi}, \token} - \frac{\token\Delta}{|\Voca|-1}
\right| 
\bigg| \bar{\pi}(\ell) = 1
\right] \le \frac{1}{|\Voca|} +  \sqrt{\sum_{i=2}^{|\Voca|} P_{(i)}^2 \cdot c_0 \log (c_0|\Voca|)}  + P_{(2)} \cdot c_0 \log (c_0|\Voca|),
\end{gather*}
where $P_{(i)}$ is the $i$th largest probability in $\bP$ and $c_0$ is universal positive constant.
\end{lem}

Finally, we observe that 
\begin{align}
\label{eq:bound-T1-four-part}
&\left| \frac{1}{|\Voca|}\sum_{\token \in \Voca} J\left( \frac{\token\Delta}{|\Voca|-1}, \eta(\token)\right) \frac{|\Voca|-\token}{|\Voca|-1}  -  \int_0^1 (1-x) J(\Delta x, x) \rd x\right| \nonumber \\
&\qquad \le \left| \frac{1}{|\Voca|}\sum_{\token \in \Voca} J\left( \eta(\token) \Delta, \eta(\token)\right) (1-\eta(\token))  -  \int_0^1 (1-x) J(\Delta x, x) \rd x\right| +  \frac{O(1)}{|\Voca|} \nonumber \\
&\qquad \le \frac{O(1)}{|\Voca|} 
\end{align}
where the first inequality uses the Lipschitz continuity of $J(\cdot, \cdot)$ and the second inequality uses the Taylor expansion and boundedness of $J(\cdot, \cdot)$. 

Combining \eqref{eq:bound-T1-first-part}, \eqref{eq:bound-T1-third-part}, and \eqref{eq:bound-T1-four-part}, we know that the convergence that 
\[
\lim\limits_{|\Voca|\to \infty}\TM_1 = \Delta \cdot \int_0^1 (1-x) J(\Delta x, x) \rd x
\]
holds uniformly over $\bP \in \TPM$ and $\Delta \in [0, 1-\frac{1}{|\Voca|}]$ in the sense that
\[
\lim\limits_{|\Voca|\to \infty} \sup_{J \in \mathcal{J}} \sup_{\Delta \in [0, 1-\frac{1}{|\Voca|}]}\sup_{\bP \in \TPM} \left| \TM_1  - \Delta \cdot \int_0^1 (1-x) J(\Delta x, x) \rd x\right| = 0.
\]

\paragraph*{For the term $\TM_2$}
The analysis for $\TM_1$ is essentially the same as previously discussed for itself. We will directly present the result without providing the proof here:
\[
\lim\limits_{|\Voca|\to \infty}\TM_2 = \Delta \int_0^1x  J(\Delta x+1-\Delta, x)\rd x
\]
holds uniformly over $\bP \in \TPM$ and $\Delta \in [0, 1-\frac{1}{|\Voca|}]$ in the sense that
\[
\lim\limits_{|\Voca|\to \infty}\sup_{J \in \mathcal{J}} \sup_{\Delta \in [0, 1-\frac{1}{|\Voca|}]}\sup_{\bP \in \TPM} \left| \TM_2  - 
\Delta \int_0^1x  J(\Delta x+1-\Delta, x)\rd x
\right| = 0.
\]

\paragraph*{For the term $\TM_3$}

The analysis for $\TM_3$ is much simpler than $\TM_1$ and $\TM_2$. It follows that

\begin{align*}
\TM_3
&=  \frac{1}{|\Voca|!} \sum_{\token \in \Voca} \sum_{\pi \in\Pi_k}  \int_{a_{\pi,\token-1}}^{a_{\pi,\token}} J(r, \eta(\token)) \rd r \\
&= \EB_{\bar{\pi}}\sum_{\token \in \Voca}  \int_{a_{\bar{\pi}, \token-1}}^{a_{\bar{\pi}, \token}} J(r, \eta(\token)) \1_{ \{\bar{\pi}(\token)=1\}}  \rd r\\
&=\frac{1}{|\Voca|} \sum_{\token \in \Voca}  \EB_{\bar{\pi}}\left[   \int_{a_{\pi,\token-1}}^{a_{\pi,\token-1}+1-\Delta } J(r, \eta(\token)) \rd r  \big| \bar{\pi}(\token)=1  \right]
\end{align*}
It then follows that
\begin{align}
&\left|  
\TM_3 - \frac{1}{|\Voca|} \sum_{\token \in \Voca}  \int_{\frac{\token-1}{|\Voca|-1}\Delta}^{\frac{\token-1}{|\Voca|-1}\Delta+1-\Delta } J(r, \eta(\token)) \rd r 
\right|\nonumber \\
&\qquad \le O(1) \cdot \sup_{\token \in \Voca} \EB_{\bar{\pi}} \left[\left|a_{\pi,\token-1}-\frac{\token-1}{|\Voca|-1}\Delta\right| \bigg|  \bar{\pi}(\token)=1  \right] \nonumber \\
& \qquad \le O(1) \cdot \left( \frac{1}{|\Voca|} + \sqrt{P_{(2)} \log |\Voca| } \right).
\label{eq:bound-T1-five-part}
\end{align}
where the first inequality uses the boundedness of $J(\cdot, \cdot)$ and the second inequality uses Lemma \ref{lem:inhomo-help1}.
On the other hand, using the same analysis in \eqref{eq:bound-T1-four-part}, we have
\begin{align}
\label{eq:bound-T1-six-part}
&\left|\frac{1}{|\Voca|} \sum_{\token \in \Voca}  \int_{\frac{\token-1}{|\Voca|-1}\Delta}^{\frac{\token-1}{|\Voca|-1}\Delta+1-\Delta } J(r, \eta(\token)) \rd r 
-\int_0^1 \rd x \int_{\Delta x}^{\Delta x + 1- \Delta } J(r, x) \rd r \right| = \frac{O(1)}{|\Voca|}.
\end{align}
Combining \eqref{eq:bound-T1-five-part} and \eqref{eq:bound-T1-six-part}, we know that the convergence that
\[
\lim\limits_{|\Voca|\to \infty}\TM_3 =\int_0^1 \rd x \int_{\Delta x}^{\Delta x + 1- \Delta } J(r, x) \rd r 
\]
holds uniformly over $\bP \in \TPM$ and $\Delta \in [0, 1-\frac{1}{|\Voca|}]$ in the sense that
\[
\lim\limits_{|\Voca|\to \infty} \sup_{J \in \mathcal{J}} \sup_{\Delta \in [0, 1-\frac{1}{|\Voca|}]} \sup_{\bP \in \TPM} \left| \TM_3  - \int_0^1 \rd x \int_{\Delta x}^{\Delta x + 1- \Delta } J(r, x) \rd r  \right| = 0.
\]

\end{proof}

In the end, we provide the missing proofs for Lemma~\ref{lem:inhomo-help1}.

\begin{proof}[Proof of Lemma~\ref{lem:inhomo-help1}]	
We first simplify the target expectations.
We note that
\begin{align*}
\EB_{\bar{\pi}} \sup_{\token \in \Voca}  I_{\bar{\pi}, \token}\left|a_{\bar{\pi}, \token}  -  \frac{\token\Delta}{|\Voca|-1}\right| 
&= \frac{1}{|\Voca|} \sum_{\ell =1}^{|\Voca|} 		\EB_{\bar{\pi}}  \left[ \sup_{\token \in \Voca}  I_{\bar{\pi}, \token}\left|a_{\bar{\pi}, \token}  -  \frac{\token\Delta}{|\Voca|-1}\right| 
\bigg| \bar{\pi}(\ell) = 1
\right]\\
&= \frac{1}{|\Voca|} \sum_{\ell =1}^{|\Voca|} 		\EB_{\bar{\pi}}  \left[ \sup_{\token \le \ell-1}  \left|a_{\bar{\pi}, \token}  -  \frac{\token\Delta}{|\Voca|-1}\right| 
\bigg| \bar{\pi}(\ell) = 1
\right]\\
&\le \sup_{\ell \in \Voca} 	\EB_{\bar{\pi}}  \left[ \sup_{\token \le \ell-1}  \left|a_{\bar{\pi}, \token}  -  \frac{\token\Delta}{|\Voca|-1}\right| 
\bigg| \bar{\pi}(\ell) = 1
\right],
\end{align*}
and
\begin{align*}
\EB_{\bar{\pi}} \sup_{\token \in \Voca} \left|\sum_{j \in \Voca}  \left(P_{\bar{\pi}(j)} -  \frac{\Delta}{|\Voca|-1}\right)I_{\bar{\pi}, j}\right| 
&= \frac{1}{|\Voca|} \sum_{\ell =1}^{|\Voca|} 		\EB_{\bar{\pi}}  \left[ \sup_{\token \in \Voca}  \left|\sum_{j \in \Voca}  \left(P_{\bar{\pi}(j)} -  \frac{\Delta}{|\Voca|-1}\right)I_{\bar{\pi}, j}\right| 
\bigg| \bar{\pi}(\ell) = 1
\right]\\
&= \frac{1}{|\Voca|} \sum_{\ell =1}^{|\Voca|} 		\EB_{\bar{\pi}}  \left[ \sup_{\token \le \ell-1} \left|\sum_{j \in \Voca}  \left(P_{\bar{\pi}(j)} -  \frac{\Delta}{|\Voca|-1}\right)\right| 
\bigg| \bar{\pi}(\ell) = 1
\right]\\
&\le \sup_{\ell \in \Voca} 	\EB_{\bar{\pi}}  \left[ \sup_{\token \le \ell-1}  \left| a_{\bar{\pi}, \token} - \frac{\token\Delta}{|\Voca|-1}
\right| 
\bigg| \bar{\pi}(\ell) = 1
\right].
\end{align*}
Hence, it suffices to focus on $\sup_{\ell \in \Voca} 	\EB_{\bar{\pi}}  \left[ \sup_{\token \le \ell-1}  \left| a_{\bar{\pi}, \token} - \frac{\token\Delta}{|\Voca|-1}
\right| 
\bigg| \bar{\pi}(\ell) = 1
\right]$.

Note that given the condition $ \bar{\pi}(\ell) = 1$, $\bar{\pi}$ is still a random permutation on the resting $|\Voca|-1$ numbers which permutes $\Voca-\{\ell\}$ to $\Voca-\{1\}$.
To that end, we will use the concentration inequality for randomly permuted sums developed by \citep{albert2019concentration}.

\begin{lem}[Proposition 2.2 in~\citep{albert2019concentration}]
\label{lem:permu-concentration}
Let $b_{j, l} \ge 0$ be a collection of non-negative numbers and $\bar{\pi}$ be a
random uniform permutation in $\Voca$, that is, $\bar{\pi} \sim \UM(\Voca)$.
Let $Z = \sum_{j \in \Voca} b_{j, \bar{\pi}(j)}$. Then for any $x > 0$,
\[
\PB\left( |Z - \EB_{\bar{\pi}} Z|\ge 2 \sqrt{ \left( \frac{1}{|\Voca|}\sum_{j,l=1}^{|\Voca|} b_{j,l}^2 \right) \cdot x } + \sup_{1 \le j,l\le |\Voca|} b_{j,l} \cdot x \right)
\le 8 \re^{1/16} \exp\left(- \frac{x}{16} \right).
\]
\end{lem}

To apply Lemma \ref{lem:permu-concentration}, we set $b_{j,l} = P_l \cdot
\1_{ \{j \le \token, l \neq 1\}}$.
Hence, with $Z = \sum_{j \in \Voca} b_{j,\bar{\pi}(j)}$, we have $Z = a_{\bar{\pi}, \token}$ and once $\token < \ell$, 
\[
\EB_{\bar{\pi}} [Z|\bar{\pi}(\ell)=1]
=\EB_{\bar{\pi}} \left[\sum_{j \in \Voca} P_{\bar{\pi}(j)}  \bigg|\bar{\pi}(\ell)=1\right]
=  \frac{\token \Delta}{|\Voca|-1}.
\]
Note that $\sup_{j, l} b_{j,l} \le P_{(2)}$ and $\sum_{j,l=1}^{|\Voca|} b_{j,l}^2 \le \token \cdot \sum_{j=2}^{|\Voca|} P_{(j)}^2 $.
As a result of  Lemma~\ref{lem:permu-concentration}, we know that for any $\ell \in \Voca$, when conditioning on $\bar{\pi}(\ell)=1$, there exists a universal constant $c_0>0$ such that for any $\delta > 0$ and any $k < \ell$, with probability at least $1 - \delta$,
\[
\left|a_{\bar{\pi}, \token} -\frac{\token \Delta }{|\Voca|-1} \right|
\le  \sqrt{\frac{\token}{|\Voca|} \sum_{j=2}^{|\Voca|} P_{(j)}^2 \cdot c_0 \log \frac{c_0}{\delta}}  + P_{(2)} \cdot c_0 \log \frac{c_0}{\delta}.
\]
Taking a union bound, setting $\delta =1/|\Voca|^2$, and slightly modifying the value of $c_0$, we then have
\[
\EB\sup_{\token < \ell} \left[ \left|a_{\bar{\pi}, \token} -\frac{\token \Delta }{|\Voca|-1}\right|
\bigg| \bar{\pi}(\ell)=1\right]
\le \frac{1}{|\Voca|} +  \sqrt{\sum_{j=2}^{|\Voca|} P_{(j)}^2 \cdot c_0 \log (c_0|\Voca|)}  + P_{(2)} \cdot c_0 \log (c_0|\Voca|).
\]
Taking the maximum over $\ell \in \Voca$, we then finish the proof.
\end{proof}

\subsection{Proof of Theorem~\ref{thm:asymptotic-dif}}

With Lemma \ref{lem:exact-CDF-dif} and Corollary \ref{cor:J-expectation}, we are ready to prove Theorem \ref{thm:asymptotic-dif}.

\begin{proof}[Proof of Theorem~\ref{thm:asymptotic-dif}]

By Lemma \ref{lem:exact-CDF-dif} and the definition of integral, it follows that for any $r \in [0, 1]$,
\begin{align*}
\lim_{|\Voca| \to \infty}
\PB_{H_0}(Y_t^{\dif} \le r) 
&=\frac{1}{|\Voca|}\sum_{i=1}^{|\Voca|} \left[ 
\min\left\{\eta(i)+r,1\right\}-\max\{\eta(i)-r, 0\}
\right]\\
&= \int_0^1 [ \min\{x+r,1\}- \max\{x-r,0\}] \rd x\\
&= 1-(1-r)^2.
\end{align*}

On the other hand, by setting $J(x, y) = \1_{\{|x-y| \le r\}}$ in Corollary \ref{cor:J-expectation}, it follows that for any $r \in [0, 1-\Delta]$,
\begin{align*}
\lim_{|\Voca| \to \infty}
\PB_{H_1}(Y_t^{\dif} \le r|\bP_t) 
&= 2\Delta \int_0^1(1-x) 
\1_{\{(1-\Delta)x\le r\}}
\rd x
+\int_0^1 \rd y
\int_{\Delta y}^{\Delta y + 1-\Delta} \1_{\{|x-y| \le r\}} \rd x\\
&= 2 \Delta \int^{\frac{r}{1-\Delta}\wedge 1 }_0(1-x) \rd x + (1-\Delta) - 2\int^{1}_0 ((1-\Delta)x-r)_+ \rd x\\
&=1-\left(1 - \frac{r}{1-\Delta} \right)^2.
\end{align*}
\end{proof}

\paragraph*{Visualization of weak convergence on an extreme case}
To better understand the weak convergence established in Theorem \ref{thm:asymptotic-dif}, we study a special case where $\Delta \to 0$ and obtain the following result.
\begin{cor}
\label{cor:CDF-exreame-case}
Let $P_{t,(1)} = \max_{\token} P_{t, \token}$ denote the largest probability in $\bP_t$.  For any $r \in \RB$,
\[
\lim_{\substack{|\Voca| \to \infty\\ \log |\Voca| \cdot P_{t,(1)} \to 0 }} 
\PB_{H_1}(Y_t^{\dif} \le r|\bP_t)  = \1_{\{r \ge 0\}}.
\]
\end{cor}

Let $F_{0}^{\dif}(r) = \PB(Y_t^{\dif} \le r|H_0) $ and $F_{1, \bP_t}^{\dif}(r) = \PB(Y_t^{\dif} \le r|\bP_t, H_1)$ denote the finite-vocabulary null and alternative CDFs respectively.
For any $r \in [0, 1]$, Lemma \ref{lem:exact-CDF-dif} shows $F_{0}^{\dif}(r) \to 1-(1-r)^2$ as $|\Voca| \to \infty$ and Corollary \ref{cor:CDF-exreame-case} shows $F_{1, \bP_t}^{\dif}(r) \to \1_{\{r \ge 0\}}$ if $\log |\Voca| \cdot P_{t,(1)} \to 0$.
It means that $\Ydif_t$ degenerates to a point mass distribution where all of the probability mass is concentrated on the original point. 

\begin{figure}[t!]
\centering
\includegraphics[width=1.0\textwidth]{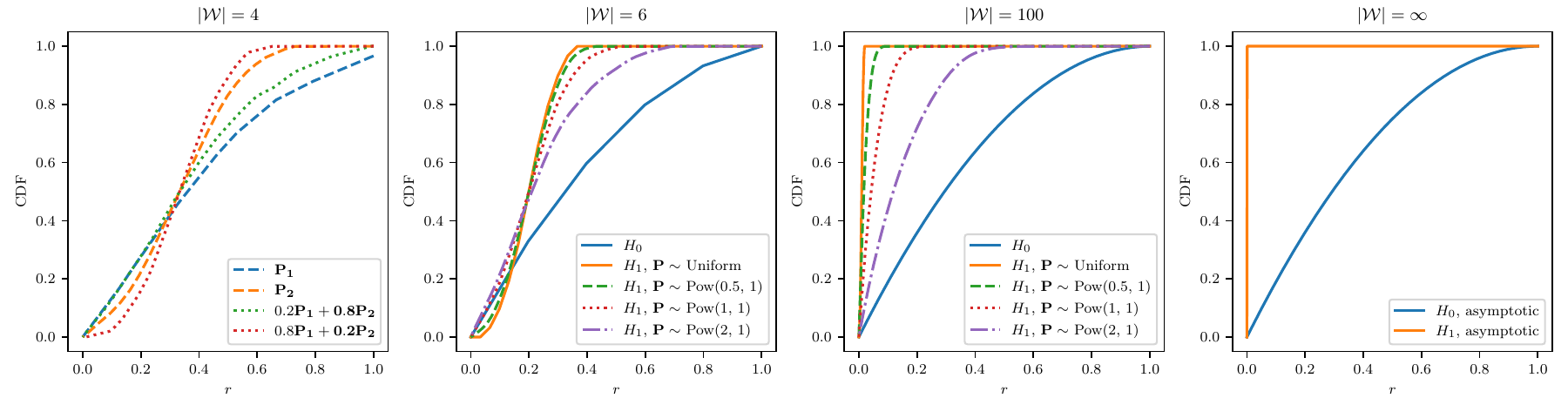}
\caption{Exact CDFs of $\Ydif_t$ under hypotheses $H_0$ and $H_1$, different NTP distributions $\bP_t$, and different vocabulary size $|\Voca|$.
The first panel shows $F_{1,\bP}^{\dif}(r)$ is not convex in $\bP$ for most values of $r \in [0, 1]$.
Here $\bP_1=(0.1,0.2,0.3,0.4)$ and $\bP_2=(0.8, 0.2, 0, 0)$.
The remaining three panels consider $|\Voca|=6$ (left column), $|\Voca|=100$ (middle column), and $|\Voca|=\infty$ (right column).
We say $\bP_t$ is $\mathrm{Pow}(a, b)$ if $P_{t,\token} \propto (\token+b)^{-a}$ and $\bP_t$ is uniform if $P_{t,\token} = \frac{1}{|\Voca|}$ for any $\token \in \Voca$.
}
\label{fig:inverse-pdf}
\end{figure}

Figure~\ref{fig:inverse-pdf} illustrates the CDF $F_{1, \bP_t}^{\dif}$ under different hypotheses, token distributions, and vocabulary sizes.
We observe that when $|\Voca|=100$, $F_{0}^{\dif}$ already align with its asymptotic null CDF well. 
Considering that in practice $|\Voca|$ typically ranges from $3 \times 10^4$ to $5 \times 10^4$, we can confidently replace the finite-vocabulary null CDF$F_{0}^{\dif}$ with the asymptotic one, that is $1-(1-r)_{[0, 1]}^2$. This stance contrasts with what \citep{piet2023mark} suggests, implying that the null hypothesis distribution is not that complex to compute from an asymptotic perspective. This closed-form distribution enables us to explicitly compute the critical value $\gamma_{n,\alpha}$ using the empirical estimate $\widehat{\gamma}_{n,\alpha}$, rather than resorting to simulation \citep{kuditipudi2023robust}.

However, whether the finite-vocabulary alternative CDF, that is, $F_{1, \bP_t}^{\dif}$, converges to its limit CDF depends on whether the condition $\log|\Voca| \cdot P_{t,(1)}=o(1)$ holds. For example, if $\bP_t \sim \mathrm{Pow}(2, 1)$, then $\log |\Voca| \cdot P_{t,(1)} =\Omega(1)$, and consequently, the finite-vocabulary alternative CDF does not converge to the limit (see the purple dotted curves in Figure \ref{fig:inverse-pdf}). In contrast, if $\bP_t$ follows $\mathrm{Pow}(a, b)$ with $a < 1$, then $\log|\Voca| \cdot P_{t,(1)}=o(1)$ is satisfied, so the finite-vocabulary CDF aligns well with its limit even if $|\Voca| = 100$.

\subsection{Proof of Lemma~\ref{lem:key-prop-inverse}}
\label{proof:optimal-score-dif}

In the following, we provide the omitted proof for lemmas in Section \ref{sec:proof-propp-key-prop-inverse}, namely 
Lemma \ref{lemma:super-technical-exchange-order} and Lemma \ref{lem:closed-form-bar-L}.

\begin{proof}[Proof of Lemma \ref{lemma:super-technical-exchange-order}]
The exchangeability is a direct consequence of Lemma \ref{lem:J}.
It suffices to fucos on $\phi_{\bP, h}(1) =\EB_{1, \bP} \re^{-h(\Ydif)}$.
We apply Lemma \ref{lem:J} by setting $J(x, y) = \exp(- h(|x-y|))$. It is easy to check this particular choice of $J$ is Lipschitz-continuous given that $h$ is Lipschitz-continuous.
Then for any $\bP \in \TPM$, we know that
\[
\lim_{|\Voca|\to \infty} \phi_{\bP, h}(1)
\]
exists and depends only on $\Delta$ and $h$.
The uniform convergence in Lemma \ref{lem:J} implies that we must have
\begin{equation}
\label{eq:uniform-convergence-h}
\lim_{|\Voca| \to \infty} \sup_{\Delta' \ge \Delta} \sup_{\bP \in \TPMo}
\left|\phi_{\bP, h}(1)  - \lim_{|\Voca|\to \infty} \phi_{\bP, h}(1) \right| = 0.
\end{equation}
As a result, we must have
\begin{align*}
\lim_{|\Voca| \to \infty} 
&\left|\sup_{\Delta' \ge \Delta} \sup_{\bP \in \TPMo}\phi_{\bP, h}(1) - \sup_{\Delta' \ge \Delta} \sup_{\bP \in \TPMo}\lim_{|\Voca|\to \infty} \phi_{\bP, h}(1) \right| \\
&\le \lim_{|\Voca| \to \infty} \sup_{\Delta' \ge \Delta} \sup_{\bP \in \TPMo}
\left|\phi_{\bP, h}(1)  - \lim_{|\Voca|\to \infty} \phi_{\bP, h}(1) \right| = 0.
\end{align*}
Therefore, it follows that
\begin{equation}
\label{eq:explain_exchange_1}
\lim_{|\Voca| \to \infty} 
\sup_{\Delta' \ge \Delta} \sup_{\bP \in \TPMo}\phi_{\bP, h}(1) = \sup_{\Delta' \ge \Delta} \sup_{\bP \in \TPMo}\lim_{|\Voca|\to \infty} \phi_{\bP, h}(1).
\end{equation}

Note that by the uniform convergence \eqref{eq:uniform-convergence-h}, we have for any fixed $\Delta'$,
\[
\lim_{|\Voca| \to \infty} \sup_{\bP \in \TPMo}
\left|\phi_{\bP, h}(1)  - \lim_{|\Voca|\to \infty} \phi_{\bP, h}(1) \right| = 0
\]
so that 
\begin{equation}
\label{eq:exchange-sup-lim}
\lim_{|\Voca|\to \infty} \sup_{\bP \in \TPMo} \phi_{\bP, h}(1) = \sup_{\bP \in \TPMo}\lim_{|\Voca|\to \infty} \phi_{\bP, h}(1).
\end{equation}
The equations \eqref{eq:explain_exchange_1} and \eqref{eq:exchange-sup-lim} imply that
\[
\lim_{|\Voca| \to \infty} \sup_{\Delta' \ge \Delta} \sup_{\bP \in \TPMo}\phi_{\bP, h}(1)
=\sup_{\Delta' \ge \Delta}  \lim_{|\Voca|\to \infty} \sup_{\bP \in \TPMo} \phi_{\bP, h}(1).
\]
Given that the limit exists, the $\limsup$ also exists and equals the limit.
\end{proof}

\begin{proof}[Proof of Lemma \ref{lem:closed-form-bar-L}]
As implied in the proof of Lemma \ref{lemma:super-technical-exchange-order}, we can replace $\limsup$ with $\lim$.
By Theorem \ref{thm:asymptotic-dif}, we already have
\[
\lim_{|\Voca| \to \infty} \EB_0 h(\Ydif) = \int_{[0, 1]} h(r) F_{\dif, 0} (\rd r).
\]
It then suffices to find the expression of $\lim\limits_{|\Voca| \to \infty} \phi_{\bP, h}(1)$.

Note that for any $\bP \in \TPM$, it follows from integration by parts that
\begin{equation}
\label{eq:help-phi}
\phi_{\bP, h}(1) 
= \int_{0}^1 \re^{- h(r)} F_{1, \bP}^{\dif}(\rd r)
= \re^{- h(1)} +  \int_0^1 
F_{1, \bP}^{\dif}(r) \re^{- h(r)} h (\rd r),
\end{equation}
where $F_{1, \bP}^{\dif}$ is the alternative CDF introduced in Lemma \ref{lem:exact-CDF-dif}.
Theorem \ref{thm:asymptotic-dif} implies the following weak convergence: for each $r \in [0, 1]$ and $\bP \in \TPM$, as $|\Voca| \to \infty$,
\[
F_{1, \bP}^{\dif}(r) \overset{d}{\to} F_{\dif, \Delta}(r)   
\]
where $F_{\dif, \Delta}(r) = 1 - (1-\frac{r}{1-\Delta})_{[0, 1]}^2$ is the limit CDF under $H_1$.
The weak convergence holds uniformly in the sense that
\begin{equation}
\label{eq:uniform-convergence-F}
\lim_{|\Voca|\to \infty}
\sup_{r \in [0, 1]}
|F_{\dif, \Delta}(r) - F_{1, \bP}^{\dif}(r)| = 0.
\end{equation}
This is because weak convergence in probability implies uniform convergence in cumulative distribution functions (see Lemma \ref{lem:uniform-convergence}).
\begin{lem}
\label{lem:uniform-convergence}
Let $\mu_n$ and $\mu$ be probability measures on $\RB$ with CDFs given by $F_n$ and $F$ respectively.
If $\mu_n$ weakly converges to $\mu$ and $F$is continuous, then $F_n$ converges uniformly to $F$ on $\RB$.
\end{lem}

Then, by \eqref{eq:exchange-sup-lim}, \eqref{eq:help-phi}, and \eqref{eq:uniform-convergence-F}, it follows that
\[
\lim_{|\Voca| \to \infty}\sup_{\bP \in \TPM}
\phi_{\bP, h}(1) =
\lim_{|\Voca| \to \infty}
\phi_{\bP, h}(1) 
= \re^{- h(1)} +  \int_0^1 
F_{\dif, \Delta}(r) \re^{- h(r)} h (\rd r)
=\int_0^1   \re^{- h(r)}F_{\dif, \Delta}(\rd r)
\]
where the last equation uses the integration by parts.
\end{proof}

\subsection{Optimal Threshold Value}

We set $\delta=\re^{-1}$ for the indicator score function $\hind$ because this particular value $\re^{-1}$ maximizes its $\Rlimit$-efficiency when $\Delta \to 1$ as the following lemma shows.

\begin{lem}\label{lem:ind_optim}
\[
\re^{-1} = \arg\max\limits_{\delta \in [0, 1]} \lim_{\Delta \to 1} (1-\Delta)\Rlimit(\hind(\Yars)).
\]
\end{lem}

To prove this lemma, we first derive the limit distribution of $\Yars$ under the same setting in Section \ref{sec:inverse}.
\begin{lem}[Asymptotics of diverging $|\Voca|$]
\label{lem:asymptotic-ars}
Recall the auxiliary belief class:
\[
\TPM=\left\{\bP: \max_{\token \in \Voca} P_{\token}=1-\Delta, \log |\Voca|\cdot P_{(2)} \le \varepsilon_{|\Voca|}\right\},
\]
where $P_{(2)}$ is the second largest probability in $\bP$.
It follows that
\[
\lim_{|\Voca| \to \infty} \sup_{\Delta \in [0, 1-\frac{1}{|\Voca|}]}
\sup_{\bP\in\TPM} |F_{1, \bP}(r)-(1-\Delta)r^{1/(1-\Delta)}-\Delta\1_{\{r=1\}}|= 0
~~\text{for any}~~ r \in [0, 1].
\]
\end{lem}

\begin{proof}[Proof of Lemma \ref{lem:asymptotic-ars}]

First, consider the case $r\in[0,1)$. 
Note that the CDF of $\Yars$ satisfies that 
\begin{equation*}
\begin{aligned}
F_{1, \bP}(r)&=\sum_{\token\in\Voca}P_\token r^{1/P_\token}=(1-\Delta)r^{1/(1-\Delta)}+\sum_{\token\ne\argmax_jP_j}P_\token r^{1/P_\token}.
\end{aligned}
\end{equation*}
Moreover, since $P_{(2)}\log|\Voca|\to0$, we have
\[
\biggl|\sum_{\token\ne\argmax_jP_j}P_\token r^{1/P_\token}\biggr|\le P_{(2)}|\Voca|r^{1/P_{(2)}}=P_{(2)}\exp\biggl\{\frac{1}{P_{(2)}}\log r+\log|\Voca|\biggr\}\to0
\]
for any fixed $r\in[0,1)$. For $r=1$,  
\[
\sum_{\token\ne\argmax_jP_j}P_\token r^{1/P_\token}=\sum_{\token\ne\argmax_jP_j}P_\token =\Delta.
\]
Then our claim holds.  
\end{proof}

With Lemma \ref{lem:asymptotic-ars}, we can prove Lemma \ref{lem:ind_optim}.

\begin{proof}[Proof of Lemma \ref{lem:ind_optim}]
For any fixed $\Delta$, it follows from Lemma \ref{lem:asymptotic-ars} that
\[
\Rlimit(\hind(\Yars))=\delta\log\biggl(\frac{\delta^{-\frac{\Delta}{1-\Delta}}}{1-\Delta}\biggr)+(1-\delta)\log\biggl(\frac{1-\delta}{1-(1-\Delta)\delta^{\frac{1}{1-\Delta}}}\biggr).
\]
By direct calculation, we have
\[
\lim_{\Delta\to1}(1-\Delta)\Rlimit(\hind(\Yars))=-\delta\log\delta.
\]
Let $f(\delta) = -\delta \log \delta$.
Then $f'(\delta)=-\log\delta-1$ and $f{''}(\delta)=-1/\delta<0$ on $\delta\in[0,1]$. It means that $f(\delta)$ is maximized at $\delta=1/e$.
\end{proof}

\section{Details of Experiments}
\label{append:experiments}

\subsection{Simulation Setup}
\label{sec:exp-setup}

\paragraph*{Choice of pseudorandom variable}

We employ a context window of size $m=5$, allowing the randomness variable $\xi_t = \AM(s_{(t-m):(t-1)}, \Key)$ to depend on the previous $m$ tokens. For the hash function $\AM$, at each step $t$, we concatenate the $m$ preceding tokens with $\Key$, that is, $(\token_{(t-m):(t-1)}, \Key)$, to generate a random seed. This seed is then used with a pseudo-random number generator, specifically the PCG-64 generator \citep{o2014pcg}, which is the default option in Python's \textsf{NumPy} package \citep{harris2020array}.

\paragraph*{Computation of critical values}

Note that $\EB_0\hoptdif(\Ydif) = - \infty$, indicating that the central limit theorem cannot be applied to calculate the critical value for $\hoptdif(\Ydif)$. To determine the critical value for $\hoptdif$, we resort to simulation. For each $n$, we generate $n$ i.i.d. copies of $\hoptdif(\Ydif_t)$ and calculate the sum $\sum_{t=1}^n \hoptdif(\Ydif_t)$. This procedure is replicated 500 times, using the empirical $1-\alpha$ quantile as an initial estimate. To enhance the precision of this estimate, we repeat the process 10 times and average these ten initial estimates to establish the final critical value. For other score functions, the central limit theorem allows us to estimate the critical value as:
\[
\hat{\gamma}_{n,\alpha} = n \cdot \EB_0 h (Y) +  \Phi^{-1}(1-\alpha) \cdot \sqrt{n \cdot \Var_0(h(Y))}.
\]
\new{
This estimate is consistent in the sense that we have $\lim_{n \to \infty} \gamma_{n, \alpha}/\hat{\gamma}_{n, \alpha} \to 1$ for any $\alpha \in (0, 1)$ and any score function satisfying $\EB_0 |h(Y)|^2 < \infty$.

\begin{rem}
Hash collisions occur when the computed pseudorandom numbers are not truly i.i.d. 
It makes the way using $\hat{\gamma}_{n, \alpha}$ to control Type I error become invalid.
This issue is well illustrated in Figure 1 of \cite{fernandez2023three}, where it is shown that the empirical false positive rates for common watermarks often exceed theoretical predictions. The primary reason for this discrepancy is that their pseudorandom numbers rely only on local information, that is, $\zeta_t = \AM(w_{(t-m):(t-1)}, \Key)$ in our notation. Consequently, if there exists any $t \neq t'$ such that $w_{(t-m):(t-1)} = w_{(t'-m):(t'-1)}$, then $\zeta_t = \zeta_{t'}$, leading to the same pseudorandom numbers appearing multiple times. This violates the assumption of true i.i.d. nature and thus our working hypothesis.
To mitigate this, \cite{fernandez2023three} proposed heuristic methods to address such repetitions (see their Part C in Section III). By implementing these methods, they were able to improve Type I error control, as demonstrated in the rightmost panel of Figure 1.

However, in our paper, we did not consider the issue of hash collisions, as this issue is orthogonal to the main results of our paper, and addressing it is beyond the scope of our current work. Instead, our approach assumes $\zeta_t = \AM(w_{1:(t-1)}, \Key)$, which effectively corresponds to the case where $m=\infty$. Given that the hash function $\AM$ is highly sensitive to its inputs, even slight differences in these inputs result in outputs that are statistically independent. Therefore, as long as $t \neq t'$, we have $w_{1:(t-1)} \neq w_{1:(t'-1)}$, ensuring that $\zeta_t \perp \zeta_{t'}$. This implies that our pseudorandom numbers are truly i.i.d., allowing us to confidently use $\hat{\gamma}_{n,\alpha}$ to control the Type I error.
% We argue that if pseudorandom numbers are truly i.i.d., using $\hat{\gamma}_{n, \alpha}$ to control Type I error is valid. However, Figure 1 in \citep{fernandez2023three} shows that this method often results in empirical false positive rates for common watermarks exceeding theoretical predictions. This discrepancy arises because their pseudorandom numbers rely on local information, that is, $\zeta_t = \AM(w_{(t-m):(t-1)}, \Key)$, which makes them not truly i.i.d. For example, if $w_{(t-m):(t-1)} = w_{(t'-m):(t'-1)}$ for some $t \neq t'$, then $\zeta_t = \zeta_{t'}$, leading to repeated values and violating the i.i.d. assumption. To address this, \cite{fernandez2023three} introduced heuristic methods to reduce such repetitions (see their Part C in Section III), which improved Type I error control. 

% In contrast, our approach assumes $\zeta_t = \AM(w_{1:(t-1)}, \Key)$, effectively setting $m = \infty$. The hash function $\AM$ is highly sensitive, ensuring that even minor input differences result in statistically independent outputs. Since $w_{1:(t-1)} \neq w_{1:(t'-1)}$ for any $t \neq t'$, we ensure $\zeta_t \perp \zeta_{t'}$, making our pseudorandom numbers truly i.i.d. This allows us to confidently use $\hat{\gamma}_{n,\alpha}$ to control the Type I error.
\end{rem}
}

\paragraph*{Additional results for other choices of \texorpdfstring{$b$}{b}}
Additional results are compiled in Figure~\ref{fig:inhomo-simu-appendix}. The performance ranking is consistent with that observed in Figure \ref{fig:inhomo-simu}. Notably, as there's a higher chance for $\Delta$ to assume larger values (e.g., $\Delta \sim \UM(0.001, 0.7)$ compared to $\Delta \sim \UM(0.001, 0.1)$), fewer tokens are needed to achieve a comparable Type II error rate.

\begin{figure}[t!]
\centering
\includegraphics[width=1.0\textwidth]{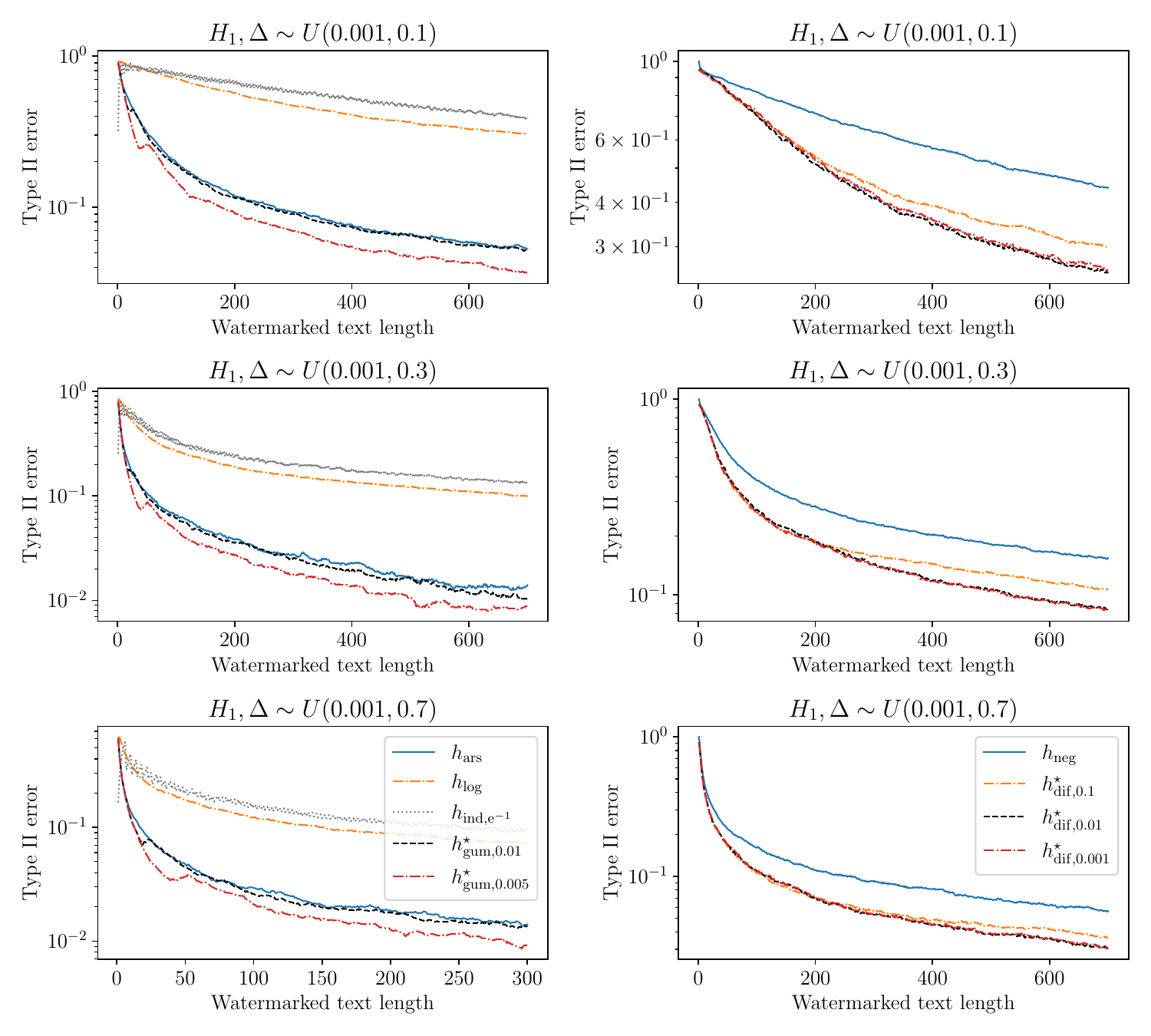}
\caption{Empirical type II errors versus the text length on simulated datasets for the Gumbel-max watermarks (left column) and inverse transform watermark (right column).
}
\label{fig:inhomo-simu-appendix}
\end{figure}

\subsection{Real-world Experimental Details}
\label{sec:real-data-detail}
\label{appen:extened-results}

\paragraph*{Experiment setup}
The experimental setup we employed is largely based on the methodology described in Appendix D.1 of \citep{kuditipudi2023robust}. 
In our approach, each generation is conditioned on a prompt which is obtained by sampling documents from the news-oriented segment of the C4 dataset \citep{raffel2020exploring}.
We enforce a minimum prompt size of 50 tokens in all experiments and skip over any document that is not long enough.
Note that retokenization may not precisely match the original tokens. Therefore, to guarantee that the verifier consistently receives at least $n$ tokens, we augment its input with special padding tokens, which vary according to the tokenizer of each model. Additionally, to mitigate the need for padding, we initially generate many buffer tokens beyond $n$. 
We set the number of buffer tokens to be 20 in every experiment.
This additional buffer typically makes padding unnecessary.

The hashing function $\AM$ is the \texttt{Skip} function from \citep{kirchenbauer2023reliability} with the window size $m=4$.
This choice is motivated by the observations of Figure 2 in \citep{kirchenbauer2023reliability} that the hash function \texttt{Skip} with context width $c=4$ has a good balance between generation diversity and adversarial robustness.
The way we use to determine critical values for different score functions mirrors the approach taken in our simulations.

\paragraph*{Extended results for the Sheared-LLaMA-2.7B model}

\begin{figure}[t!]
\centering
\includegraphics[width=\textwidth]{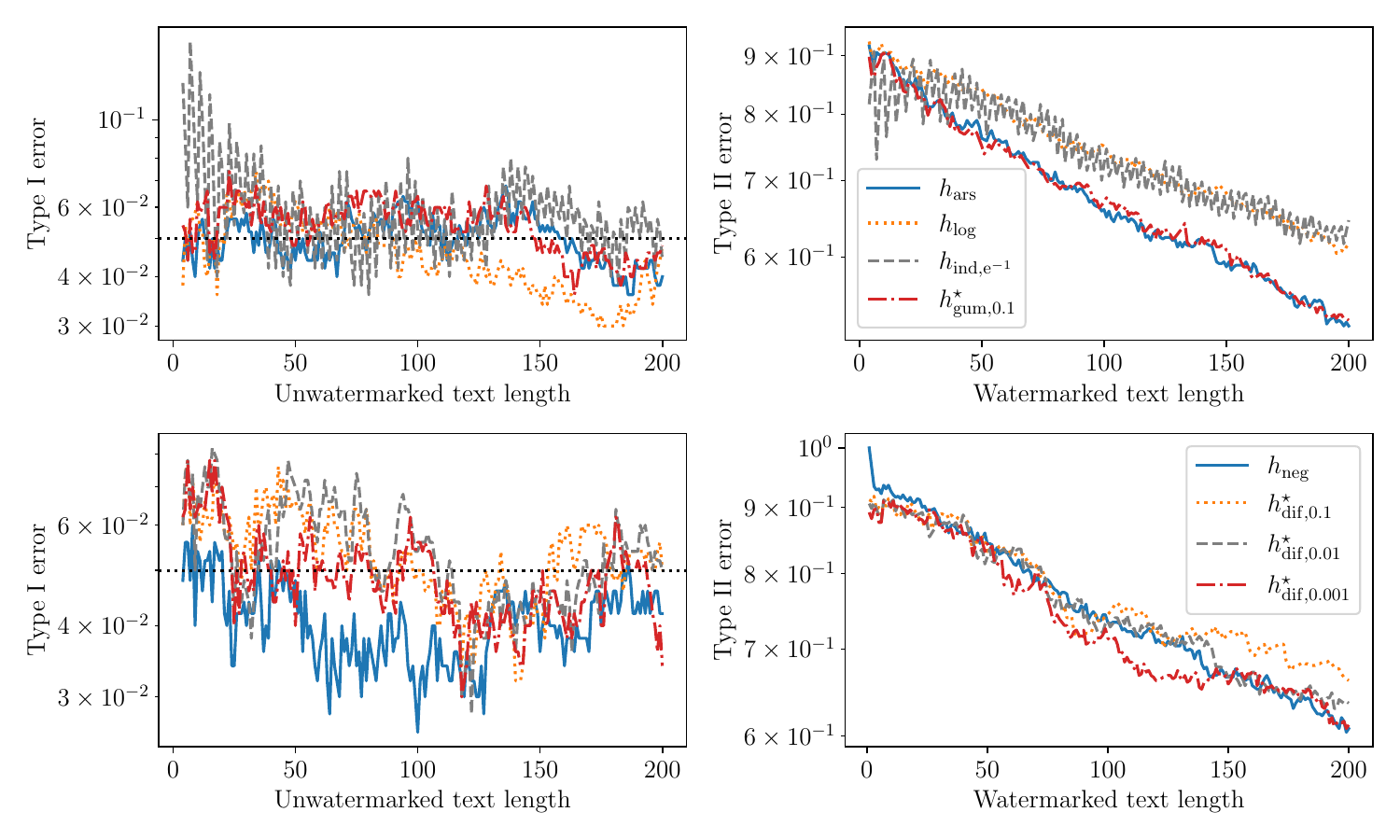}
\caption{
Type I (first column) and type II errors (second column) versus the text length for Sheared-LLaMA-2.7B. 
Each curve is averaged over 500 samples from the C4 dataset.
Top row: Gumbel-max watermark. Bottom row: Inverse transform watermark.
}
\label{fig:c4-2p7B}
\end{figure}

In our real-world experiments, we evaluate two models: OPT-1.3B \citep{zhang2022opt} and Sheared-LLaMA-2.7B \citep{xia2023sheared}. For both models, we adjust the temperature parameter to $0.1$. Due to limited space, the experimental outcomes related to Sheared-LLaMA-2.7B \citep{xia2023sheared} are documented in the appendix (refer to Figure \ref{fig:c4-2p7B}). It will become clear that most patterns observed are consistent with those found in the OPT-1.3B experiments, which are elaborated on in the main text of the paper. It is worth noting that our introduced score functions continue to exhibit comparable effectiveness with existing methods in the context of the Sheared-LLaMA-2.7B model.

\subsection{Details of Figure \ref{fig:empirical-delta}}
\label{detail:empirical-delta}
We prompted {ChatGPT-4} to generate a list of twenty open-ended and engaging questions to foster thought-provoking conversations. 
These questions, detailed below, encompass a broad spectrum of topics, including personal growth, societal challenges, and theoretical future scenarios.
Subsequently, we employed these questions as prompts, directing {ChatGPT-3.5-turbo} to respond by forecasting the next token in the sequence where the temperature parameter is set to be $1$.
 We gathered the highest probability predictions $\max_{\token} P_{t, \token}$ across various steps $t$ and different questions.
 About 5000 points of $\max_{\token} P_{t, \token}$ have been recorded, and a frequency histogram is plotted in Figure \ref{fig:empirical-delta}.

We list the considered twenty questions below.
\begin{enumerate}
    \item  What do you think are the most significant changes humanity will face in the next 50 years?
\item If you could gain one quality or ability that you currently don't have, what would it be and why?
\item What book or movie has profoundly impacted your view of the world, and in what way?
\item  If you could live in any historical period, when would it be and why?
\item  What do you believe is the most important issue facing the world today, and how would you propose we address it?
\item  If you could instantly become an expert in any subject or skill, what would it be and why?
\item  How do you think technology will affect human relationships in the future?
\item  If you could solve one unsolved mystery, which one would you choose and why?
\item  What does success mean to you, and do you feel you've achieved it?
\item If you had the power to change one law or policy in your country, what would it be and why?
\item  What's a belief you had as a child that you've completely changed your opinion on?
\item If you could have a conversation with any person from history, who would it be and what would you ask them?
\item  How do you think the concept of work will evolve in the next 100 years?
\item If you could witness any event in history, what would it be and why?
\item What's something that you've learned about yourself during a difficult time?
\item  How do you define happiness, and what makes you happy?
\item What's one piece of advice you would give to your younger self?
\item  If you could invent something that would make life easier for people, what would it be?
\item How do you think education will change in the future?
\item If humanity were to establish a colony on another planet, what do you think are the most important considerations to ensure its success?
\end{enumerate}

\new{
\subsection{Discussion on the Selection of 
\texorpdfstring{$\Delta$}{Delta}}
\label{appen:Delta-selection}

In practice, we choose $\Delta$ using prior knowledge. It is important to note that there is a trade-off when selecting the value of $\Delta$.

Assume that an exact $\gamma$-fraction of $\bP_1, \ldots, \bP_n$ belongs to $\FPM$, and we lack information about the remaining NTP distributions. Consequently, the rest are considered to be in the full class $\PM_0 = \Simplex(\Voca)$, which includes all categorical distributions over the vocabulary $\Voca$. Proposition \ref{thm:inhomo-type-II-general} indicates that $\lim\sup\limits_{n \to \infty}\PB_{H_1}(T_h(Y_{1
}) = 0)^{1/n} \le \mathrm{e}^{- \gamma R_{\FPM}(h)}$. This inequality is tight in the worst-case scenario, implying that without additional knowledge, this is the best convergence rate we can expect.

Theoretically, we view $\gamma$ as a function of $\Delta$, interpreting it as the fraction of $\Delta$-regular NTP distributions among $\bP_1, \ldots, \bP_n$. Clearly, $\gamma(\Delta)$ is a decreasing function of $\Delta$. As $\Delta$ increases, we impose a stricter condition on the largest probability in $\bP$, requiring $\max_{\token \in \Voca} P_{\token} \le 1 - \Delta$, resulting in a smaller fraction satisfying this condition. However, as shown in Figure \ref{fig:inhomo-efficiency} (or the middle panel of Figure \ref{fig:empirical-gamma}), $R_{\FPM}(h)$ is an increasing function of $\Delta$.
Therefore, the final efficiency rate, $\gamma(\Delta)R_{\FPM}(h)$, is a complex function of $\Delta$, neither strictly increasing nor decreasing. This implies a trade-off in choosing $\Delta$, and an optimal choice exists.

\begin{figure}[t!]
\centering
\includegraphics[width=\textwidth]{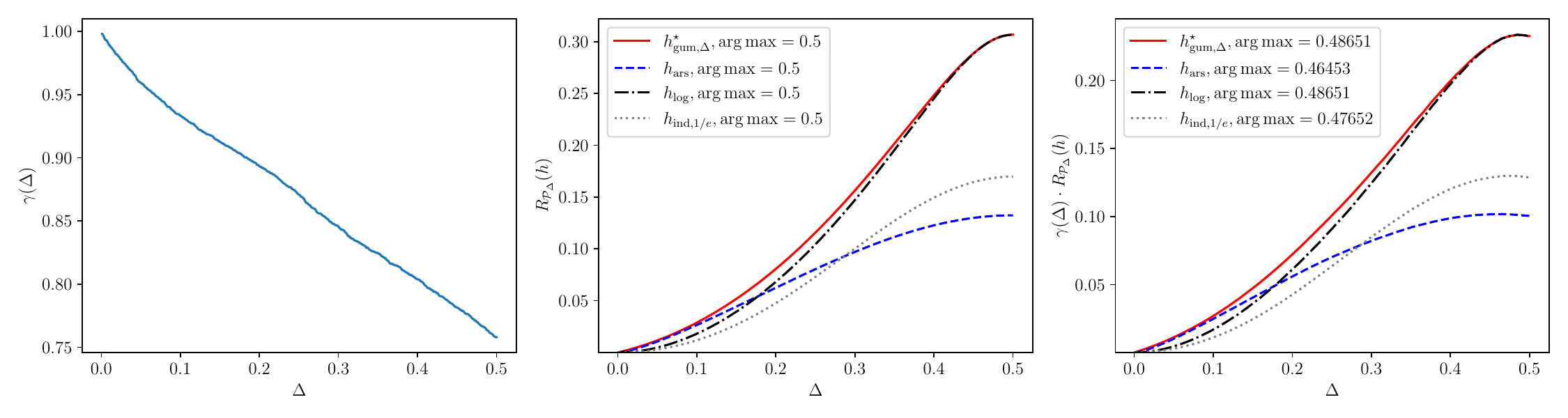}
\caption{
Illustration of $\gamma(\Delta)$ (left), $R_{\FPM}(h)$ (middle), and $\gamma(\Delta) \cdot R_{\FPM}(h)$ (right) as functions of $\Delta$. Since $\Delta$ is typically small in practice, we focus on the range $\Delta \in [0.001, 0.5]$. The experimental setup is consistent with that of Figure \ref{fig:empirical-delta}.
}
\label{fig:empirical-gamma}
\end{figure}

We support the above analysis with further empirical study, following the same setup as in Figure \ref{fig:empirical-delta} and presenting the results in Figure \ref{fig:empirical-gamma}.
Its left panel shows the empirical estimation of $\gamma(\Delta)$, which is clearly a decreasing function of $\Delta$. 
The middle panel, essentially a zoomed-in version of Figure \ref{fig:inhomo-efficiency}, is presented here within the interval $\Delta \in [0.001, 0.5]$ for reader convenience.
The $\arg\max$ denotes the value of $\Delta$ at which the value of the $y$-axis reaches its maximum. Since theoretically $R_{\FPM}(h)$ is an increasing function of $\Delta$, we find that $\arg\max \equiv 0.5$ for all score functions.
However, when we focus on the product $\gamma(\Delta) \cdot R_{\FPM}(h)$, it typically increases at first and then decreases. This can be validated by observing the value of $\arg\max$. In the right panel of Figure \ref{fig:empirical-gamma}, the computed values of $\arg\max$ for all score functions shift from $0.5$, indicating that the inclusion of $\gamma(\Delta)$ indeed impacts the final efficiency. 
We highlight that in our particular example, the effect of $\gamma(\Delta)$ seems limited, as the shift in $\arg\max$ is moderate. However, there are cases where this deviation can be significant.
Hence, choosing a good value of $\Delta$ would be a case-by-case argument, but the idea of choosing an optimal $\Delta$ would always be applicable.

}
\end{appendix}

\end{document}